\documentclass{article} % replace ECP by EJP if needed.
\usepackage{mathrsfs,dsfont}
\usepackage{amsmath,amsfonts,graphicx,amsthm,amssymb,color}  
\usepackage{geometry}[a4paper,portrait,left=3.5cm,right=3.5cm,top=3.5cm,bottom=3.5cm]

\title{Hidden regular variation for stochastic recursions with diagonal matrices\ 
    \thanks{E.D. was supported by NCN grant 2019/33/B/ST1/00207. S.M. was supported by DFG grant ME 4473/2-1.}} % \thanks is optional. Insert line breaks with \\

\author{%
  Ewa~Damek\footnote{University of Wroc\l aw, Poland.
    {ewa.damek@math.uni.wroc.pl}}%\orcid{0000-0003-0986-3622}
  \and %% remove this line and below if single author
  Sebastian~Mentemeier\footnote{University of Hildesheim,
    Germany. {mentemeier@uni-hildesheim.de} }}%AUTHORS

\DeclareMathOperator{\Id}{\mathrm{I}}

\DeclareMathOperator{\N}{\mathds{N}}
\DeclareMathOperator{\E}{\mathds{E}}
\DeclareMathOperator{\Cf}{\mathcal{C}}
\newcommand{\supp}{\mathrm{supp}}

\newcommand{\interior}[1]{\mathrm{int}(#1)}

\def\R{{\mathbb{R}}}
\newcommand{\s}{\sigma}
\renewcommand{\a}{\alpha }
\newcommand{\eps}{\varepsilon}

\renewcommand{\P}{\mathds{P}}

\newcommand{\m}[1]{\bb{#1}}

\newcommand{\imag}{\mathrm{i}}

\renewcommand{\t}{\top}

\newcommand{\norm}[1]{\|#1 \|}
\newcommand{\abs}[1]{|#1 |}

\newcommand{\skalar}[1]{\langle #1 \rangle}
\newcommand{\8}{\infty}

\newcommand{\eqdist}{\stackrel{\mathrm{law}}{=}}

\newtheorem{theorem}{Theorem}[section]

\newtheorem{lemma}[theorem]{Lemma}

\newtheorem{proposition}[theorem]{Proposition}

\theoremstyle{definition}

\newtheorem{remark}[theorem]{Remark}

\newcommand{\norma}[1]{\norm{#1}_{\bb \alpha}}

\newcommand{\bb}{\boldsymbol}

\newcommand{\dri}{\mathrm{dRi}}

\renewcommand{\d}{\delta}

\def\bfA{{\bb{A}}}
\def\bfB{{\bb{B}}}
\def\bfC{{\bb{C}}}

\def\bfX{{\bb{X}}}

\def\Z{{\mathbb{Z}}}
\newcommand{\wt}{\widetilde}
\renewcommand{\phi}{\varphi}

\newcommand{\Ind}{\mathbf{1}_}

\renewcommand{\L}{\Lambda}

\def\bfm{{\bb{m}}}

\numberwithin{equation}{section}

\begin{document}

\maketitle

\begin{abstract}
 We consider random vectors $X$ that satisfy the equation in law $X=AX+B$, where $A$ is a given random diagonal matrix and $B$ a given random vector, both independent of $X$. It is well known by the works of Kesten and Goldie that the marginals of $X$ may exhibit heavy tails, with possibly different tail indices.
 In recent works (Damek 2025, Mentemeier and Wintenberger 2022) it was observed that asymptotic independence may occur despite strong dependencies in the entries of $A$: The probability that both marginals are simultaneously large decays faster than the marginal probability of an extreme event; the tail measure is concentrated on the axis. In this work, we analyse the hidden regular variation properties of $X$, that is, we find the proper scaling for which one observes simultaneous extremes.
\end{abstract}


%\AMSSUBJSECONDARY{FIXME:} % Optional, separate items with ;

%\subsection*{Notation to be replaced}
%
%\begin{itemize}
%%	\item $\d_t$ becomes $t^{1/\bb \a} \bb x$
%%	\item $\bb U$ is already normalized with $\alpha_1, \alpha_2$ 
%	\item %$\phi$ below corresponds to Ewas $\widetilde{\phi}$: 
%	$\phi(\bb \xi)=\E \big[ |A_1|^{\alpha_1 \xi_1} |A_2|^{\alpha_2 \xi_2} \big]$; $\Phi(\theta)=\E \big[ |A_1|^{\theta_1} |A_2|^{\theta_2} \big]$; hence $\phi(\bb \xi)=\Phi(\bb \xi \bb \alpha)$.
%	\item General notation: small letters for functions of $\bb \xi$, where $\alpha_1, \alpha_2$ are normalized to 1; capital letters for functions of $\bb \theta$ as above
%%	\item blocks are $\{1,3,\dots, d_1+1\}$ and $\{2, d_1+2, \dots, d_2\}$ because this allows to unify the roles of $A_1$ and $A_2$.
%%	\item $\E [ \cdot]$ - always use $[]$ brackets
%%	\item unifiy $(a,b)$ or $(a,b)^\top$, respectively. Maybe transpose only needed for $B$ / in the introduction, then drop??
%	\item Use $\eps$ in proofs; $\epsilon$ for Hölder: $H^\epsilon$. 
%	%\item unify $\bb X$ and $\bfX$ (change the commands)
%\end{itemize}
%
%\subsection*{ToDo}
%
%\begin{itemize}
%	\item Add Steps in long proofs
%\end{itemize}
%
%

\section{Introduction}
On a probability space $(\Omega, \mathcal{A}, \P)$
let $\m{A}=\mathrm{diag}(A_1,A_2)$ be a $2 \times 2$ random diagonal matrix with real-valued random variables $A_1, A_2$ on the diagonal.  Let $\bb{B}=(B_1, B_2)^\top$ be a random vector in $\R^2$. Note that $(A_1,A_2,\bb{B})$ may have any dependence structure; we are particular interested in settings with strong dependence between $A_1$ and $A_2$.
Given a sequence $(\m{A}_n,\bb{B}_n)_{n \ge 1}$ of independent and identically distributed (i.i.d.) copies of $(\m{A},\bb{B})$, we define a stochastic recursion by
\begin{equation}\label{eq:recc}
 \bb{X}_{n}=\m{A}_n \bb{X}_{n-1} + \bb{B}_n, \qquad n \ge 1,\end{equation}
where $\bb{X}_0$ is an arbitrary initial value (deterministic or random), which we assume to be independent of $(\m{A}_n, \m{B}_n)_{n \ge 1}$. This model includes several classes of autoregressive time series processes, such as diagonal BEKK-ARCH or some CCC-GARCH models, see \cite{Mentemeier2022} and Section \ref{sect:Examples} below for details.

Under weak assumptions, there is a unique stationary distribution for the Markov chain $\bb{X}_n$. If $\bb{X}=(X_1,X_2)^\t$ is a vector with  this  stationary distribution that is independent of $\m{A}$ and $\bb{B}$, then $\m{X}$ satisifies the stochastic fixed point equation $\m{A} \bb{X} + \bb{B}\eqdist \bb{X}$, where $\eqdist$ denotes equality in law. Under suitable assumptions, it holds that $\lim_{t \to \infty} t^{\alpha_i}\P(X_i >t)=c_i>0$ for $i=1,2$; where $\alpha_i$ are positive constants given as the unique positive solution to the equation $\E [ A_i^{\alpha_i}]=1$; see \cite{Goldie1991,Kesten1973}.
Equivalently,
\begin{equation}\label{eq.intro.marginal.tails}
	\lim_{t \to \infty} t\P(\abs{X_i} >t^{1/\alpha_i})=c_i>0.
\end{equation}

In this work, we study the multivariate regular variation properties of $\bb X$.
It was shown in \cite{Damek2021} that -- unless $\abs{A_1}=\abs{A_2}^c$ a.s.\ for some constant $c>0$ -- {\em asymptotic independence} (cf. e.g. \cite[Section 5.6]{Resnick2024}) between the components $X_1$ and $X_2$ occurs, that is  
\begin{equation}\label{eq.intro.asymptotic.independece}
	\lim_{t \to \infty} t^{1+\delta} \P(\abs{X_1} >t^{1/\alpha_1}, \abs{X_2} > t^{1/\alpha_2})=0
\end{equation}
for some $\delta >0$. Note that the definition of asymptotic independence (\cite[Section 5.6]{Resnick2024}) is stated with $\delta=0$. Hence, a stronger property was proved in \cite{Damek2021}. 

In particular, upon conditioning on $\P(X_1>t^{1/\alpha})$, say, it follows by a combination of \eqref{eq.intro.marginal.tails} and \eqref{eq.intro.asymptotic.independece} that, conditioned on the first component $X_1$ being large, the probability that the second component is also large tends to zero; even with a rate $t^{\delta}$. The value of $\delta $, however, is not optimal in \cite{Damek2021}. In this paper we obtain the precise asymptotics in \eqref{eq.intro.asymptotic.independece}. Hence, for example in the case of heavy tailed financial assets described by \eqref{eq:recc}, our results give an exact probability of both of them being large.

It is useful to express asymptotic independence  \eqref{eq.intro.asymptotic.independece}  as a multivariate regular variation property of $\bb X$ (cf.  \cite[Theorem 2.17]{Damek2021},\cite[Proposition 5.1]{Resnick2024}). Write $\bfC(\R^2\setminus F)$ for the set of bounded continuous functions $f:\R^2 \setminus F \to \R$ the support of which is bounded away from the {\em forbidden zone} $F$. It holds that there is a {\em limit measure} $\Lambda$ such that for any $f \in \bfC(\R^2 \setminus\{\bf 0\})$
$$ \lim_{t \to \infty} t \E \big[ f\big(t^{-1/\alpha_1}X_1, t^{-1/\alpha_2}X_2\big) \big] = \int_{\R^2 \setminus \{\bb 0\}} f(x,y) \Lambda(dx,dy).$$
and the support of $\Lambda$ is restricted to $[\mathsf{axes}]:=\big(\R \times \{0\}\big) \cup \big(\{0\} \times \R \big)$. It is hence an open question, what are the asymptotics for functions $f \in \bfC\big(\R^2 \setminus [\mathsf{axes}]\big) $. Following \cite[Section 3.1]{Resnick2024}, (the distribution of) a random vector $\bb X$ exhibits hidden regular variation if there is another limit measure $\Lambda_2$ such that when enlarging the forbidden zone (here: from $\{\bb 0\}$ to $[\mathsf{axes}]$), multivariate regular variation holds with a different and slower rate. This is exactly what we are going to prove for our model:
Assuming in addition some strong nonlattice condition (satisfied e.g. if the law of $\bb A$ is absolutely continuous) and some conditions on the mixed moments of $A_i, B_i$, we will prove (see Theorem \ref{thm:main}) that there is $\eta>0$ and a limit measure $\Lambda_2$ such that for all $f \in \bfC(\R^2 \setminus [\mathsf{axes}])$
\begin{equation}\label{eq:result:intro}
	 \lim_{t \to \infty} (\log t)^{1/2}\, t^{1+\eta} \E \big[ f\big(t^{-1/\alpha_1}X_1, t^{-1/\alpha_2}X_2\big) \big] = \int_{\R^2 \setminus [\mathsf{axes}]} f(x,y) \Lambda_2(dx,dy).
\end{equation}
The value of  $\eta$ is determined as follows. Considering 
$$ \phi(\xi_1, \xi_2) :=\E \Big[ |A_1|^{\xi_1\alpha_1} |A_2|^{\xi_2 \alpha_2} \Big],$$
we require the existence of $\bb \xi^* = (\xi_1^*, \xi_2^*)$ such that $\bb \xi^* \in (0,1)^2$, $\phi(\bb \xi^*)=1$ and such that the gradient $\nabla \phi(\bb \xi^*)$ is parallel to $(1,1)$. Then such $\bb \xi^*$ is unique, and $1+\eta:=\xi_1^*+\xi_2^*$.

The hidden regular variation property \eqref{eq:result:intro} extends to stationary solutions of Eq. \eqref{eq:recc} with matrices and vectors of arbitrary dimension provided the diagonal matrix $\bfA$ has a special structure. This is contained in the main theorem \ref{thm:main:blocks}. We think however that describing first the two dimensional case is beneficial for the reader.
%This is embedded in a general result showing that for any choice of exponents $\beta_1>0, \beta_2>0$ from a suitable range there is a corresponding $\eta(\beta) \in (1,2)$ 
%$$\lim_{t \to \infty} t^{1+\eta(\beta)} \P(X_1 >t^{1/\beta_1}, X_2 > t^{1/\beta_2}) \in (0,\infty).$$

\subsubsection*{Related Works}

A recent textbook dedicated to the topic of hidden regular variation is \cite{Resnick2024}, see also \cite{Basrak2025,Das2013,Resnick2003} and references therein. Analysing hidden regular variation in various models is an active branch of current research, e.g. in the fields of L\'evy processes \cite{Lindskog2014}, stable random vectors \cite{Forsstrom2020}, point processes \cite{Dombry2022} and branching processes \cite{Blanchet2025}.

Asymptotic independence for stochastic recursions with diagonal matrices was considered first in \cite{Mentemeier2022} for particular cases of $\bb A$, and in a general setting in \cite{Damek2021}. For general information about multivariate stochastic recursions we refer to Chapter 4 of the book \cite{Buraczewski2016} as well as to the fundamental work \cite{Kesten1973}  and its extensions e.g. in \cite{Alsmeyer2012,Buraczewski2009a, Guivarch2016,Kluppelberg2004}. 
In the cited works the structure of the matrix $\bb A$ is always such that it acts in an irreducible way implying multivariate regular variation with a limit measure that charges the full state space; hence hidden regular variation cannot occur.

\subsection{Structure of the Paper}

Having given an adhoc version of our main result in the introduction, we proceed by stating the precise formulation of our assumptions in Section \ref{sect:assumptions.notations}. There, we also collect for the reader's convenience notation and symbols used throughout the work. In particular, notation needed to formulate the multidimensional theorem is introduced in Subsection 2.2.3.
We state our main results in Section \ref{sect:main result}. It is possible to consult only  Subsection \ref{subsect:assumptions} containing the assumptions and then proceed directly to the statement of the main result, the essential notions have been discussed in the introduction. In Section \ref{sect:Examples}, we apply our results to analyze hidden regular variation properties of two nonlinear multivariate time series models, namely CCC-GARCH and BEKK-ARCH. For further illustration, we consider the case where $A_1, A_2$ have a log-normal distribution. Section \ref{sect.phi} is devoted to properties of the function $\phi$. In Section \ref{sect:moments}, we provide estimates for mixed moments $\E \big[ |X_1|^{\xi_1 \alpha_1} |X_2|^{\xi_2 \alpha_2}\big]$ that are used later on. The convergence to a limit measure $\Lambda_2$ is proved in Section \ref{sect:implicit.renewal}, leaving open the possibility that $\Lambda_2$ could be degenerate (with total mass 0). This is resolved in Section \ref{sect:positivity} where we prove that $\Lambda_2$ is nontrivial. 

Our main tools are two-dimensional Edgeworth expansions  and a  renewal theorem on $\R^2 \times K$, where $K$ is a finite group. We provide such a renewal theorem in Section \ref{sect:2dRenewal.with.K}. In the appendix, we provide for the reader's convenience some frequently used results for two-dimensional random walks, adopted to the present notation.

\section{Assumptions, Notation and Preliminaries} \label{sect:assumptions.notations}

\subsection{Assumptions}\label{subsect:assumptions}

Let $d \ge 1$ be an integer. In larger parts of the paper, we will study the two-dimensional case $d=2$, but the general result will be stated for a setup where $d >2$ may occur.
As already stated in the Introduction, we assume $A_i$ and $B_i$, $1 \le i \le d$, to be real-valued random variables, satisfying the following additional assumptions.
\begin{equation}\label{assump:alpha} \tag{A1}%\tag{exist:alpha}
	\text{For $1 \le i \le d$ there exists } \alpha_i>0 \text{ such that } \E[ \abs{A_i}^{\alpha_i}] =1; \text{ and } \P(\abs{A_i}=1)<1,\, \P(\abs{A_i}=0)=0.\ \end{equation}

Note that the mapping $s \mapsto \E [\abs{A_i}^s]$ is convex, which together with assumption \eqref{assump:alpha} readily implies that  for $1 \le i \le d$, 
\begin{equation}\label{eq:logAi}
	\E[ \log \abs{A_i}] <0.
\end{equation}
Further, with the values of $\alpha_i$, $1 \le i \le d$ given by \eqref{assump:alpha}, assume that
\begin{equation}\label{assump:moments} \tag{A2}%\tag{moments} 
		\mbox{for}\ 1 \le i \le d\ \mbox{it holds } \ 0< \E [\abs{B_i}^{\alpha_i}] < \infty,\ \E[ \abs{A_i}^{\alpha_i}\log ^+|A_i|] < \infty \end{equation} 
Under assumptions \eqref{assump:alpha} and \eqref{assump:moments} it holds that 
the marginal equations $A_i X_i + B_i \eqdist X_i$ have a unique solution (in distribution), that is given by the law of the then almost surely convergent series
\begin{equation}\label{eq:perpetuity}
	X_i =\sum_{k=0}^\infty A_{1,i} \cdots A_{k-1,i} B_{k,i},
\end{equation}
where $(A_{k,i}, B_{k,i})_k$ are sequences of i.i.d. copies of $(A_i, B_i)$ for $1 \le i \le d$, that are also independent of $(A_i,B_i)$.
Imposing assumptions \eqref{assump:alpha} and \eqref{assump:moments}, it also follows from the convexity of the mapping $s \mapsto \E [\abs{A_i}^s]$ that
\begin{equation}\label{eq:AilogAi}
	\E \big[\abs{A_i}^{\alpha_i} \log \abs{A_i}\big]>0.
\end{equation}
Observe that \eqref{assump:alpha} allows to introduce for $1 \le i \le d$ an exponential change of measure (Esscher transform) for $\log |A_i|$ (with parameter $\alpha_i$). Eq. \eqref{eq:AilogAi} then states that $\log |A_i|$ has positive drift under the exponential change of measure with parameter $\alpha_i$.

To avoid degeneracy of $X_i$, we require 
\begin{equation} \label{assump:nondeg} \tag{A3}%\tag{nontrivial} 
	\text{For $1 \le i \le d$, } \P(A_ix+ B_i= x)<1 \quad \text{ for all } x \in \R.\end{equation}

%As said before, we will be mainly concerned with 
 Due to a special structure of $\bfA$ (explained in subsection 2.2.3), as far as the entries of $\bfA$ are considered, we may restrict our attention to the bivariate case, for which we introduce some further conditions now.
For a probability measure $Q$ on $\R^2$, denote its characteristic function by $\hat{Q}(t_1, t_2)=\int e^{i (t_1 x_1 + t_2x_2)} \, Q(d(x_1, x_2))$. We say that $Q$ is {\em nonarithmetic}, if $\{ (t_1, t_2) \, : \, \hat{Q}(t_1, t_2)=1\}=\{(0,0)\}$.
We will require that
\begin{equation}\label{assump:nonarithmetic}\tag{A4a} %\tag{nonarithmetic}
	\text{The law of $(\log |A_1|, \log |A_2|)$ is nonarithmetic.}
\end{equation}
 In particular, its support is not contained in a line.
%For $\alpha_1, \alpha_2$ given by \eqref{assump:alpha}, define 
%$$ (U_1,U_2):=(\alpha_1 \log |A_1|, \alpha_2 \log |A_2|)$$
%
%
%This condition implies in particular, that the law of $(U_1,U_2)$ is {\em strongly nonlattice} in the sense that there does not exist $(\xi_1, \xi_2) \neq (0,0)$ such that $|\hat{\mu}(\xi_1,\xi_2)| =1$. 

Under assumptions \eqref{assump:alpha}, \eqref{assump:moments}, \eqref{assump:nondeg} and\eqref{assump:nonarithmetic} the Kesten-Goldie theorem \cite{Goldie1991,Kesten1973} gives that there are $c_1,c_2>0$ such that for $i \in \{1,2\}$,
\begin{equation}\label{eq:tailsMarginal}
	\lim_{t \to \infty} t \P(X_i > t^{1/\alpha_i}) = c_i.
\end{equation}
	Note that  \eqref{assump:nonarithmetic} implies the following property:
	\begin{equation}\label{assump:notapower} \tag{A4b}%{A1notA2togamma}
		\P(\abs{A_1}=\abs{A_2}^c)<1 \text{ for all } c>0\end{equation}
	which was recognized in \cite{Damek2021} to be necessary and sufficient for asymptotic independence between $X_1$ and $X_2$ (given that \eqref{assump:alpha}, \eqref{assump:moments} and \eqref{assump:nondeg} are satisfied). Observe that \eqref{assump:alpha} forbids $c\leq 0$.
%	To wit, if \eqref{assump:notapower} would fail, then for all $t \in \R$,
%	$$ \E \exp(\imag (t \log |A_1| - ct \log |A_2| )) = \E \exp(\imag (t \log |A_1| - t \log |A_1| ))=1.$$
%	Hence, if \eqref{assump:notapower} fails, then also \eqref{assump:nonarithmetic} is violated.

We also need an assumption on mixed moments: With $\alpha_1, \alpha_2$ given by \eqref{assump:moments}, we require
\begin{equation}\label{assump:mixedmoments}\tag{A5}%{mixed:moments}
	\E \big[ |B_1|^{\alpha_1}|B_2|^{\alpha_2}\big]+\E \big[ |B_1|^{\alpha _1}\ |A_2|^{\alpha_2}\big] + 
	\E\big[ |A_1|^{\alpha_1} |B_2|^{\alpha _2}\big]  + \E \big[ |A_1|^{\alpha_1} |A_2|^{\alpha_2}\big]< \infty
\end{equation}

Finally, we will in some places assume a condition that is complementary to \eqref{eq:AilogAi} in the sense that $\log |A_2|$ has positive drift as well under the exponential change of measure related to $\log |A_1|$, and vice versa:
\begin{equation}\label{assump:positivedependence} \tag{A6}%{posdep}
\E [\abs{A_1}^{\alpha_1} \log \abs{A_2} ]>0, \quad \E [\abs{A_2}^{\alpha_2} \log \abs{A_1}] >0 \end{equation} 

\subsection{Notation}
\subsubsection{Sets, Norm-like Functions, Function Spaces}

For vectors and matrices, we use bold notation ${\bb x}=(x_1, \dots, x_d)$. Whether $\bb x$ is a column or row vector will be clear from the context. The Euclidean scalar product between vectors $\bb x, \bb y$ is denoted by $\skalar{\bb x, \bb y}=\sum_{i=1}^d x_i y_i$. Without further subscript, $\norm{\bb x}=\sqrt{\skalar{\bb x, \bb x}}$ denotes the Euclidean norm. The nonnegative orthant in $\R^d$ is denoted by
$$ \R^d_+:=\{ \bb x \in \R^d \, : \, x_i \ge 0 \text{ for all } 1 \le i \le d\}$$

For any $x \in \R$, we write $x^+:=\max\{0,x\}$.

Operations between vector and scalar or vector and vector are interpreted coordinatewise, in particular, we write (for $t>0$)
\begin{equation}\label{dilat}
	% t^{1/\bb \alpha} \bb x=\d _t(\bb x)= (t^{1\slash \a _1}x_1,...,t^{1\slash \a _d}x_d ), 
	\bb t^{1/\bb \a}\bb x =(t_1^{1\slash \a _1}x_1,...,t_d^{1\slash \a _d}x_d ).
\end{equation}
Also, $\bb x \le \bb y$ is to interpreted as $x_i \le y_i$ for all $1 \le i \le d$.
For the vector $\bb \alpha=(\alpha_1, \dots, \alpha_d)$ with positive entries given by \eqref{assump:alpha}, we define the norm-like function $\norma{\cdot}$ by
$$ \norma{\bb x}:=\max_{1 \le i \le d} |x_i|^{\alpha_i}$$
both for $d$-vectors $\bb x$ as well as $d \times d$-diagonal matrices. Note that $\norma{\cdot}$ is in general not a norm: it is neither homogeneous nor subadditive, but there is $c_{\bb \alpha} \in (0,\infty)$ such that
\begin{equation}\label{eq:norm:subadditiv}
	 \norma{\bb x + \bb y} \le c_{\bb \alpha}\big( \norma{\bb x} + \norma{\bb y}\big).
\end{equation}Observe that for any $t\ge 0$,we have the ($\bb \alpha$-)homogenity property
%\begin{equation*}
%\d _t(x)=(t^{1\slash \a _1}x_1,..., t^{1\slash \a _d}x_d ).\end{equation*} Then
\begin{equation}\label{eq:homogenity}
\norma{t^{1/\bb \alpha} \bb x} = t\norma{\bb x}.
\end{equation}
We will consider spheres with respect to the norm-like function,
\begin{align*}
	S_r^{d-1}=\{ \bb x\in \R ^d: \norma{\bb x}=r\}, &\quad S^{d-1}:=S_1^{d-1},
%	\\
%	B_r(\bb y)=\{ \bb x\in \R ^d: \norma{\bb x-\bb y}<r\}, &\quad B_r(\bb y)^c=\{ \bb x\in \R ^d: \norma{\bb x-\bb y}\geq r\},
\end{align*}
and 
introduce polar coordinates related to  $\norma{\cdot}$ using the map $\Theta : \R ^+\times S^{d-1}\to \R ^d\setminus \{ 0\}$,
\begin{equation}\label{polar} 
	\Theta (s,\omega )=s^{1/\bb \alpha} \omega \quad s>0,\ \omega \in S^{d-1} 
\end{equation} 
which is a homeomorphism (see e.g. \cite[(2.15)]{Damek2021} for a proof).

For a Borel subset $E \in \R^d$, $\mathcal{C}(E)$ denotes the space of bounded continuous functions on $E$, and $\mathcal{C}_c(E)$ the space of continuous functions on $E$ with compact support, both equipped with the norm
$\| f\| _{\8} = \sup _{ x\in E}|f(x)|$. For a function $f$ on $E$, its support is defined as the closure of $\{\bb x \in E \, : \, f(\bb x) \neq 0\}$. For a closed subset $F \subset E$, we denote by $\bfC (E\setminus F)$ the space of {\it bounded} continuous functions on $E$ whose support is bounded  away from the set $F$. 
Denote by $\bb e_i$, $1 \le i \le d$, the standard basis of $\R^d$.

The following notation will be used only for $d=2$.
	We write
	$[\mathsf{axes}]=(\R \bb e_1) \cup (\R \bb e_2),$
	for the set of (coordinate) axes. The 1-spheres without $[\mathsf{axes}]$  are denoted by
	$$ \widetilde{S}_r:=S_r^{1} \cap (\R^2 \setminus [\mathsf{axes}]) =S_r^{1} \setminus \{r \bb e_1,-r \bb e_1,r \bb e_2,-r \bb e_2\}.$$
	We will consider in particular
\begin{equation*}
	\bfC(\R^2 \setminus [\mathsf{axes}]):= \big\{ f \in \mathcal{C}(\R^2 \setminus [\mathsf{axes}]) \, : \, \supp f\subset \{\bb x \, :\, |x_i|>\eta_i \text{ for }  1 \le i \le 2\} \text{ for some } \eta_1, \eta_2>0 \big\}.
\end{equation*}
%as well as \begin{equation*}
%	\bfC(\R^d \setminus \{0\}):= \big\{ f \in \mathcal{C}(\R^d \setminus \{0\}) \, : \, \supp f\subset \R ^d\setminus B_r(0) \text{ for some } r>0 \big\}
%\end{equation*}
Let $0<\epsilon <1$ be such that $\max_i \epsilon \alpha_i <1$ and $U$ be a Borel subset of $\R^d$. We  say that a function $f: U \to \R$ is $\epsilon$-$\bb \alpha$-H\"older if there is a constant $C_f$ such that
	\begin{equation*}
		|f(\bb x)-f(\bb y)|\leq C_f\norma{\bb x-\bb y}^{\epsilon}
	\end{equation*}
	for every $\bb x,\bb y\in U $. Note that a bounded function $f$ is $\epsilon$-$\bb \alpha$-H\"older if it is H\"older continuous in the classical sense with respect to a norm, with a possibly different exponent. 
	Let
%	$$ \bb H^\epsilon:= \Big\{ f \in \mathcal{C}(\R^d) \, : \, f \text{ is } \epsilon\text{-H\"older and }\supp f\subset \{\bb x: |x_1|^{\alpha_1} \geq \eta _1, |x_2|^{\alpha_2}\geq \eta _2\} \text{ for some } \eta_1, \eta_2 >0 \Big\}$$
	$$ \bb H^\epsilon(\R^2 \setminus [\mathsf{axes}]):= \Big\{ f \in \bfC(\R^2 \setminus [\mathsf{axes}]) \, : \, f \text{ is } \epsilon\text{-}\bb \alpha\text{-H\"older } \Big\}$$
 As usual, $L^1(\R^2)$ denotes the space of measurable functions $f: \R^2\to \R$ that are integrable with respect to the Lebesgue measure $\lambda^2$ on $\R^2$. The interior of a Borel set $U$ is denoted by $\mathrm{int}(U)$.

\subsubsection{Random Variables, Moments and a Change of Measure}

With an slight abuse of notation, we will also speak of the support of a measure $\mu$ or a random variable $Z$, defined as follows. For a measure $\mu$ on $\R^d$, $$\supp(\mu):= \{ \bb x \, :\, \mu(U)>0 \text{ for all open neighborhoods } U \text{ of } \bb x \},$$
and $\supp(Z)$ is defined as the support of the law of $Z$.

Under assumptions \eqref{assump:alpha} and \eqref{assump:moments}, the series
% \begin{equation}\label{eq:def:X}
 $	\bfX := \sum_{k=1}^\infty \bfA_1 \cdots \bfA_{k-1} \bfB_k$
% \end{equation}
converges almost surely, and it holds that $\bfX \eqdist \bfA \bfX + \bfB$, stipulating that $\bfX$ and $(\bfA, \bfB)$ are independent.
We may split the series for $\bb X$ at any $n \in \N$ and denote
$$\bfX_{>n} :=\sum _{k=1}^{\8 }\bfA_{n+1}\cdots \bfA _{n+k-1}\bfB _{n+k} \quad \text{and} \quad \bb X_{\le n} :=\sum _{k=1}^{n }\bfA_1\cdots \bfA _{k-1}\bfB _k.$$
It then holds that $\bb X_{>n} \eqdist \bb X$ and $\bb X = \bb X_{\le n} + \bfA_1\cdots \bfA_n \bb X_{>n}$, and $\bb X_{\le n}$ is independent of $\bb X_{>n}$.

With the values of $\alpha_1$, $\alpha_2$ determined by \eqref{assump:alpha}, we introduce the random vector $\bb U$ and the random $2\times 2$-diagonal matrix $\bb K$ by
$$ \bb U:=(U_1,U_2)=(\alpha_1 \log |A_1|, \alpha_2 \log |A_2|), \qquad  \bb K = \mathrm{diag}(\mathrm{sign}(A_1), \mathrm{sign}(A_2)).$$
We may hence decompose 
%\begin{equation}\label{eq:K:2dim}
	 $\bb A =\mathrm{diag}(A_1, A_2) = \mathrm{diag} (e^{U_1/\alpha_1}, e^{U_2/\alpha_2}) \bb K =: e^{\bb U / \bb \alpha} \bb K$.
%\end{equation}
We write $K$ for the smallest group generated by the support of $\bb K$, which can easily identified with a subgroup of $\Z_2 \times \Z_2$.
We denote by $\wt{\mu}$ the joint law of $(\bb U, \bb K)$ and by $\mu$ the law of $\bb U$.
Let $(\bb U_i, \bb K_i)_{i \ge 1}$ be a sequence of i.i.d.\ copies of $(\bb U, \bb K)$ and set
$ \bb S_n:= \sum_{i=1}^n \bb U_i$, $\bb L_n:= \prod_{i=1}^n \bb K_i$.
Note that $\bb S_n$ has negative drift (in both components) due to \eqref{eq:logAi}.
The renewal measure of $\bb S_n$ and $(\bb S_n, \bb K_n)$ are given, respectively,  by
$$ \mathds{U}(C):=  \sum_{n=0}^\infty \P(\bb S_n \in C) = \sum_{n=0}^\infty \mu^{*n}(C), \quad  \wt{\mathds{U}}(C \times D):=  \sum_{n=0}^\infty \P(\bb S_n \in C, \bb L_n \in D) = \sum_{n=0}^\infty \wt \mu^{*n}(C\times D) $$
for any Borel subset $C$ of $\R^2$ and $D \subset K$. 
Here and below, $*$ denotes convolution of measures.
A central concept in renewal theory are directly Riemann integrable functions, which are defined in our context as follows.
	Denote $F_{\bb 0}:=[0,1]^2$ and for $\bfm =(m_1,m_2)\in \Z ^2$, let $F_{\bfm}:=\bb m + F_{\bb 0}$, hence $F_{\bb m}=\{ \bb s\in \R^2: m _j\leq s_j\leq m_{j}+1, j=1,2\}$.
	A continuous function $g:\R^2 \times K$ is said to be directly Riemann integrable, if $\norm{g}_{\dri}<\infty$, where
$$\| g\| _{\dri}= \sum _{k\in K}\sum _{\bfm}\sup _{\mathbf{s}\in F_{\bfm} }(1+\| \mathbf{s}\| )|g (\mathbf{s},k)|.$$
Consider the moment generating function $$\phi(\bb \xi):= \E[ \exp(\skalar{\bb \xi, \bb U}) ]= \E[ |A_1|^{\xi_1 \alpha_1} |A_2|^{\xi_2 \alpha_2}]$$ with domain
$ I:= \{  \bb \xi \in \R^2 \, : \, \E[ \exp \big( \skalar{\bb \xi, \bb U}] \big) < \infty\}. $ Note that $\phi$ is convex on $I$ and analytic on  $\interior{I}$.
 We further denote 
$$\Delta :=\{ \bb \xi \in I \cap \R_+^2\, : \, \phi(\bb \xi)\le1\}, \qquad D :=\{ \bb \xi \in I \cap \R_+^2  \, : \, \phi(\bb \xi) =1\},$$ 
 where $\R_+=[0,\8 )$. 

For each $\bb \xi \in I,$ we can define a finite measure $\mu_{\bb \xi}$ by setting $\mu_{\bb \xi}(d\bb u):= e^{\skalar{\bb \xi ,\bb u}} \mu(d\bb u)$ as well as $\wt{\mu}_{\bb \xi}(d\bb u, k):= e^{\skalar{\bb \xi ,\bb u}} \wt{\mu}(d\bb u,k)$. Note that $\mu_{\bb \xi}$ is a subprobability measure whenever $\phi(\bb \xi)<1$. If $\bb \xi \in D$, then $\mu_{\bb \xi}$ is a probability measure. For $\bb \xi \in D$, we denote by $\P_{\bb \xi}$ the probability measure on $(\Omega, \mathcal{A})$ under which $\bb U_i$ are i.i.d. with law $\mu_{\bb \xi}$, hence
$$ E_{\bb \xi} [f(\bb S_1, \dots, \bb S_n)]:= \E \big[ e^{\skalar{\bb \xi,\bb S_n}} f(\bb S_1, \dots, \bb S_n) \big]$$
for all $n \in \N$. Denote by $\mathds{U}_{\bb \xi}$ and $\wt{\mathds{U}}_{\bb \xi}$ the renewal measures associated with $\mu_{\bb \xi}$ and $\wt{\mu}_{\bb \xi}$, respectively. We write
$$ \bb m_{\bb \xi} := \E_{\bb \xi} [\bb U] = \E \big[ e^{\skalar{\bb \xi, \bb U}} \bb U \big].$$
Note that for  $\bb \xi \in D \cap \interior{I}$, $\bb m_{\bb \xi}=\nabla \phi(\bb \xi)$ and $\E_{\bb \xi} \norm{U}^m < \infty$ for all $m \in \N$.

	Our results will inter alia rely on asymptotic expansions for which we require a stronger nonlattice assumption, also known as {\em Cram\'er's condition}. Denoting by
	$$ \hat{\mu}_{\bb \xi}(t_1, t_2):= \E_{\bb \xi} \big[\exp(\imag (t_1 U_1 + t_2 U_2))] = \E \big[ \exp(\skalar{\bb \xi + \imag \bb t, \bb U}) \big]$$ the characteristic function of $\mu_{\bb \xi}$, we will assume that
	\begin{equation}\label{assump:cramer}\tag{A4c}%{Cramer}
		\limsup_{(|t_1|+|t_2|) \to \8} \big| \hat{\mu}_{\bb \xi}(t_1, t_2)\big| <1,
	\end{equation}
	where the value of $\bb \xi$ will be fixed in the main theorem.
	Note that \eqref{assump:cramer} implies \eqref{assump:nonarithmetic} both for $\mu$ and $\mu_{\bb \xi}$. 	
	Condition \eqref{assump:cramer} is satisfied, for example, if $\mu_{\bb \xi}$ (or equivalently, $\mu$)  has a nonzero, absolutely continuous component with respect to Lebesgue measure $\lambda^2$ on $\R^2$. See e.g. \cite[Chapter 20]{Bhattacharya2010} for further information.

\subsubsection{Notation and Preliminaries for Blocks}\label{subsect:blocks}

In a more general framework, we can consider recursions on $\R^d$ with diagonal matrices $\bb A =(A_1, \dots, A_d)$, such that the state space $\R^d$ can be seperated into two {\em blocks}, as follows. 
For $i,j \in \{1,\dots, d\}$, we define an equivalence relation $\sim$ by 
\begin{equation}\label{eq:equivalence.relation} %\tag{}{equiv.rel}
	i \sim j \quad :\Leftrightarrow \quad |A_i|^{\alpha_i}= |A_j|^{\alpha_j} \text{ a.s.}
\end{equation}
We restrict our analysis to the case when there are two equivalence classes $J_1$, $J_2$ of sizes $d_1$ and $d_2$, respectively. To avoid notational ambiguity, it is convenient to assume without loss of generality that
\begin{equation}\label{eq:equivalence.classes} %\tag{equiv.class}
	J_1=\{1, 3,4, \dots, d_1+1\} \text{ and } J_2=\{2, d_1+2, d_1+3, \dots, d_1+d_2\}.
\end{equation}
This gives in particular that \eqref{assump:notapower} is satisfied; and we may consider $|A_1|$ and $|A_2|$ as representatives (of the entries of $\bb A$) for the two equivalence classes. As a consequence, we may continue to consider assumptions formulated in terms of $A_1$ and $A_2$.

According to the equivalence classes $J_1$, $J_2$, we decompose $\bb A= \mathrm{diag}(\bb A^{(1)}, \bb A^{(2)})$, $\bb B=(\bb B^{(1)}, \bb B^{(2)})$ and $\bb X = (\bb X^{(1)}, \bb X^{(2)})$ with $\bb A^{(1)}=\mathrm{diag}(A_1, A_3, \dots, A_{d_1+1})$ and $\bb A^{(2)}=\mathrm{diag}(A_2, A_{d_1+2}, \dots, A_{d_1+d_2})$. The random vectors $\bb B^{(i)}$ and $\bb X^{(i)}$ are defined in the same way according to the equivalence classes.
It then holds for $i \in \{1,2\}$ that
\begin{equation}\label{eq:SRE.blocks}
	 \bb X^{(i)} \eqdist \bb A^{(i)} \bb X^{(i)} + \bb B^{(i)}.
\end{equation}
 We may restrict $\norma{\cdot}$ to $R^{d_1}$ or $\R^{d_2}$ in a natural way by setting for ${\bb x}=(\bb x^{(1)}, \bb x^{(2)})$ 
$$ \norma{\bb x^{(1)}}=\max_{i \in \{1, 3, \dots, d_1+1\}} |x_i|^{\alpha_i}, \qquad \norma{\bb x^{(2)}}=\max_{i \in \{2, d_1+2, \dots, d_2\}} |x_i|^{\alpha_i}.$$
 
As before, we identify $\bfA$ with $(\mathbf{U},\bb K)$, where $\bb U=(\alpha_1 \log |A_1|, \alpha_2 \log|A_2|)$ and $\bb K=\mathrm{diag} (k_1,...,k_{d_1+d_2})$ with $k_i=\mathrm{sign} A_i$.  Note that $K$, the group generated by $\supp(\bb K)$, is now a subgroup of $(\Z_2)^{d_1+d_2}$.
 It holds for any $\bb x \in \R^{d_1+d_2}$,
\begin{equation}\label{eq:decomp}
	\bfA \bb x = \left (e^{U_1\slash \bb \a^{(1)}}(\bb K\mathbf{x})^{(1)}, e^{U_2\slash \bb \a^{(2)}}(\bb K\mathbf{x})^{(2)} \right)
	=:e^{\mathbf{U}\slash \bb \a} \bb K \bb x.
\end{equation}

Observe that thus for $j=1,2$, 
%\begin{equation}\label{eq:decompNorm}
$	\norma{\bfA^{(j)}\bfX^{(j)}}=e^{U_j}\norma{\bfX^{(j)}}=|A_j|^{\alpha_j}\norma{\bfX^{(j)}}$.
%\end{equation}
Similar as before, let $\wt \mu $ be the law of $(\mathbf{U},\bb K)$ on $\R ^2 \times K$. Writing
$$ [\mathsf{blocks}]:=\big\{ \bb x \in \R^d : x_i=0 \text{ for all } i \in J_1\big\} \cup \big\{ \bb x \in \R^d : x_i=0 \text{ for all } i \in J_2\big\},$$
it has been proved in \cite[Theorem 3.1]{Damek2021} that there exists a nonzero Radon measure $\Lambda$ such that for all $f \in \bfC(\R^d \setminus \{ \bb 0\})$,
$$ \lim_{t \to \infty} t \E \big[ f\big(t^{-1/\bb \alpha} \bb X\big) \big] = \int_{\R^d \setminus \{\bb 0\}} f(\bb x) \Lambda(\bb dx),$$
and the measure $\Lambda$ is supported on $[\mathsf{blocks}]$. That is, between blocks, we have asymptotic independence. 
We will hence consider
\begin{equation*}
	 \bfC \left (\R^d\setminus [\mathsf{blocks}]\right ):= \left \{ f\in \bfC (\R ^d \setminus [\mathsf{blocks}]):  \supp f\subset \{ \mathbf{x}: \|\mathbf{x}^{(i)}\| _{\a} >\eta _i, i=1,2\ \mbox{for some}\ \eta _i>0\} \right \}
\end{equation*}
as well as
$$ \bb H^\epsilon(\R^d \setminus [\mathsf{blocks}]):= \Big\{ f \in \bfC(\R^d \setminus [\mathsf{blocks}]) \, : \, f \text{ is } \epsilon\text{-}\bb \alpha\text{-H\"older } \Big\}.$$
Upon defining
\begin{equation}\label{eq:psimoments}
	\psi (\bb \xi):=\E\big[ |A_1|^{\xi_1\alpha _1} \norma{\bfB ^{(2)}}^{\xi_2}\big] +\E \big[ \norma{\bfB ^{(1)}}^{\xi_1}\ |A_2|^{\xi_2\alpha _2}\big] + \E \big[ \norma{\bfB ^{(1)}}^{\xi_1}\norma{\bfB ^{(2)}}^{\xi_2}\big]
	+ \phi(\bb \xi)
	%\E \Pi _{j\in I}|\bfB ^{(j)}|^{\xi _j}\ \Pi _{j\notin I}a_j^{\xi _j}\ \mbox{for}\ I\subset 
	%\{ 1,2\}.
\end{equation}
we will require the following replacement for Assumption \eqref{assump:mixedmoments}:
\begin{equation}\label{eq:mixedblocks} \tag{A5b}%{mixed:moments:blocks}
\psi(1,1)<\infty	
%\E \left [|A_1|^{\alpha _1}\|\bfB ^{(2)}\| _{\a }\right ] +\E \left [\|\bfB ^{(1)}\| _{\a }|A_2|^{\alpha _2}\right ]+\E \left [\| \bfB ^{(1)}\| _ {\alpha }\|\bfB ^{(2)}\| _{\a }\right ] + \E \big[ |A_1|^{\alpha_1} |A_2|^{\alpha_2} \big]<\8
\end{equation}
Finally, we extend the notion of spheres without axes to spheres without blocks:
\begin{equation*}
	\wt S_r : = S_r^{d-1} \cap \left (\R ^d\setminus [\mathsf{blocks}]\right ).
\end{equation*}
The meaning will be clear from the context.
%\begin{equation}\label{eq:mixedA} \tag{mixed:moments:A}
%	\E \left [|A_1|^{\alpha _1}|A_2|^{\a _2}\right ]+\E \left [|A_i|^{\alpha _i}\left %|\log |A_2|\right |^2\right ]<\8
%\end{equation}

\section{Main Result}\label{sect:main result}

We start by providing more information about the  moment generating function 
$$\phi(\bb \xi)= \E \big[\exp(\skalar{\bb \xi, \bb U})\big] = \E \big[ |A_1|^{\alpha_1 \xi_1} |A_2|^{\alpha_2 \xi_2} \big].$$ 
Let $h(\bb \xi)=\xi _1+\xi _2$ and 
$$m:=\max \{ h(\bb \xi) \, : \, \bb \xi \in \Delta \cap [0,1]^2\}.$$
Our analysis will require that there is $\bb \xi^* \in D \cap (0,1)^2$ such that $h(\bb \xi^*)=m$.
The following lemma shows that such $\bb \xi^*$ is unique (given its existence); and provides simple sufficient conditions for the existence of $\bb \xi^* \in D \cap (0,1)^2$. Its proof will be given in Section \ref{sect.phi}.

\begin{lemma}\label{lem:properties.phi2}
	Assume \eqref{assump:alpha} and \eqref{assump:notapower}. 
	\begin{enumerate}
		\item It holds that $\phi(\bb \xi)<1$ on $\{ \bb \xi \in (0,1)^2 \, : \, \xi_1+\xi_2 \le 1 \} \cup \Big( (0,1) \times \{0\} \Big) \cup \Big( \{0\} \times (0,1) \Big).$
		\item If there is $\bb \xi^* \in \interior{D \cap I}$ with $h(\bb \xi^*)=m$, then $\nabla \phi(\bb \xi^*)$ is parallel to (1,1) and $\bb \xi^*$ is unique.
		\item Assume in addition \eqref{assump:moments}, \eqref{assump:positivedependence} and $\phi(1,1)<\infty$. Then $[0,1]^2 \subset I$, it holds that $\Delta \subset [0,1]^2$ and $D$ is a connected path from $(1,0)$ to $(0,1)$ in $[0,1]^2$ with the property that $1 < \xi_1 + \xi_2 < 2$ for all $\bb \xi \in D \cap (0,1)^2$. There is a unique $\bb \xi^* \in D \cap (0,1)^2$ with the property $\nabla \phi(\bb \xi^*)$ is proportional to $(1,1)$. This $\bb \xi^*$ satisfies $h(\bb \xi^*)=m$.
	\end{enumerate}
\end{lemma}

\begin{remark}\label{rem:xi.interior}
	Note that each of the assumptions \eqref{assump:mixedmoments} and \eqref{eq:mixedblocks} implies that $\phi(1,1)<\infty$. Given $\phi(1,1)<\infty$, the convexity of $\phi$ readily implies that $[0,1]^2 \subset I$. Then any $\bb \xi^* \in (0,1)^2$ satisfies $\bb \xi^* \in \interior{I}$, a fact that will be useful for us later.
\end{remark}

Rephrasing the last part of the lemma, we have that, under assumption \eqref{assump:positivedependence}, the set $D$ is an arc from $(0,1)$ to $(0,1)$ that, except for the endpoints, is strictly inside the square $[0,1]^2$ and strictly above the triangle $\{ \xi_1+\xi_2\le 1\}$. There is a unique point $\bb \xi^*$ on this arc, where the function $h(\bb \xi)=\xi_1+\xi_2$ is maximized, namely where the gradients $\nabla h(\bb \xi^*)$ and $\nabla \phi(\bb \xi^*)$ are parallel. This $\bb \xi^*$ will give the right order  of decay in \eqref{eq:result:intro}.

If \eqref{assump:positivedependence} fails, then it may happen that $D$ is partly  outside the square $[0,1]^2$. For our analysis it is sufficient that the function $h$ on $D \cap [0,1]^2$ is maximized inside the square. See Figure \ref{fig:plotPhi} for two plots of $\phi$, pertaining to the log-Gaussian example described below in Section \ref{sect:Examples}. In both cases, we have $\bb \xi^* \in D \cap (0,1)^2$.

\begin{figure}[h]
	\centering
	\includegraphics[width=.90\textwidth]{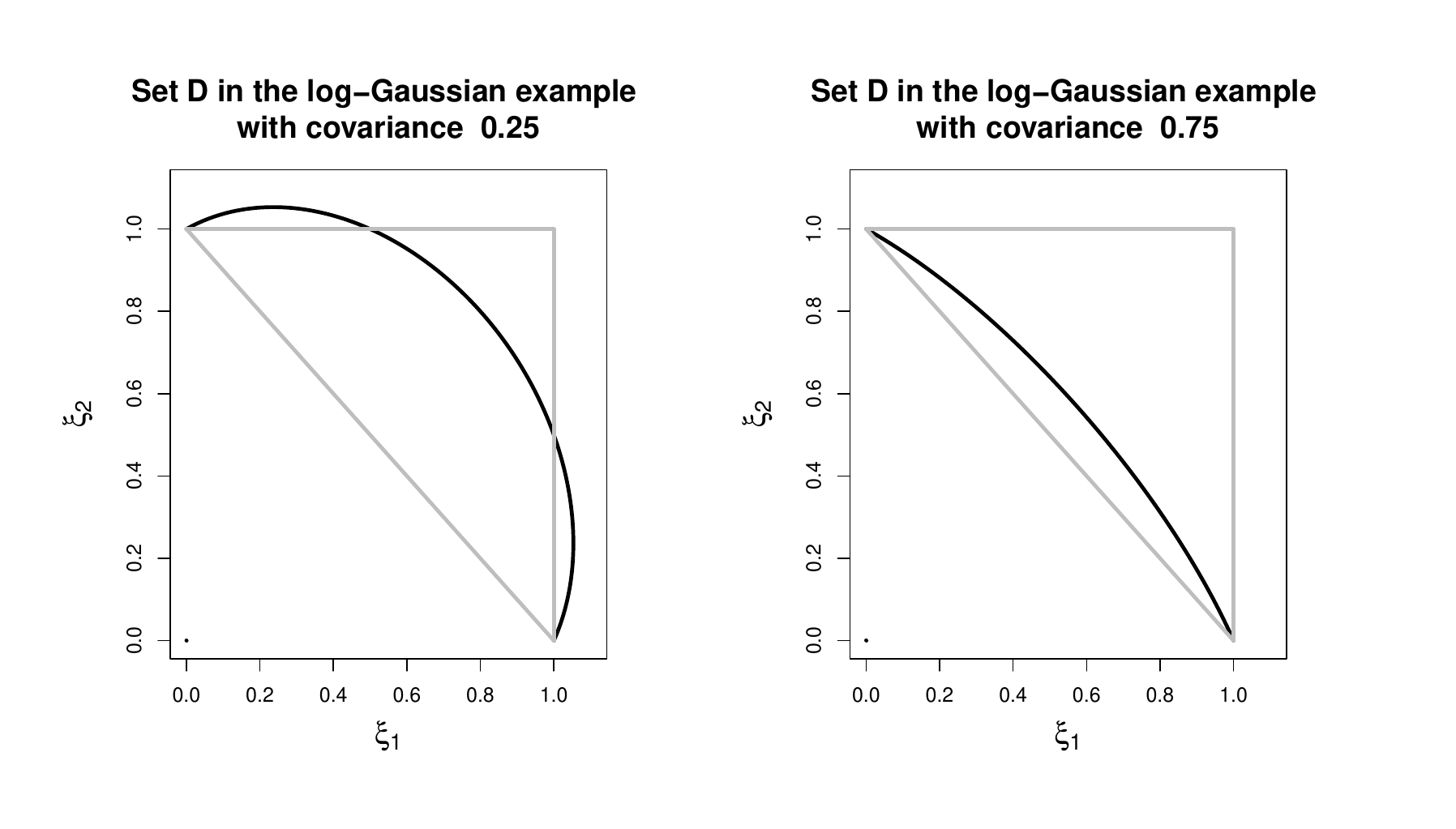}
	\caption{Plots of the level set $D$. Left: \eqref{assump:positivedependence} fails, Right: \eqref{assump:positivedependence} holds.}
	\label{fig:plotPhi}
\end{figure}

\medskip

Our first result is concerned with values $\bb \xi \in (0,1)^2 \cap  \mathrm{int}(\Delta)$, that is for $\bb \xi$ inside the square and below the arc $D$.

\begin{proposition}\label{prop:m-delta}
		Assume \eqref{assump:alpha}, \eqref{assump:moments}, \eqref{assump:nondeg},\eqref{assump:notapower}  and \eqref{assump:mixedmoments}. For any $\bb \xi \in [0,1)^2$ such that $\phi(\bb \xi)<1$, it holds that
	\begin{equation}\label{eq:m-delta}
		\lim_{t \to \infty} t^{\xi_1+\xi_2} \P\big( |X_1|>t^{1/\alpha_1}, |X_2|>t^{1/\alpha_2} \big)=0.
	\end{equation}
\end{proposition}

By Lemma \ref{lem:properties.phi2}, we have that  $\phi(\bb \xi)<1$ on the line $\{ \bb \xi \in (0,1)^2 \, : \, \xi_1+\xi_2=1\}$. Hence the proposition gives that the probability of a simultaneous extreme event at the $\bb \alpha$-scale is of order less than $t^{-(1+\delta)}$ for some $\delta>0$. Note that if $X_1$ and $X_2$ were independent, then the probability of a simultaneous extreme event at the $\bb \alpha$-scale would be of order $t^{-2}$.
Proposition \ref{prop:m-delta} will be proved in Section 6.

Our main result, as already announced in the introduction, shows that $\bb X$ exhibits hidden regular variation which is detected by extending the forbidden zone to the axes.

\begin{theorem}\label{thm:main}
	Assume \eqref{assump:alpha}, \eqref{assump:moments}, \eqref{assump:nondeg}, %\eqref{assump:notapower},
	 \eqref{assump:mixedmoments} and (as a sufficient condition for \eqref{assump:cramer}) that the law of $(\log |A_1|, \log |A_2|)$ has an absolutely continuous component with respect to the Lebesgue measure on $\R^2$.
	 Suppose $\P(|A_1|>1, |A_2|>1)>0$ and that $\supp(X)$ contains an open set.
	
		Assume that there is $\bb \xi^* \in D \cap (0,1)^2$ such that  $\nabla \phi(\bb \xi^*)$ is parallel to $(1,1)$. 
		Then there exists a non-zero Radon measure $\Lambda_2$ on $\R^2\setminus[\mathsf{axes}]$ such that for all $f\in \bfC (\R ^2\setminus[\mathsf{axes}])$, 
	\begin{equation}\label{eq:precise.simple}
	\lim_{t \to \infty}	(\log t)^{1/2}	t^{\xi^*_1+\xi^*_2}\E\big[ f(t^{-1/\alpha_1}X_1,t^{-1/\alpha_2}X_2) \big]= \int f(\bb x) \Lambda_2(d\bb x).
	\end{equation}
	In particular,
	\begin{equation}\label{eq:precise.prob}
		\lim_{t \to \infty}	(\log t)^{1/2}	t^{\xi^*_1+\xi^*_2}\P\big( |X_1|>t^{1/\alpha_1}, |X_2|>t^{1/\alpha_2} \big) \in (0,\infty).
	\end{equation}
\end{theorem}

\begin{remark}\label{rem:suppX.open.set}
	By \cite[Proposition 4.3.2]{Buraczewski2016} it holds that the law of $\bb X$ does not have atoms and is of pure type: either singular or absolutely continuous. Under the assumption that the law of $(\log |A_1|, \log |A_2|)$ has an absolutely continuous component, the absolute continuity of the law of $\bb X$ follows e.g. if $\bb B$ is independent of $\bb A$ (or constant). Then, in particular, $\supp(\bb X)$ contains an open set.
\end{remark}

Theorem \ref{thm:main} is formulated in such a way that it is directly applicable to various examples (see Section \ref{sect:Examples} below). It may be generalized in several ways. 
%As said before, assumption \eqref{assump:positivedependence} is not necessary for the assertions to hold; what is needeed is that the function $h(\bb \xi)=\xi_1+\xi_2$ achieves $m=\max \{ h(\bb \xi) \, : \, \bb \xi \in \Delta \cap [0,1]^2\}$ {\em inside} the square. 
Firstly, we may consider a situation where $d>2$ and the state space can be partioned into two blocks as described in Subsection \ref{subsect:blocks}. Then \eqref{assump:notapower} is automatically satisfied. Secondly, we can weaken the assumptions on the law of $(\log |A_1|, \log |A_2|)$. 
We include these generalizations together with additional information about the structure of $\Lambda_2$ in our final result which proves hidden regular variation on $\R^d$ where the forbidden zone is $[\mathsf{blocks}]$. 
%
%
%, it has been proved in \cite[Theorem 3.1]{Damek2021} that within each {\em block} $\R^{d_1}$ or $\R^{d_2}$, simultaneous extremes occur; while between blocks, we have asymptotic independence. Our final result provides

\begin{theorem}\label{thm:main:blocks}  Assume \eqref{assump:alpha},  \eqref{assump:moments}, \eqref{assump:nondeg} and \eqref{eq:mixedblocks}. Suppose that the equivalence relation defined in \eqref{eq:equivalence.relation} has two equivalence classes as described in \eqref{eq:equivalence.classes}. Assume that there  is  $\bb \xi^* \in D \cap (0,1)^2$ such that $m=h(\bb \xi^*)$ and that \eqref{assump:cramer} is satisfied for $\mu_{\bb \xi^*}$. Then there exists a Radon measure $\Lambda_2$ on $\R^d\setminus[\mathsf{blocks}]$ such that for all $f\in \bfC (\R ^d\setminus[\mathsf{blocks}])$,
	\begin{equation}\label{eq:precise}
		\lim_{t \to \infty} (\log t)^{1/2}	t^{\xi^* _1+\xi^* _2}\E \big[ f(t^{-1/\bb \alpha}\bb X)\big]= \int f(\bb x) \Lambda_2(d\bb x) .
	\end{equation} 
	The measure $\Lambda_2$ has the following invariance properties: For all $k \in K$ and $t>0$,
	\begin{equation}\label{eq:measure1}
		\int f(k \bb x) \Lambda_2(d \bb x) = \int f(\bb x) \Lambda_2(d \bb x) \ \text{ and } \ \int f(t^{1/\bb \alpha} \bb x) \lambda(d \bb x) = t^{\xi^*_1 + \xi^*_2} \int f(\bb x) \Lambda_2(d \bb x)
	\end{equation}
	as well as, for all $r>0$,
	\begin{equation}\label{eq:measure3} 
		\Lambda _2(\wt S _r) = 0.
	\end{equation}
	Moreover, there is a Radon measure $\nu $ on $\wt S_1$   such that for all $f\in \bfC \left (\R ^d\setminus [\mathsf{blocks}]\right )$ 
	\begin{equation}\label{eq:measure4}
		\int f(\bb x)\Lambda _2(d\bb x) = \int _{\wt S_1}\int _0^\8 f(s^{1\slash {\bb \a} }\omega )\frac{ds}{s^{\xi^* _1+\xi^* _2+1}}d\nu (\omega),
	\end{equation}
	and this measure $\nu$ is invariant under the action of $K$.
	 and satisfies, for all $g \in \mathcal{C}_c(\wt{S}_1)$,
%	\begin{equation}
%		\int g(\omega) \nu(d\omega) = \E \int \norma{\bb A}^{\xi^*_1+\xi^*_2} g(\norma{\bb A\omega}^{-1/\bb \alpha}\bb A\omega) \nu(d\omega). \label{eq:invariance.nu}
%		\end{equation}
	The measures $\Lambda_2$ and $\nu$ are nonzero if and only if for every $R>0$
	\begin{equation}\label{eq:suppX:unbounded}
		\P( \norma{X^{(1)}}>R, \, \norma{X^{(2)}}>R)>0.
	\end{equation}
If \eqref{eq:suppX:unbounded} holds, then the measure $\nu$ is unbounded. 
\end{theorem}

\begin{remark}
	The property \eqref{eq:suppX:unbounded} is obviously necessary for \eqref{eq:precise}. Its sufficiency (under the assumptions of Theorem \ref{thm:main:blocks}) is proved in Section \ref{sect:positivity}. We prove in Lemma \ref{lem:X:unbounded} the following. 	Suppose there are $(\bb a, \bb b)$  in the semi-group generated by  $\supp (\bfA , \bfB )$ such that both $|a_1|>1$ and $|a_2|>1$, and let $\bb x_0:=(\Id - \bb a)^{-1} \bb b$. Suppose that there is $\bb x_1  \in \supp (\bb X)$ such that $(\bb x_1 - \bb x_0)^{(i)} \neq \bb 0$ for both $i \in \{1,2\}$. Then \eqref{eq:suppX:unbounded} holds. 
	
	These assumptions are weak and easy to check in examples. Further sufficient conditions for the latter assumptions are provided in Remark \ref{rem:suff.cond.unbounded}.
\end{remark}

Theorem \ref{thm:main} is a special case of Theorem \ref{thm:main:blocks}. The proof of Theorem \ref{thm:main:blocks} will be given in Section \ref{sect:implicit.renewal} except for the last two assertions (nontriviality and unboundedness of $\nu$), which will be proved in Section \ref{sect:positivity}.

\section{Examples}\label{sect:Examples}

\subsection{CCC-GARCH(1,1)}

The Constant Conditional Correlation (CCC) GARCH model, introduced by \cite{Bollerslev1990}, is defined as follows (cf. \cite[Definition 10.4]{Francq2019}). Let $(\bb{Z}_n)$ be a sequence of i.i.d. centered random vectors with covariance matrix $\bb{C}$. A process $(\sigma_{n,1}^2, \dots, \sigma_{n,d}^2)$ is called CCC-GARCH(1,1) if it satisfies
\begin{equation}
	\begin{cases}
		& {\bb R}_n = \mathrm{diag}(\sigma_{n,1}^2, \dots, \sigma_{n,d}^2) \bb{Z}_n \\[.2cm]
		& \sigma_{n,i}^2= a_i +b_i \sigma_{n-1,i}^2 + c_i R_{n-1,i}^2 = a_i +b_i \sigma_{n-1,i}^2 + c_i Z_{n-1,i}^2 \sigma_{n-1,i}^2
	\end{cases}
\end{equation}
for positive constants $a_i$, $b_i$ and $c_i$. Note that this is the original model introduced by \cite{Bollerslev1990}; the general CCC-GARCH model of \cite{Starica1999} contains additional cross terms $\sum_{j\neq i} b_{i,j} \sigma_{t-1,j}^2 + c_{i,j} R_{t-1,j}^2$ in the formula for $\sigma_{t,i}$.

If we put $\bb{X}_n=(\sigma_{n,1}^2, \sigma_{n,2}^2)$, the CCC-GARCH(1,1) model with dimension $d=2$ fits in our framework with  
$$ \bb{A}=(b_1+c_1 Z_{1}^2, b_2+c_2 Z_{2}^2), \qquad \bb{B}_n = (a_1, a_2),$$
where $\bb{Z}=(Z_1,Z_2)$ is a centered random vector with a fixed covariance matrix $\bb{C}$. 
 Validity of \eqref{assump:alpha}, \eqref{assump:moments}, \eqref{assump:nondeg} has been checked in \cite{Mentemeier2022} for the case where $\bb Z$ is Gaussian. Assuming furthermore that $\bb C$ is strictly positive definite, we have that $(\log |A_1|, \log|A_2|)$ has a continuous density with respect to Lebesgue measure on $\R^2$ and $\P(|A_1|>1, |A_2|>1)>0$. Since $\bb B$ is constant and all moments of $\bb Z$ are finite, Assumption \eqref{assump:mixedmoments} is also satisfied. By Remark \ref{rem:suppX.open.set} it also follows that $\supp(\bb X)$ contains an open set.
	
For given model parameters, the existence and value of $\bb \xi^*$ can be obtained by numerical approximation of $\phi$. If $\bb \xi^* \in (0,1)^2$ exists, then Theorem \ref{thm:main} is applicable and we obtain hidden regular variation at the rate $t^{\xi_1^*+\xi_2^*}.$

\begin{figure}[h]
	\centering
	\includegraphics[width=.95\textwidth]{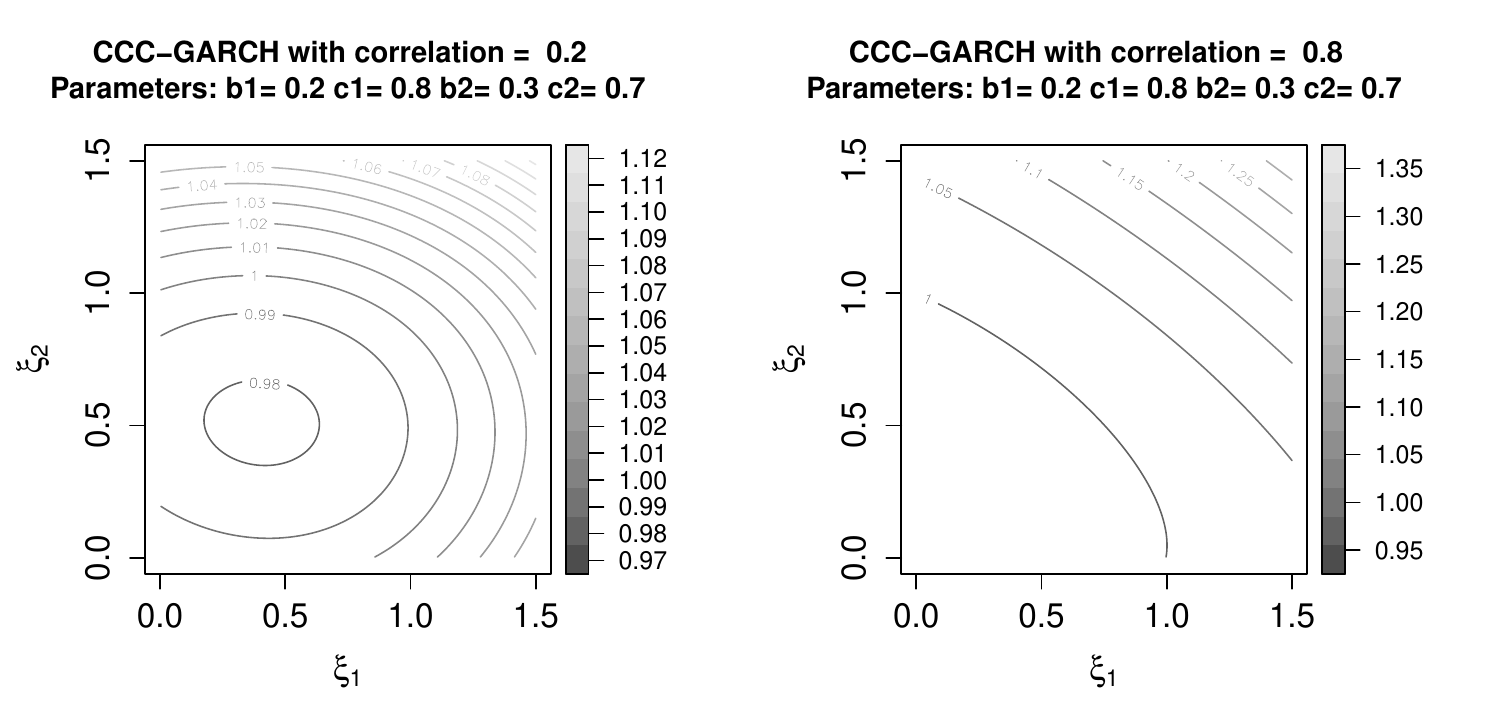}
	\caption{Plots of $\phi$ for different parameters in the CCC-GARCH-model}
	\label{fig:plotCCC}
\end{figure}

Whether \eqref{assump:positivedependence} is satisfied or not depends on the model parameters, see Figure \ref{fig:plotCCC} for two plots of $\phi$ for different sets of parameters. We may also give a qualitative illustration in a simple case where  $Z_1, Z_2$ are standard Gaussian with covariance $\eta \in (-1,1)$.
%, that is
%$$ \bb C = \left(\begin{matrix}
%	1 & \eta \\
%	\eta & 1
%\end{matrix}\right).$$
We may use that $(Z_1,Z_2)$ has the same law as $(Z_1,\eta Z_1 + (1-\eta)Z')$ where $Z_1$ and $Z'$ are independent. Consider further the parameters $b_1=b_2=c_1=c_2=0.5$. Then $\alpha_1=\alpha_2=1$
and we obtain
\begin{equation}\label{eq:posdep:CCC}
	 \E \big[ |A_1|^{\alpha_1} \log |A_2| \big] = \E \big[ (b_1 + c_1 Z_1^2) \log (b_1 + c_1 (\eta^2 Z_1^2 + 2\eta(1-\eta) Z_1 Z') + (1-\eta)^2Z'^2))  \big] 
\end{equation}
as well as a similar expression for $\E \big[ |A_2|^{\alpha_2} \log |A_1| \big]$. If we let $\eta \to 1$ in \eqref{eq:posdep:CCC}, then the expression tends to $\E \big[ |A_1|^{\alpha_1} \log |A_1| \big]>0$, while for $\eta \to 0$ the expression tends to $\E \big[ |A_1|^{\alpha_1}\big] \E [\log|A_2|]<0$.

\subsection{Diagonal BEKK-ARCH(m)}

The Gaussian BEKK-ARCH(m) (or BEKK(1,0,m)) model, introduced by \cite{Engle1995}, is defined as follows (cf. \cite[Definition 10.5]{Francq2019}). Let $(\bb Z_n)$  be a sequence of i.i.d.\ Gaussian random vectors with $\bb Z_1 \sim \mathcal{N}(\bb 0,\Id_d)$. A process $\bb X_n \in \R^d$ is called a Gaussian BEKK-ARCH(m)-process if it satisfies
\begin{equation}
	\begin{cases}
		& {\bb X}_n = \bb \Sigma_n^{1/2} \bb{Z}_n \\[.2cm]
		& \bb \Sigma_n = \bb C + \sum_{\ell =1}^m \bb A_{\ell} \bb{X}_{n-1} \bb{X}_{n-1}^\top A_{\ell}^\top	\end{cases}
\end{equation}
for a deterministic positive definite $d \times d$ matrix $C$ and deterministic $d \times d$ square matrices $A_{\ell}$, $1\le \ell \le m$.

It was observed in \cite{Pedersen2018} that (a version of) $\bb X_n$ satisfies a stochastic recurrence equation
$$ \bb X_n = \bb A_n \bb X_{n-1} + \bb B_n$$
where $(\bb A_n, \bb B_n)$ are i.i.d.\ copies of $(\bb A, \bb B)$ which are given as follows: The matrix $\bb A$ is independent of $\bb B \sim\mathcal{N}(\bb 0, \bb C)$, and
$$ \bb A = \sum_{\ell=1}^m M_{\ell} A_{\ell}$$
where $M_1, \dots, M_m$ are i.i.d.\ standard Gaussian random variables.

Hidden regular variation properties of the stationary solution can be analyzed using our results if $d=2$, $m \ge 2$ and $A_{\ell}$, $1 \le \ell \le m$ are diagonal matrices. Observe that in this setting,
$$ 
\left( \begin{matrix}
	A_1 \\
	A_2
\end{matrix}
\right)  =
\left(
\begin{matrix}
	A_{1,1} & A_{2,1} & \cdots & A_{m,1} \\
	A_{1,2} & A_{2,2} & \cdots & A_{m,2} 
\end{matrix}
\right)
\left(
\begin{matrix}
	M_1 \\ M_2 \\ \vdots \\ M_m
\end{matrix}
\right)
$$
and $(A_1, A_2)$ is a fully bivariate Gaussian as soon as the matrix $(A_{\ell,i})_{1 \le \ell \le m, 1\le i \le 2}$ has rank 2; this we need to assume. 
If the entries of $A_{\ell,m}$ are sufficiently small, the conditions \eqref{assump:alpha}, \eqref{assump:moments} and \eqref{assump:nondeg} are readily checked, see \cite[Section 2]{Pedersen2018} for details. Since $\bb A$ and $\bb B$ are independent with Gaussian laws, Assumption \eqref{assump:mixedmoments} is satisfied, as well as the conditions $\P(|A_1|>1, |A_2|>1)>0$ and that $\supp(\bb X)$ contains an open set.

As in the previous example, Assumption \eqref{assump:positivedependence} may be valid or not, depending on the correlation between $A_1$ and $A_2$. Nevertheless, existence and the value of $\bb \xi^*$  can be assessed in a simple way through numerical approximation of $\phi$. If $\bb \xi^*$ exists, then Theorem \ref{thm:main} is applicable.

Note that here $K=\Z_2 \times \Z_2$ which implies that the tail behavior is equal in all quadrants, i.e.,
$$ \lim_{t \to \infty} \, (\log t)^{1/2} t^{\xi^*_1+\xi^*_2} \P( \pm X_1 >t, \pm X_2 >t) =c >0$$
for any combination of signs, as a consequence of the $K$-invariance of $\Lambda_2$.

\begin{figure}[h]
	\centering
	\includegraphics[width=.95\textwidth]{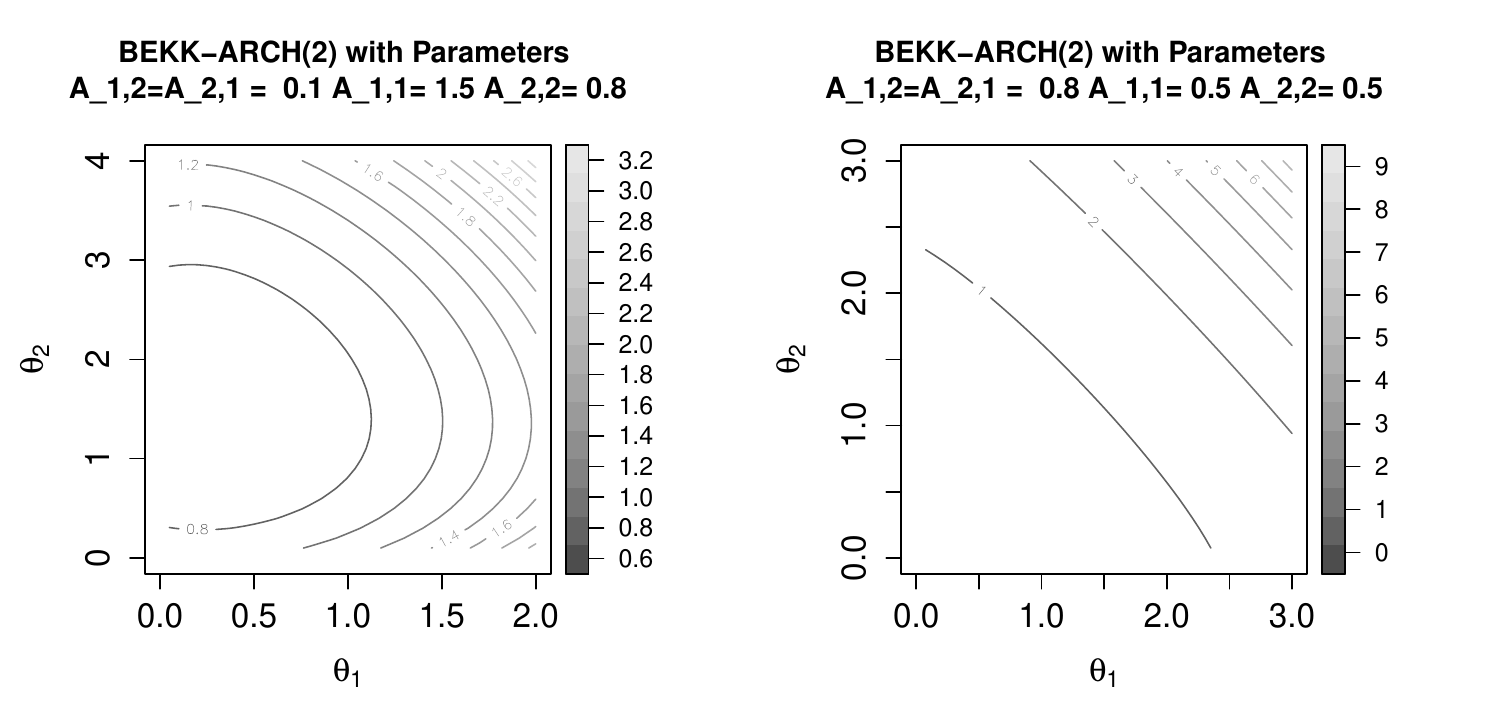}
	\caption{Plots of $\bb \Phi(\bb \theta):=\E\big[ |A_1|^{\theta_1} |A_2|^{\theta_2} \big]$ for different parameters in the BEKK-ARCH(2)-model}
	\label{fig:plotBEKK}
\end{figure}

\subsection{A Diagonal SRE with Log-Gaussian Multiplicative Noise}
The following example is not motivated by applications, but by the fact that explicit (analytic) calculations of  $\bb \Phi(\bb \theta):=\E\big[ |A_1|^{\theta_1} |A_2|^{\theta_2} \big]$ are available.
%; which allows to directly check assumptions of our main results on the parameters of the model, without possibly relying on numerical approximations.
Suppose that $\m{A}=(\exp(Z_1), \exp(Z_2))$ where $(Z_1,Z_2) \sim \mathcal{N}(\bb{m}, \m{C})$, {\em i.e.}, the entries of $\m{A}$ have a joint log-Gaussian distribution. We assume that $\m{C}$ is positive definite. Let the random vector $\bb{B}$ be independent of $\m{A}$ and have any distribution that satisfies the moment assumptions \eqref{assump:alpha} for $\alpha_i$ to be calculated below - for simplicity, we may assume a distribution such that $\bb{B}$ has all moments. 

By the choice of our model, we have $\bb{U}=(Z_1,Z_2)$, and by direct computations
\begin{align*}
	\bb \Phi(\bb \theta) &= \int_{\R^2} \frac{1}{2 \pi \det(\m{C})^{1/2}} \exp(\skalar{\bb \theta,\bb{z}}) \exp(-\frac12 \skalar{\bb{z}-\bb{m},\m{C}^{-1}(\bb{z}-\bb{m})}) \, d\bb{z}\\
	&= \exp \big( \skalar{\bb{m},\bb \theta} + \frac12\skalar{\bb \theta, \m{C} \bb \theta}\big)
\end{align*}
Hence the set $\{\bb  \theta \, : \, \bb \Phi (\bb \theta)=1\}$ is a quadric, namely an ellipse with principal axes parallel to the eigenvectors of $\m{C}$.

% and with center at
%$$ \frac{1}{\sigma_1^2 \sigma_2^2(\rho^2-4)} (2\sigma_2^2 m_1- \rho \sigma_1 \sigma_2 m_2, 2 \sigma_1^2 m_2 - \rho \sigma_1 \sigma_2 m_1)$$
%The length ratios (semi-major axis $a$:semi-minor axis $b$)$^2$ and (larger eigenvalue : smaller eigenvalue) are equal; but the {\em minor} axis is in direction of the principal eigenvector.  The actual scaling follows from the relation $a^2+b^2=\sigma_1^2 + \sigma_2^2$. The focal points are at distance $c$, satisfying $c^2=a^2-b^2$ from the center, in direction of the semi-major axis.
%??? Wikipedia, search reference
% An ellipse possesses the following property:
%
%The normal at a point P {\displaystyle P} bisects the angle between the lines P F 1 ¯ , P F 2 ¯ {\displaystyle {\overline {PF_{1}}},\,{\overline {PF_{2}}}}.
% F_1, F_2 Brennpunkte

Turning first to the assumptions on the marginal laws, we have,  with $\bb{e}_1=(1,0)$ and $\bb{e_2}=(0,1)$, the identity
\begin{align}
	\log \E [A_i^s] = \log  \bb \Phi(s \bb{e}_i) = sm_i + \frac12 s^2 \skalar{\bb{e}_i, \m{C} \bb{e_i}}. \label{eq:EAs:logGaussian}
\end{align}
Since the second term is a positive definite quadratic form, $m_i<0$ for $i \in \{1,2\}$ is necessary and sufficient for the existence of $\alpha_i$ as in condition \eqref{assump:alpha}: these are the positive roots of the quadratic functions in \eqref{eq:EAs:logGaussian}.
For explicit calculations, let us assume further that
$$ \m{C} = \left( \begin{matrix}
	1 & \eta \\ \eta & 1
\end{matrix} \right)$$
for some $\eta \in (-1,1)$.
Then
$$ \log \bb \Phi(\bb \theta)= m_1 \theta_1 + m_2 \theta_2 + \frac12  \theta_1^2+ \eta \theta_1 \theta_2 + \frac12\theta_2^2; \qquad \log \bb \Phi(s \bb{e}_i)= sm_i +\frac12 s^2 $$
and we obtain $ \alpha_i = - 2 m_i$. In order to check \eqref{assump:positivedependence}, we consider
 $$ \E \big[ A_1^{\alpha_1} \log A_2\big] =\frac{\partial}{\partial \theta_2} \bb \Phi( (\alpha_1,0) ) = m_2 +\eta \alpha_1 = m_2 -2 \eta m_1; \quad \E \big[ A_2^{\alpha_2} \log A_1\big] = m_1 - 2 \eta m_2.$$
Recalling that $m_1, m_2$ are negative, it is required here that the value of $\eta$ is sufficiently large, and we have to exclude cases where $|m_1|>2 |m_2|$ or vice versa.
Figure \ref{fig:plotPhi} shows two plots of the set $D$ for this model with $m_1=m_2=-0.5$ and different values of $\eta$. 

We briefly comment on the remaining assumptions: Conditions \eqref{assump:mixedmoments} and \eqref{assump:nondeg} are satisfied since $A_i$ are independent of $B_i$.  Condition \eqref{assump:cramer} is satisfied since $\bb U$ has a nondegenerate Gaussian distribution on $\R^2$. This also implies that $\P(|A_1|>1, |A_2|>1)>0$ and, using as well the independence of $\bb A, \bb B$ that $\supp(X)$ contains an open set (cf. Remark \ref{rem:suppX.open.set}). Hence Theorem \ref{thm:main} applies.
%
%
%\medskip
%
%Let us finish with explicit computations.
%If we choose $m_1=m_2=-\frac12$, we have $\alpha_1=\alpha_2=1$; and hence $\phi(\bb \xi)=\bb \Phi(\bb \xi)$. We have to solve
%$$ -\xi_1-\xi_2 +2 \eta \xi_1\xi_2 + \xi_1^2+\xi_2^2 = 0.$$
%With $\eta=\frac{15}{49}$ % solving $(1+\eta)\frac{98}{256}=\frac12$
%we have that $$\xi^*_1=\xi^*_2= \frac{15}{16}$$ is  a solution, at which the gradient (due to symmetry) is parallel to $(1,1)$. %??? check, simplify. 
%Indeed, the gradient is given by
%$$ (m_1+\xi_1+\eta \xi_2, m_2+\xi_2 + \eta \xi_1)$$
%using here as before that whenever $\phi(\xi)=1$, then its derivatives equal that of $\log \phi$ in $\xi$.  Hence
%$$ \nabla \phi \left( \frac{15}{16},\frac{15}{16} \right) = (-\frac12 + (1+\eta)\frac{15}{16},-\frac12 + (1+\eta)\frac{15}{16})=\frac{71}{98}(1,1).$$

\section{Properties of $\phi$}\label{sect.phi}

In this section, we prove Lemma \ref{lem:properties.phi2}, thereby providing additional properties of the moment generating function $\phi$ that are valid in our setting.

\begin{proof}[Proof of Lemma \ref{lem:properties.phi2}]
	
	 The first assertion follows since $\phi$ is convex and we have $\phi(0,0)=1$ as well as $\phi(1,0)=\phi(0,1)$ due to Assumption \eqref{assump:alpha}. In addition, due to \eqref{assump:notapower}, we have a strict inequality when applying the Hölder inequality to points on the line $\xi_1+\xi_2=1$:
	 \begin{equation}\label{eq:phi.xi1xi2} \phi((\xi, 1-\xi)) = \E \big[ |A_1|^{\xi \alpha_1} |A_2|^{(1-\xi)\alpha_2} \big]~<~ \big( \E [|A_1|^{\alpha_1}] \big)^{1/\xi} \big( \E [|A_2|^{\alpha_2}]\big)^{1/(1-\xi)} =1.\end{equation}
	 %This shows that $D$ is above the line $\xi_1+\xi_2$.
	 Considering the second assertion, we have that $D \neq \{\bb 0\}$. It then follows from \cite[Lemma 1.1]{Hoeglund1988} that $\nabla \phi(\bb \xi) \neq 0$ for all $\bb \xi \in D \cap \interior{I}$. 
	 If there is now $\bb \xi^* \in \interior{D \cap I}$ with $h(\bb \xi^*)=m$, we may parametrize $D$ locally around $\bb \xi^*$ by a $C^1$ curve  $\bb \gamma :(-\eps ,\eps ) \to D$  with $\bb \gamma (0)=\bb \xi ^*$ and $\bb \gamma '(0)\neq \bb 0$. 
	Observe that $\phi (\bb \gamma (r))=1$ and so
	\begin{equation*}
		0=\frac{d}{dr}(\phi (\bb \gamma (r)))\big |_ {r=0}=\langle \nabla \phi (\bb \gamma (0)), \bb \gamma '(0)\rangle .
	\end{equation*}
	Moreover, $r\to h(\bb \gamma (r))$ has its local maximum at $0$, hence
	\begin{equation*}
		0=\frac{d}{dr}(h (\bb \gamma (r)))\big |_ {r=0}=\langle \nabla h(\bb \gamma (0)), \bb \gamma '(0)\rangle =
		\langle (1,1), \gamma '(0)\rangle.
	\end{equation*}
	This proves that 
	$\nabla \phi (\bb \xi ^*)$ and $(1,1)$ are proportional (since both of them are perpendicular to $\bb \gamma'(0)$). But  it is proved in \cite[Lemma 1.3]{Hoeglund1988} that the mapping 
	$\bb \xi \to \frac{\nabla \phi(\bb \xi)}{\norm{\nabla \phi(\bb \xi)}}$ is one-to-one on $D \cap \interior{I}$. Hence,  $\bb \xi ^*$ is unique on $D$, and $m$ cannot be obtained for $\bb \xi \in \Delta$ with $\phi(\bb \xi)<1$ (since by increasing one or two coordinates there will still be $\bb \xi \in \Delta$ with larger value of $h$.).
	
	Turning to the third assertion, we first note that $\phi(1,1)<\infty$ implies that $[0,1]^2 \subset I$.  Under 
	Assumption \eqref{assump:positivedependence}, it holds by a Taylor expansion at $(1,0)$ that
	$$ \phi(1,h) = 1 + h \E [|A_1|^{\alpha_1} \log |A_2|] + o(h)>1 \quad \text{for }h>0,$$
	By convexity it follows that $\phi(1,\xi_2)>1$ for all $\xi _2\in (0,1]$.  Using further that $\phi(0,\xi_2)<1$ for all $\xi _2 \in (0,1)$, we obtain, using again the convexity, that for all $\xi_2 \in (0,1)$ there is a $\xi_1 \in (0,1)$ such that $\phi(\xi_1, \xi_2)=1$. 
	In the same way, for all $\xi_1 \in (0,1)$ there is $\xi_2 \in (0,1)$ such that $\phi(\xi_1, \xi_2)=1$; and $\phi(\xi_1, 1)>1$ for all $\xi_1>0$. It follows that $D \subset [0,1]^2$.
	
	We next show that $D \cap [0,1]^2$ is a connected path. By the above, we know that $D \cap (0,1)^2$ is nonempty. Note that $D \cap [0,1]^2$ is a closed set. Suppose $D$ would be separated into disjoint closed sets $D_1$, $D_2$, then  there would be endpoints $\bb \xi_1$ and $\bb \xi_2$ in $(0,1)^2$. But since $\nabla \phi \neq \bb 0$, we can employ the implicit function theorem to show that we can locally extend $D$ around $\bb \xi_1$ and $\bb \xi_2$ into a neighbourhood; which gives a contradiction. Hence $D$ is connected. 
 As a consequence of \eqref{eq:phi.xi1xi2}, $D \cap (0,1)^2$ is always above the line  $\xi_1+\xi_2=1$. 
  In particular, $h(\bb \xi)>1$ on $D \cap (0,1)^2$, hence the function $h$ attains it maximum on the compact set $D \cap [0,1]^2$ inside the square, i.e., at some value $\bb \xi^* \in D \cap (0,1)^2$. By the above considerations, $\nabla \phi(\bb \xi^*)$ is parallel to $(1,1)$. 
\end{proof}

\section{Moments}\label{sect:moments}

In this section we provide the proof of Propositon \ref{prop:m-delta} together with moment estimates that will be used in the subsequent sections. The results of this Section are, for simplicity, formulated for the two-dimensional setup, but extend without effort to the block setup, as we will explain in Theorem \ref{thm:moments:blocks} at the end of this section. 

The proof of Proposition \ref{prop:m-delta} proceeds by an application of the Markov inequality, the burden of the proof will be to prove the finiteness of suitable moments of $\bb X$.
\begin{proof}[Proof of Proposition \ref{prop:m-delta}]
Since $\bb \xi \in  [0,1)^2$ with $\phi(\bb \xi)<1$ and $[0,1]^2\subset I$ due to assumption \eqref{assump:mixedmoments}, there is $\delta>0$ such that $\bb \xi':=(\xi_1+\delta, \xi_2)$ also satisfies $\phi(\bb \xi')<1$ and $\bb \xi' \in [0,1)^2$. 
	We have
	\begin{align*}
		\P\big( |X_1|>t^{1/\alpha_1}, |X_2|>t^{1/\alpha_2} \big) &= \P\big( |X_1|^{(\xi_1+\delta)\alpha_1}>t^{\xi_1+\delta}, |X_2|^{\xi_2 \alpha_2}>t^{\xi_2} \big) \\
		&\le \P\big( |X_1|^{(\xi_1+\delta)\alpha_1} \cdot|X_2|^{\xi_2 \alpha_2}>t^{\xi_1+\delta+\xi_2} \big) \\
		&\le \E \Big[ |X_1|^{(\xi_1+\delta)\alpha_1} \cdot|X_2|^{\xi_2 \alpha_2} \Big] \cdot t^{-(\xi_1+\xi_2+\delta)}
	\end{align*}
	and the result follows once it is proved that the expectation is finite. This is the content of Theorem \ref{thm:moments} below.
\end{proof}

To obtain the required estimates for the mixed moments of $X_1$, $X_2$, it is - within this section - useful to consider the functions
\begin{equation}\label{eq:bigphi}
\bb \Phi (\bb \theta )=\E \big[ |A_1|^{\theta _1}\ |A_2|^{\theta _2}\big] 
\end{equation}\begin{equation}\label{eq:mixedmoments}%\tag{mixed:moments}
\bb \Psi (\bb \theta )	= \E \big[ |B_1|^{\theta _1}|B_2|^{\theta_2}\big]+\E \big[ |B_1|^{\theta _1}\ |A_2|^{\theta _2}\big] + 
\E\big[ |A_1|^{\theta _1} |B_2|^{\theta  _2}\big] + \bb \Phi(\bb \theta)
\end{equation}
instead of $\psi$ and $\phi$ - observe that $\psi(\bb \xi)=\bb \Psi(\bb \xi \bb \alpha)$ and $\phi(\bb \xi)=\bb \Phi(\bb \xi \bb \alpha)$. The reason will become transparent in Lemma \ref{lem:g}: we will proceed by some kind of induction on integer values with respect to $\bb \theta$; working with $\bb \xi$, we would have to consider non-integer steps.

\begin{theorem}\label{thm:moments} Consider $\bb \xi \in [0,1)^2$ with $\phi(\bb \xi)<1$. Set $\bb \theta:=\bb \xi \bb \alpha$. Suppose that $\Psi(\bb \theta )<\8$. Then
	\begin{equation*}
		\mathfrak{M}(\bb \theta):=\E \big[|X_1|^{\theta _1}|X_2|^{\theta _2} \big] = \E \big[ |X_1|^{\xi _1\a _1}|X_2|^{\xi _2\a _2}\big]<\8 .
	\end{equation*}
\end{theorem}

The proof of Theorem \ref{thm:moments} will be based on several lemmata which we will introduce first; the proofs are given at the end of the section.

Recall that we write $\bb \theta '<\bb \theta $ if $\bb \theta '_i\leq \bb\theta _i$ for $i \in \{1,2\}$ and there is $i$ such that $\bb \theta '_i< \bb \theta _i$.
\begin{lemma}\label{lem:properties.Psi}
	Assuming \eqref{assump:alpha}, \eqref{assump:moments} and \eqref{assump:mixedmoments}, the following holds:
	$\Psi(\bb \alpha)<\infty$, and for every $\bb \theta $ with $\bb 0 < \bb \theta \le \bb \alpha$, it holds as well that $\Psi(\bb \theta)<\8$.
	
	In the same way, if $\mathfrak{M}(\bb \theta)$ (or $\bb \Phi(\bb \theta)$) is finite for some $\bb \theta \le  \bb \alpha$, then $\mathfrak{M}(\bb \theta')<\infty$ (or $\bb \Phi(\bb \theta')< \infty$) for any $\bb 0 < \bb \theta' \le  \bb \theta$.
\end{lemma}

The next two lemmata prove the assertion of Theorem \ref{thm:moments} in the special cases that a) $\bb \theta \in [0,1]^2$ or b) one of the components of $\bb \theta$ is zero.

\begin{lemma}\label{lem:one}
	Let $\bb \theta \in [0,1]^2$ be such that $\bb \theta < \bb \alpha$. Assume that $\bb \Phi(\bb \theta)<1$ and $\bb \Psi(\bb \theta)<\infty$. Then $\mathfrak{M}(\bb \theta)<\infty$. % $M_{\bb \theta}$ is integrable.
\end{lemma}

\begin{lemma}\label{lem:two}
	Let $\bb \theta =(\theta_1,0)$ with $0<\theta_1 <\alpha_1$ or $\bb \theta =(0,\theta_2)$ with $\theta_2<\alpha_2$. Then $\mathfrak{M}(\bb \theta)<\infty$.% Then $M_{\bb \theta}$ is integrable.
\end{lemma}

In other words, Lemma \ref{lem:two}  states that $\E \big[ \abs{X_i}^{\theta_i}\big]<\infty$ for all $\xi_i < \alpha_i$. This follows directly from \eqref{eq.intro.marginal.tails}, but can also proved in an elementary way, see e.g. \cite[Lemma 2.3.1]{Buraczewski2016}.

If $\theta_1 >1$ or $\theta_2 >1$, then we will consider the auxiliary function
$$
G_{\bb \theta}(\bb s)= e^{\langle \bb \theta , \bb s\rangle } \big(\P( \left |X_i|> e^{s_i}, i=1,2\right ) - \P \left ( |A_iX_i|> e^{s_i}, i=1,2\right ) \big),$$
where $A_i$ and $X_i$ are chosen to be independent of each other.

\begin{lemma}\label{lem:three}\label{lem:gtheta.implies.Mtheta} We have the following implications.
	\begin{enumerate}
		\item 
		If $\mathfrak{M}(\bb \theta)<\infty$ and $\bb \Phi(\bb \theta)<\8$, then $G_{\bb \theta} \in L^1(\R^2)$.
		\item If  $G_{\bb \theta} \in L^1(\R^2)$ and $\bb \Phi(\bb \theta)<1$, % $\mu _{\theta}$ is a strictly subprobability measure 
		%i.e. $\phi _{\theta}<1$ 
		then  $\mathfrak{M}(\bb \theta)<\infty$.% $M_{\bb \theta }$ is integrable.
	\end{enumerate}
\end{lemma}

The crucial property of $G_{\bb \theta}$ is that we can in turn conclude its integrability already if we know the  the finiteness of $\mathfrak{M}(\bb \theta')$ for $ \bb \theta' < \bb \theta$, plus some moment assumptions on $\bb A$ and $\bb B$.

\begin{lemma}\label{lem:g} For $\bb \theta \in \R^2_+$, denote $\bb \theta '=((\theta _1-1)^+,\theta _2) $, 
	$\bb \theta ''=(\theta _1,(\theta _2-1)^+) $.
	Assume $\Psi(\bb \theta )<\8 $, $\bb \Phi (\bb \theta)<\8$ and that  $\mathfrak{M}(\bb \theta')<\infty$ and $\mathfrak{M}(\bb \theta'')<\infty$.   Then $G_{\bb \theta} \in L^1(\R^2)$.  
\end{lemma}
\begin{remark}
	Observe, that Lemma \ref{lem:g} does {\em not} require $\bb \Phi (\bb \theta)\leq 1$.
\end{remark}

Now we can give the proof of Theorem \ref{thm:moments}.

\begin{proof}[Proof of Theorem \ref{thm:moments}]
	If $\theta _1=0 $ or $\theta _2=0$ then $\mathfrak{M}(\bb \theta)$ is finite by Lemma \ref{lem:two}. Hence we assume that both $\theta_1, \theta_2>0$.

	By assumption, $\bb \Phi(\bb \theta)<1$. Then, by Lemma \ref{lem:gtheta.implies.Mtheta}, $\mathfrak{M}(\bb \theta)<\infty$ follows once we prove that $G_{\bb \theta} \in L^1(\R^2)$. 
	
		Define $|\bb \theta |= \lceil \theta _1\rceil + \lceil \theta _2\rceil$. %Observe that $|\bb \theta '|=|\bb \theta ''|=|\bb \theta |-1.$ 
	We will prove the integrability of $G_{\bb \theta}$ by induction on $|\bb \theta |$. 
	
	If $\theta_1, \theta_2>0$, then $|\bb \theta |$ is at least 2, in which case  
	$\theta _1, \theta _2\leq 1$ and $G_{\bb \theta}$ is integrable by Lemma \ref{lem:one}. This is the base case.
	
	Suppose for the induction step that $|\bb \theta | = m\geq 3$. Then  either $|\bb \theta '| = |\bb \theta ''|= m-1$ or $\theta '_1=0, |\bb \theta ''|= m-1$ or $\theta '_2=0, |\bb \theta '|= m-1$.  It holds that $\bb \Phi (\bb \theta '), \bb \Phi (\bb \theta '')<1$ as well as $\Psi (\bb \theta '), \Psi (\bb \theta '')<\8 $ by Lemma \ref{lem:properties.Psi}. Further, by the induction assumption, $G_{\bb \theta '}, G_{\bb \theta ''}$ are integrable.
	Hence by Lemma \ref{lem:gtheta.implies.Mtheta}, $\mathfrak{M}(\bb \theta')< \infty$ and $\mathfrak{M}(\bb \theta'')<\8$. In conclusion, the assumptions of Lemma \ref{lem:g} are satisfied and the integrability of $G_{\bb \theta}$ follows.
\end{proof}

\subsection{Proofs of the Lemmata}

\begin{proof}[Proof of Lemma \ref{lem:properties.Psi}]
	The finiteness of $\Psi(\bb \alpha)$ is asserted by assumption \eqref{assump:mixedmoments}. We prove that $\E \big[|B_1|^{\theta_1} |A_2|^{\theta_2}\big]< \infty$ for any $\bb 0 <\bb \theta < \bb \alpha$, the remaining part of the proof is along similar lines. By assumption \eqref{assump:alpha}, $\E \big[ |A_2|^{\alpha_2}\big]<\infty$ and hence $\E \big[ |A_2|^{\theta_2} \big] < \infty$ for any $0 \le\theta_2<\alpha_2$. This allows to consider a new probability measure $\widetilde{P}$ such that for any measurable set $C$
	$$ \widetilde{\P}(C):= \frac{\E \big[ \bb 1_{C} |A_2|^{\alpha_2} \big]}{\E \big[ |A_2|^{\alpha_2} \big]}$$ 
	By assumption, $\Psi(\bb \alpha)<\8$ and hence $\widetilde{\E}[|B_1|^{\alpha_1}] < \infty$. It follows that for any $0\le \theta_1 \le \alpha_1$, $\widetilde{\E}[|B_1|^{\theta_1}] < \infty$. This gives in turn the finiteness of $\E \big[|B_1|^{\theta_1} |A_2|^{\alpha_2}\big]=\E \big[ |A_2|^{\alpha_2} \big] \cdot \widetilde{\E}[|B_1|^{\theta_1}]$.
	
	Now we use that by assumption \eqref{assump:moments}$, \E \big[|B_1|^{\alpha_1}\big]<\infty$, and hence $\E \big[|B_1|^{\theta_1}\big]<\8$ for all $0 \le \theta_1 \le \alpha_1$. Considering now 
	$$ \widehat{\P}(C):= \frac{\E \big[ \bb 1_{C} |B_1|^{\theta_1} \big]}{\E \big[ |B_1|^{\theta_1} \big]},$$
	we can deduce from the finiteness of $\widehat{\E}\big[ |A_2|^{\alpha_2}\big]$ the finiteness of $\widehat{\E}\big[ |A_2|^{\theta_2}\big]$ for any $0 \le \theta_2\le \alpha_2$, and conclude that 
	$$ \E \big[|B_1|^{\theta_1} |A_2|^{\theta_2}\big] = \widehat{\E}\big[ |A_2|^{\theta_2}\big] \cdot \E \big[ |B_1|^{\theta_1} \big] < \infty. $$
\end{proof}

\begin{proof}[Proof of Lemma \ref{lem:one}]
	Recalling \eqref{eq:perpetuity}, we have, using the subadditivity of $|\cdot|^\theta$ for $\theta<1$,
	\begin{align}
		\E \big[|X_1|^{\theta _1}|X_2|^{\theta _2} \big]& \le \E\bigg[  \Big|\sum_{k=0}^\infty \abs{A_{1,1}} \cdots \abs{A_{k-1,1}} \cdot \abs{B_{k,1}}\Big|^{\theta _1} \cdot 
		\Big|\sum_{\ell=0}^\infty \abs{A_{1,2}} \cdots \abs{A_{\ell-1,2}} \cdot \abs{B_{\ell,2}}\Big|^{\theta _2} \bigg] \notag \\
		& \le \E \bigg[ \Big( \sum_{k=0}^\infty \abs{A_{1,1}}^{\theta _1} \cdots \abs{A_{k-1,1}}^{\theta _1} \cdot \abs{B_{k,1}}^{\theta _1} \Big) \cdot 
		\Big(\sum_{\ell=0}^\infty \abs{A_{1,2}}^{\theta _2} \cdots \abs{A_{\ell-1,2}}^{\theta _2} \cdot \abs{B_{\ell,2}}^{\theta _2} \Big)\bigg] \notag \\
		& =   \sum_{k,\ell=0}^\infty \E \big[\abs{A_{1,1}}^{\theta _1} \cdots \abs{A_{k-1,1}}^{\theta _1} \cdot \abs{B_{k,1}}^{\theta _1}  \abs{A_{1,2}}^{\theta _2} \cdots \abs{A_{\ell-1,2}}^{\theta _2} \cdot \abs{B_{\ell,2}}^{\theta _2} \big] \notag \\
		& = \sum_{0 \le k < \ell < \infty} + \sum_{0 \le k = \ell <\infty} + \sum_{0\le \ell < k < \infty} \label{eq:bound.theta.less.1.one}
	\end{align}
	We bound the three sums in \eqref{eq:bound.theta.less.1.one} separately, starting with the first one. It equals
	\begin{align*}
		&\sum _{0 \le k < \ell < \infty}\E \left[ \left(\Pi _{i=1}^{k-1}   \abs{A_{i,1}}^{\theta _1}\abs{A_{i,2}}^{\theta _2} \right) |B_{k,1}|^{\theta _1}\abs{A_{k,2}}^{\theta _2}\left (\Pi _{i=k+1}^{\ell-1} \abs{A_{k,2}}^{\theta _2}\right )|B_{\ell,2}|^{\theta _2}\right]\\
		=&\sum _{0 \le k < \ell < \infty}  \Big(\Pi _{i=1}^{k-1}   \E\big[\abs{A_{i,1}}^{\theta _1}\abs{A_{i,2}}^{\theta _2}\big] \Big) \E \big[|B_{k,1}|^{\theta _1}\abs{A_{k,2}}^{\theta _2}\big] \Big(\Pi _{i=k+1}^{\ell-1} \E\big[\abs{A_{k,2}}^{\theta _2}\big]\Big) \E \big[|B_{\ell,2}|^{\theta _2}\big]\\
		\leq& \sum _{0 \le k < \ell < \infty} \bb \Phi (\bb \theta)^{k-1} \cdot \Psi(\bb \theta) \cdot \bb\Phi (0,\theta _2 )^{\ell-k-2} \cdot \Psi(0,\theta_2) 
		\leq  \frac{\Psi(\bb \theta) \Psi(0,\theta_2)}{\bb \Phi(\bb \theta) \bb \Phi(0, \theta_2)} \frac{1}{1-\bb \Phi(\bb \theta)} \frac{1}{1-\bb \Phi(0,\theta_2)} < \infty
	\end{align*}
	Note that $\bb \Phi(0,\theta_2)=\E \big[ |A_2|^{\theta_2}<1$ for every $0<\theta_2<\alpha_2$ due to convexity.
		The third sum in \eqref{eq:bound.theta.less.1.one} is estimated in the same way, using $\Psi(\theta_1,0)<\infty$ and $\bb \Phi(\theta_1,0)<1$ instead. Turning to the second sum in \eqref{eq:bound.theta.less.1.one}, it equals
	\begin{align*}
		&\sum _{0 \le k  <\infty} \left ( \Pi _{i=1}^{k-1} \E \left[  \abs{A_{i,1}}^{\theta _1}\abs{A_{i,2}}^{\theta _2}\right]\right ) \E \big[ |B_{k,1}|^{\theta _1}|B_{k,2}|^{\theta _2} \big] \leq   \sum _{0 \le k  <\infty} \bb \Phi (\bb \theta) ^{k-1} \cdot \Psi(\bb \theta) <\8
	\end{align*}
	This finishes the proof of the Lemma.
\end{proof}

\begin{proof}[Proof of Lemma \ref{lem:three}]
		Let us abbreviate
	$$	M(\bb s)=  \P \left (|X_i|> e^{s_i}, i=1,2\right ), \qquad M_{\bb \theta}(\bb s)= e^{\langle \bb \theta , \bb s\rangle}  \P \left (|X_i|> e^{s_i}, i=1,2\right ) = e^{\langle \bb \theta , \bb s\rangle} M(\bb s)$$
	By a standard application of Fubini's theorem, we have
	\begin{align*}
		\int _{\R ^2} M_{\bb \theta} (\bb s)\ d\bb s&=\int _{\R ^2}e^{\langle \bb  \theta , \bb s\rangle }\E \big[ \Pi _{i=1}^2\Ind {\{|X_i|> e^{s_i}\}} \big] d\bb s =\E \big[ \int _{\R ^2} \Pi _{i=1}^2e^{\theta _i s_i }\Ind {\{|X_i|> e^{s_i}\}} d \bb s \big]\\
		&= (\theta _1\theta _2)^{-1}\E \Big[ |X_1|^{\theta _1}|X_2|^{\theta _2} \Big] \\
		&= (\theta _1\theta _2)^{-1}  \mathfrak{M}(\bb \theta) 
	\end{align*}
	 Thus we have the equivalence $\mathfrak{M}(\bb \theta)<\infty \Leftrightarrow M_{\bb \theta} \in L^1(\R^2)$.
	
	Writing $\zeta$ for the law of $\bb W=(\log |A_1|, \log |A_2|)$ and $\zeta_{\bb \theta}(d \bb w):=e^{\skalar{\bb \theta, \bb w}}(d \bb w)$ we claim that 
	\begin{equation}\label{eq:lem:three:a}
	G_{\bb \theta}(\mathbf{s})=M_{\bb \theta}(\mathbf{s})-M_{\bb \theta} * \zeta_{\bb \theta} (\bb s) = M_{\bb \theta}(\mathbf{s})-\E \big[ e^{\skalar{\bb \theta, \bb W}} M_{\bb \theta}(\mathbf{s}-\bb W) \big]	
	\end{equation}
	Then, if $M_{\bb \theta }$ is integrable and $\zeta_{\bb \theta}(\R^2)=\bb \Phi(\bb \theta)$ is finite then $G_{\bb \theta }$ is integrable; and the first assertion follows.
	But
	\begin{align*}
	M_{\bb \theta} * \zeta_{\bb \theta} (\bb s) =&~ \int_{\R^2} e^{\skalar{\bb \theta, \bb s-\bb w}} \P ( \left |X_i|> e^{s_i-w_i}, i=1,2\right ) e^{\skalar{\bb \theta, \bb w}} \zeta(d \bb w) \\
	=&~ e^{\skalar{\bb \theta, \bb s}} \int_{\R^2} \P ( \left |e^{w_i}X_i|> e^{s_i}, i=1,2\right ) \zeta(d \bb w) = e^{\skalar{\bb \theta, \bb s}}  \P ( \left |A_iX_i|> e^{s_i}, i=1,2\right ) 
	\end{align*}
	with $(|A_1|, |A_2|)$ being independent of $(|X_1|, |X_2|)$. This proves \eqref{eq:lem:three:a}.
Therefore, if $M_{\bb \theta }$ is integrable and $\mu _{\theta}$ is finite then $G_{\theta }$ is integrable. This proves the first assertion.

Turning to the second assertion, suppose now that $G_{\bb \theta }$ is integrable and $\zeta _{\bb \theta}(\R^2)=\bb \Phi(\bb \theta)<1$, that is, $\zeta_{\bb \theta}$ is a strictly subprobability measure.  
Then $\mathbb{V}_{\bb \theta}= \sum _{n=1}^{\8}   \zeta _{\bb \theta}^n$ is a finite measure and $  G_{\bb \theta}*\mathbb{V} _{\bb \theta}$ is integrable. The assertion  follows if we can prove that
\begin{equation}\label{eq:potential}
	M_{\bb \theta}=G_{\bb \theta}*\mathbb{V} _{\bb \theta}.\end{equation}

Indeed, on one hand
$$
\sum _{n=1}^N \left (M_{\bb \theta}*\zeta_{\bb \theta} ^{n-1}(\mathbf{s}) - M_{\bb \theta}*\zeta_{\bb \theta} ^n(\mathbf{s})\right )= M_{\bb \theta}(\mathbf{s}) - M_{\bb \theta}*\zeta_{\bb \theta} ^N(\mathbf{s})$$
and on the other
$$
\sum _{n=1}^N \left (M_{\bb \theta}*\zeta_{\bb \theta} ^{n-1}(\mathbf{s}) - M_{\bb \theta}*\zeta_{\bb \theta} ^n(\mathbf{s})\right ) = \sum _{n=1}^N \left (M_{\bb \theta}* - M_{\bb \theta}*\zeta_{\bb \theta} \right )* \zeta_{\bb \theta}^{n-1}(\bb s) = \sum _{n=0}^{N-1}  G_{\bb \theta} * \zeta_{\bb \theta} ^n(\mathbf{s}).$$
But
$$
\lim_{N \to \infty} M_{\bb \theta}*\zeta_{\bb \theta} ^N(\mathbf{s})= \lim_{N \to \infty} e^{\skalar{\bb \theta, \bb s}} \E \big[ M(\bb s - \sum_{i=1}^N \bb W_i)] =0 
$$
for every $\bb s \in \R^2$, using dominated convergence: $M$ is bounded with $\lim_{|\bb r| \to \infty} M(\bb r)=0$, and  $\lim_{N \to \infty} (\sum_{i=1}^N \bb W_i)_j=-\infty$ a.s. for both $j \in {1,2}$ as a consequence of \eqref{eq:logAi}.
Thus we obtain \eqref{eq:potential}.
\end{proof}

\begin{proof}[Proof of Lemma \ref{lem:g}] We use the identity $A_iX_i +B_i\eqdist X_i$ inside the definition of $G_{\bb \theta}$ and compute
	\begin{align*}
		\int _{\R ^2} |G_{\bb \theta}(\bb s)|\ d\bb s
		=&~\int _{\R ^2}e^{\langle \bb \theta , \bb s\rangle }|\P \left ( |A_iX_i+B_i|> e^{s_i}, i=1,2\right ) - \P \left (|A_iX_i|> e^{s_i}, i=1,2\right )| d\bb s\\
		\leq&~  \int _{\R ^2}e^{\langle \bb \theta , \bb s\rangle }|\P \left(  |A_iX_i+B_i|> e^{s_i}, i=1,2\right )
		-\P( \left |A_1X_1+B_1|> e^{s_1}, |A_2X_2|>e^{s_2} \right )| d\bb s \\  
		&~+\int _{\R ^2}e^{\langle\bb \theta , \bb s\rangle }|\P( \left |A_1X_1+B_1|> e^{s_1}, |A_2X_2|>e^{s_2} \right )- \P \left (|A_iX_i|> e^{s_i}, i=1,2\right )| d\bb s \\
		=&~\int _{\R ^2}e^{\langle \bb \theta , \bb s\rangle }
		\P  (|A_1X_1+B_1|> e^{s_1}, 
		|A_2X_2|< e^{s_2}\leq |A_2X_2+B_2|) d\bb s\\
		&~+\int _{\R ^2}e^{\langle \bb \theta , \bb s\rangle }
		\P  (|A_1X_1+B_1|> e^{s_1}, 
		|A_2X_2+B_2|<e^{s_2}\leq |A_2X_2|) d\bb s\\
		&~+ \int _{\R ^2}e^{\langle \bb \theta , \bb s\rangle }
		\P  (|A_1X_1|< e^{s_1}\leq |A_1X_1+B_1|,
		|A_2X_2|> e^{s_2}) d\bb s\\
		&~+ \int _{\R ^2}e^{\langle \bb \theta , \bb s\rangle }
		\P  (|A_1X_1+B_1|< e^{s_1}\leq |A_1X_1|,
		|A_2X_2|> e^{s_2}) d\bb s
	\end{align*}
	Applying Fubini we can further equate this to 
	\begin{align}
		&~	 \E \left [\int _{\R ^2}e^{\langle \bb \theta , \bb s\rangle }
		\Ind {\{ |A_1X_1+B_1|> e^{s_1} \}}\Ind {\{
			\min (|A_2X_2|, |A_2X_2+B_2|)< e^{s_2}\leq \max (|A_2X_2|,|A_2X_2+B_2|) \}} d\bb s\right ] \notag \\
		&~	+ \E \left [\int _{\R ^2}e^{\langle \bb \theta , \bb s\rangle }
		\Ind {\{ \min (|A_1X_1|, |A_1X_1+B_1|)< e^{s_1}\leq \max (|A_1X_1|,|A_1X_1+B_1|) \}}\Ind {\{ |A_2X_2|> e^{s_2}, \}} d\bb s \right ]\notag\\
		=&~ \E \left [\left ( \int _{-\8}^{\log |A_1X_1+B_1|}e^{\theta _1 s_1 }\ ds_1\right ) 
		\left |\int _{\log |A_2X_2|}^{\log |A_2X_2+B_2|}e^{\theta _2 s_2 } \ ds_2\right |\right ]\notag \\
		&~+\E \left [ \left ( \int _{-\8}^{\log |A_2X_2|}e^{\theta _2 s_2 }\ ds_2\right ) 
		\left |\int _{\log |A_1X_1|}^{\log |A_1X_1+B_1|}e^{\theta _1 s_1 } \ ds_1\right |\right ] \notag\\
		\leq&~ (\theta _1\theta _2)^{-1} \E \left [|A_1X_1+B_1|^{\theta _1 } 
		\left | |A_2X_2+B_2|^{\theta _2 }-|A_2X_2|^{\theta _2}\right | \right ] \notag \\
		&~+ (\theta _1\theta _2)^{-1} \E \left [|A_2X_2|^{\theta _2} 
		\left | |A_1X_1+B_1|^{\theta _1 }-|A_1X_1|^{\theta _1}\right |\right ] \label{eq:bound.EG}
	\end{align}
	Depending on whether $\theta_k \le 1$ or $\theta_k>1$, we proceed slightly differently.
	
	\textbf{Case $\theta_1\le 1, \theta_2\le 1$}. Using subadditivy, We estimate the sum of the expectations in \eqref{eq:bound.EG} by
	\begin{align*}
		&~\E [|A_1X_1|^{\theta _1 } 
		| B_2|^{\theta _2 } + |A_2X_2|^{\theta _2}|B_1|^{\theta _1}+|B_1|^{\theta _1}| B_2|^{\theta _2 }]\\
		=&~\E [|A_1^{\theta _1 }| 
		| B_2|^{\theta _2 }]\E [|X_1|^{\theta _1 }] + \E [|A_2^{\theta _2}||B_1|^{\theta _1}]\E [|X_2|^{\theta _2}]+\E [|B_1|^{\theta _1}| B_2|^{\theta _2 }] \\
		\le&~ \Psi(\bb \theta) \mathfrak{M}(\theta_1,0) + \Psi(\bb \theta) \mathfrak{M}(0,\theta_2) + \Psi(\bb \theta)
	\end{align*}
	and the latter is finite by assumption.
	
	\textbf{Case $\theta_1 \wedge \theta_2>1$}.
	If $\theta_k \leq 1$, we use the inequality$|\bb x + \bb y|^{\theta_k} \le 2^{\theta_k-1}(|\bb x|^{\theta_k} + |\bb y|^{\theta_k})$, the inequality $|\bb x|^{\theta_k-1}|\bb y| \le |\bb x|^{\theta_k} + |\bb y|^{\theta_k}$ and the inequality $\abs{| \bb x|^{\theta_k}-|\bb y|^{\theta_k}} \le \max\{ |\bb x|, |\bb y|\}^{\theta_k-1} |\bb x - \bb y|$ to obtain
	\begin{align*}
		| |A_kX_k+B_k|^{\theta _k }-|A_kX_k|^{\theta _k}|&\leq \max \left (|A_kX_k+B_k|^{\theta _k -1},|A_kX_k|^{\theta _k-1}\right )|B_k|\\
		&\leq 
		C (|A_kX_k|^{\theta _k-1}+|B_k|^{\theta _k-1})|B_k| \leq C(|A_kX_k|^{\theta _k-1}|B_k|+|B_k|^{\theta _k}) \\
		&\leq C (|A_k|^{\theta_k} |X_k|^{\theta _k-1} + |B_k|^{\theta_k} |X_k|^{\theta_k-1}+|B_k|^{\theta _k}).
	\end{align*}
With this, we provide below an estimate for the first expectation in \eqref{eq:bound.EG} in the case where both $\theta_1, \theta_2>1$. The remaining expectation as well as the other cases can be treated then in a similar way.
\begin{align*}
	&~ \E \left [|A_1X_1+B_1|^{\theta _1 } 
	\left | |A_2X_2+B_2|^{\theta _2 }-|A_2X_2|^{\theta _2}\right | \right ]  \\
	\leq&~ C \E \big[ (|A_1 X_1|^{\theta_1} + |B_1|^{\theta_1}) (|A_2|^{\theta_2} |X_2|^{\theta _2-1} + |B_2|^{\theta_2} |X_2|^{\theta_2-1}+|B_2|^{\theta _2})\big] \\
	\leq&~ C \Big( \E \big[ |A_1|^{\theta_1} |A_2|^{\theta_2} \big]  \E \big[ |X_1|^{\theta_1} |X_2|^{\theta_2-1} \big] + \E \big[ |A_1|^{\theta_1} |B_2|^{\theta_2} \big] \E \big[ |X_1|^{\theta_1} |X_2|^{\theta_2-1} \big]  \\
	&~ +\E \big[ |A_1|^{\theta_1} |B_2|^{\theta_2} \big] \E \big[ |X_1|^{\theta_1}  \big] + \E \big[ |B_1|^{\theta_1} |A_2|^{\theta_2} \big] \E \big[ |X_2|^{\theta_2-1} \big] + \E \big[ |B_1|^{\theta_1} |B_2|^{\theta_2} \big] \big( \E \big[  |X_2|^{\theta_2-1} \big] + 1 \big) 	\Big) \\
	\leq &~ C \Psi(\bb \theta) \big( \mathfrak{M}(\theta_1,\theta_2-1)  + \mathfrak{M}(\theta_1,0) + \mathfrak{M}(0,\theta_2-1) + 1\big)
\end{align*}
This last bound is finite by Lemma \ref{lem:two} and the assumptions.
\end{proof}

\subsection{The Case of Blocks}
For the general case of blocks, we have the following version of Theorem \ref{thm:moments}. Recall from \eqref{eq:psimoments} the definition of $\psi$, if necessary.

\begin{theorem}\label{thm:moments:blocks} Consider $\bb \xi \in [0,1)^2$ with $\phi(\bb \xi)<1$ and assume $\psi(\bb \xi)<\infty$.
	Then
	\begin{equation*}
		\E \big[\norma{\bb X^{(1)}}^{\xi_1} \norma{\bb X^{(2)}}^{\xi_2} \big] <\8 .
	\end{equation*}
\end{theorem}

\begin{proof}
	Recalling the definition of $\norma{\cdot}$, we have for $I=\{1,3, \dots, d_1+1\}$, $J=\{2,d_1+2, \dots, d_1+d_2\}$
	\begin{align*}
		\norma{\bb X^{(1)}}^{\xi_1} \norma{\bb X^{(2)}}^{\xi_2}&= \big( \sup_{i \in I} |\bb X_i|^{\alpha_i \xi_1} \big) \big( \sup_{j \in J} |\bb X_j|^{\alpha_j \xi_2} \big) \le \big( \sum_{i \in I} |\bb X_i|^{\alpha_i \xi_1} \big) \big( \sum_{j \in J} |\bb X_j|^{\alpha_j \xi_2} \big) \\
		&\le \sum_{i \in I}\sum_{j \in J} |\bb X_i|^{\alpha_i \xi_1} |\bb X_j|^{\alpha_j \xi_2}
	\end{align*}
	Hence it suffices to prove finiteness for pairs of coordinates from different blocks, which follows from Theorem \ref{thm:moments} applied on coordinates $i,j$ with $\theta=(\xi_1 \alpha_i, \xi_2 \alpha_j)$. Note that $\Psi(\bb \theta) \le \psi(\bb \xi)$ and that, by \eqref{eq:equivalence.relation}, $|A_i|^{\theta_1}=(|A_i|^{\alpha_i})^{\xi_1}=|A_1|^{\alpha_1 \xi_1}$ and $|A_j|^{\alpha_j \xi_2}=(|A_j|^{\alpha_j})^{\xi_2}=|A_2|^{\alpha_2 \xi_2}$ for $i \in I$, $j \in J$.
\end{proof}

We also note the following "block" version of Lemma \ref{lem:properties.Psi}.

\begin{lemma}\label{lem:properties.Psi.blocks}
	Assuming \eqref{assump:alpha}, \eqref{assump:moments} and \eqref{eq:mixedblocks}, it holds that $\psi(\bb \xi)<\infty$ for all $\bb \xi \in [0,1]^2$.
\end{lemma}

\begin{proof}
	Assumption \eqref{eq:mixedblocks} implies that $\psi(1,1)<\infty$, while \eqref{assump:alpha} and \eqref{assump:moments} imply that $\psi(0,1)<\infty$ and $\psi(1,0)<\infty$. Obviously, $\psi(0,0)<\infty$ as well. Given this data, we may now proceed as in the proof of Lemma \ref{lem:properties.Psi} to show that $\psi(\bb \xi)<\infty$ for all $\bb \xi \in [0,1]^2$.
\end{proof}

\section{Two-Dimensional Implicit Renewal Theory}\label{sect:implicit.renewal}
In this section we will provide the proof of Theorem \ref{thm:main:blocks}, except for showing that $\Lambda_2$ is nonzero which will be done in Section \ref{sect:positivity}.
Our main tool is a renewal theorem on $\R^2 \times K$, a proof of which (extending the renewal theorem on $\R^2$) is given in the final Section \ref{sect:2dRenewal.with.K}. Recall that the value of $\bb \xi^*$ has been fixed in Theorem \ref{thm:main:blocks}.

The proof of Theorem \ref{thm:main:blocks} is based on implicit renewal theory (see \cite{Goldie1991}), namely the fact (proved in Lemma \ref{lem:implicit.renewal} below) that for any $k \in K$ and $f \in \bb H^{\epsilon}(\R^d \setminus [\mathsf{blocks}])$ (the value of $\epsilon$ is fixed in Proposition \ref{prop:gdri}),
$$ e^{\skalar{\bb \xi^*, \bb s}} \E \big[ f( e^{-\bb s/\bb  \alpha} k\bb X) \big]= g_{\bb \xi^*} * \wt{\mathbb{U}}_{\bb \xi^*}(\bb s,k) ,$$
where $\wt{\mathbb{U}}_{\bb \xi^*}$ is the renewal measure with respect to  $\wt{\mu}_{\bb \xi^*}=e^{\skalar{\bb \xi^*, \bb s}}\wt \mu $, (recall \eqref{eq:decomp}), and 
\begin{equation*}
	 g_{\bb \xi^* }(\mathbf{s},k ):=e^{\langle \bb \xi^* ,\mathbf{s}\rangle }(\E [f(e^{-\mathbf{s} / \bb \a}k(\bfA\bfX+\bfB))]-\E [f(e^{-\mathbf{s} / \bb \a} k\bfA \bfX)] ).
\end{equation*}
In the case of blocks, this has to be understood as
\begin{align*}
	g_{\bb \xi^* }(\bb s,k )= &e^{\langle \bb \xi^* ,\mathbf{s}\rangle }\bigg(\E \Big[f\Big(e^{-s_1 / \bb \a^{(1)}}k^{(1)}(\bfA^{(1)}\bfX^{(1)}+\bfB^{(1)}),\, e^{-s_2 / \bb \a^{(2)}}k^{(2)}(\bfA^{(2)}\bfX^{(2)}+\bfB^{(2)})\Big) \Big]  \\
		&~~~-\E \Big[f\Big(e^{-s_1 / \bb \a^{(1)}}k^{(1)}(\bfA^{(1)}\bfX^{(1)}),\, e^{-s_2 / \bb \a^{(2)}} k^{(2)}(\bfA^{(2)}\bfX^{(2)})\Big) \Big] \bigg).
\end{align*}
Before giving the proof of Theorem \ref{thm:main:blocks}, we need some preparation, which we record in the following proposition.
Given a function $f$ on $\R^{d_1+d_2}$, we denote
\begin{equation}
		f_{M,1}(\bb x)=f(M^{-1\slash \bb \alpha^{(1)} } \bb x^{(1)}, \bb x^{(2)}) \text{ or } 
f_{1,M}(\bb x)=f(\bb x^{(1)},M^{-1\slash \bb \alpha^{(2)} } \bb x^{(2)}) \label{eq:fM}
	\end{equation} and use notations $g_{\xi^*,M,1}$ and $g_{\xi^*,1,M}$, respectively. Finally, recall
	\begin{equation*}
		\psi (\bb \xi)=\E\big[ |A_1|^{\xi_1\alpha _1} \norma{\bfB ^{(2)}}^{\xi_2}\big] +\E \big[ \norma{\bfB ^{(1)}}^{\xi_1}\ |A_2|^{\xi_2\alpha _2}\big] + \E \big[ \norma{\bfB ^{(1)}}^{\xi_1}\norma{\bfB ^{(2)}}^{\xi_2}\big]
		+ \phi(\bb \xi)
		%\E \Pi _{j\in I}|\bfB ^{(j)}|^{\xi^* _j}\ \Pi _{j\notin I}a_j^{\xi^* _j}\ \mbox{for}\ I\subset 
		%\{ 1,2\}.
	\end{equation*}
%For \eqref{eq:dri} we need to assume that $\wt \psi (\xi^*)<\8 $ for $\xi^* $ such that $\| \xi^* -\xi^* _0\| <\d $ for some positive $\d $.

\begin{proposition}\label{prop:gdri}
	Suppose that there is $\bb \xi^* \in (0,1)^2$ with $\phi (\bb \xi ^*)= 1$.  Assume further that there is $\d' >0$ such that $ \psi (\bb \xi) <\8$ 
	for every  $\bb \xi \in \{\bb \xi^* \pm \delta' \bb e_1, \bb \xi^* \pm \delta' \bb e_2\} $. Then for any $0<\epsilon< \min\{\xi_1^*, \xi_2^*\}$, the following holds:
	\begin{enumerate}
		\item For every $f \in \bb H^{\epsilon}(\R^d \setminus [\mathsf{blocks}])$,
		\begin{equation}\label{eq:dri}
			\| g_{\bb \xi ^*}\| _{\dri}= \sum _{k\in K}\sum _{\bfm}\sup _{\mathbf{s}\in F_{\bfm} }(1+\| \mathbf{s}\| )|g _{\bb \xi^*}(\mathbf{s},k)|<\8. 
			\end{equation}
		\item There is $C>0$ such that for every $f \in \bb H^{\epsilon}(\R^d \setminus [\mathsf{blocks}])$
		\begin{equation}\label{eq:boundLimitDRI}
			\limsup_{t \to \infty} (\log t)^{1/2} t^{\xi_1^* + \xi_2^*} \E \Big[ f\big( t^{-1/\bb \alpha} \bb X\big)\Big] ~ \le ~ C \| g_{\bb \xi ^*}\| _{\dri}
		\end{equation} 
		and 
		there is $c=c(f)>0$ such that for every $M>1$,
		\begin{align}
			\limsup_{t \to \infty} (\log t)^{1/2} t^{\xi_1^* + \xi_2^*} \E \Big[ f_{M,1}\big( t^{-1/\bb \alpha} \bb X\big)\Big] ~ \le ~ c M^{-\xi_1^*}, \notag \\
			\limsup_{t \to \infty} (\log t)^{1/2} t^{\xi_1^* + \xi_2^*} \E \Big[ f_{1,M}\big( t^{-1/\bb \alpha} \bb X\big)\Big] ~ \le ~ c M^{-\xi_2^*} \label{eq:boundLimitDRIwithM}
		\end{align}
	\end{enumerate}
\end{proposition}

Note that by Lemma \ref{lem:properties.Psi.blocks}, $\Psi(\bb \xi)<\infty$ for all $\bb \xi \in [0,1]^2$.
Since $\bb \xi^* \in (0,1)^2$, we can always find $\delta'$ as required in Proposition \ref{prop:gdri}.
The proof of the Proposition will be based on several Lemmata.

%\eqref{eq:dri} with
%$g _{\bb \xi  }$ for $\| \bb \xi -\bb \xi ^*\| <\d $. So we are going to prove
 \begin{lemma}\label{lem:dri}
Under the assumptions of Proposition \ref{prop:gdri}, choose $\d < \min \{\epsilon , \xi ^*_1-\epsilon , \xi ^*_2-\epsilon, \delta' \}.$
Then for every $f \in \bb H^{\epsilon}(\R^d \setminus [\mathsf{blocks}])$ and every $\bb \xi \in \{\bb \xi^* \pm \delta \bb e_1, \bb \xi^* \pm \delta \bb e_2\}$, it holds that
	\begin{equation}\label{eq:driLem}
		\sum_{k \in K} \sum _{\bfm}\sup _{\mathbf{s}\in F_{\bfm} }|g _{\bb \xi}(\mathbf{s},k)|<\8 
		%\quad \mbox{is d.r.i}
	\end{equation} 
		 Moreover, for every $f \in \bb H^{\epsilon}(\R^d \setminus [\mathsf{blocks}])$	there is $C=C(f)$ such that for every $\bb \xi \in \{\bb \xi^* \pm \delta \bb e_1, \bb \xi^* \pm \delta \bb e_2\}$ 	
		\begin{equation}\label{eq:dRiwithM}
			\sum_{k \in K} \sum _{\bfm}\sup _{\mathbf{s}\in F_{\bfm} }|g _{\bb \xi,M,1}(\mathbf{s},k)|<C
			M^{-\xi _1 }, \quad  \sum_{k \in K} \sum _{\bfm}\sup _{\mathbf{s}\in F_{\bfm} }|g _{\bb \xi,1,M}(\mathbf{s},k)|< CM^{-\xi _2}.
						%\sum _{\bfm}\sup _{\mathbf{s}\in F_{\bfm} }|g _{\bb \xi}(\mathbf{s},k)|
	\end{equation}
\end{lemma}

\begin{proof}
	We start by making use of the assumption that $f \in \bb H^{\epsilon}$. Namely, there are $\eta_1, \eta_2$ such that
	\begin{align*}
		f(e^{-\mathbf{s} / \bb \a}k(\bfA \bfX+\bfB))\neq 0 \qquad &\Rightarrow \qquad \norma{e^{-s_j / \bb \a^{(j)}}(\bfA\bfX+\bfB)^{(j)}}>\eta _j \text{ for both } j \in \{1,2\}; \\
		 f(e^{-\mathbf{s} / \bb \a}k\bfA \bfX)\neq 0 \qquad &\Rightarrow \qquad \norma{e^{-s_j / \bb \a^{(j)}}(\bfA\bfX)^{(j)}}>\eta _j \text{ for both } j \in \{1,2\}.
	\end{align*}
	Here and below we use that $\norma{k\bb x} = \norma{\bb x}$ for all $k \in K$.
	Using that $\norma{\bb x + \bb y} \le c_\alpha (\norma{\bb x} + \norma{\bb y}),$
	we obtain
	\begin{equation*}
		\Big( f(e^{-\mathbf{s} / \bb \a}(\bfA \bfX+\bfB)) - f(e^{-\mathbf{s} / \bb \a}\bfA \bfX) \Big) \neq 0 \ \Rightarrow \ \norma{(\bfA\bfX)^{(j)}}+\norma{\bfB^{(j)}}>\eta _jc^{-1}_{\a}e^{s_j} \text{ for both } j \in \{1,2\}.
	\end{equation*}
%	
%	Moreover, $f(e^{-\mathbf{s} / \bb \a}(\bfA \bfX+\bfB))\neq 0$ implies that $\norma{e^{-s_j / \bb \a^{(j)}}(\bfA\bfX+\bfB)^{(j)}}>\eta _j$, $j=1,2$. Similarly, if $ f(e^{-\mathbf{s} / \bb \a}\bfA \bfX)\neq 0$ then $\norma{e^{-s_j / \bb \a^{(j)}}(\bfA\bfX)^{(j)}}>\eta _j$, $j=1,2$. 
%	In any case $|(\bfA\bfX)^{(j)}|_{\a}+|\bfB^{(j)}|_{\a}>\eta _jc^{-1}_{\a}e^{s_j}$, $j=1,2$.
Finally, using the $\epsilon$-$\bb \alpha$-Hölder continuity,
\begin{align*}
	&~ \abs{ f(e^{-\mathbf{s} / \bb \a}k(\bfA \bfX+\bfB)) - f(e^{-\mathbf{s} / \bb \a}k\bfA \bfX)} \le C_f \norma{e^{-\bb s / \bb \alpha} k\bb B}^\epsilon \\
	\le &~C_f c_\alpha^\epsilon \big( \norma{e^{-s_1/\bb \alpha^{(1)}} \bb B^{(1)}}^\epsilon + \norma{e^{-s_2/\bb \alpha^{(2)}} \bb B^{(2)}}^\epsilon \big)
\end{align*}
	Therefore, recalling the homogenity property \eqref{eq:homogenity},%$\|\bb \xi -\bb \xi ^*\| <\d ^* $ we have
	\begin{equation*}
		\abs{g _{\bb \xi}(\mathbf{s},k)}
		\leq Ce^{\langle \bb \xi ,\mathbf{s}\rangle }\E \left [\left (e^{-s_1\epsilon}\norma{\bfB^{(1)}}^{\epsilon}+e^{-s_2\epsilon}\norma{\bfB^{(2)}}^{\epsilon} \right )\Pi _{j=1}^2\Ind {\{\norma{(\bfA\bfX)^{(j)}}+\norma{\bfB^{(j)}}>\eta _jc^{-1}_{\a}e^{s_j} \}}\right ] .
	\end{equation*}
	%For $\bfm =(m_1,m_2)\in \Z ^2$, let $F_{\bfm}=\{ \bb s\in \R^2: m %_j\leq s_j\leq m_{j}+1, j=1,2\}$. 
	Next, we  consider any fixed $k \in K$ and take the sum over all $F_{\mathbf m}$. Using the Fubini theorem for nonnegative functions, we may take the summation below inside  the expectation. The indicators will then restrict the domain of summation, as follows.
	\begin{align}
		&~\sum _{\bfm\in \Z ^2}\sup _{\mathbf{s}\in F_{\bfm}}|g _{\bb \xi}(\bb s,k)| \notag \\
		\leq&~ C\sum _{\bfm\in \Z ^2} e^{\langle \bb \xi ,\bfm \rangle }\E \left [e^{-m_1\epsilon} \norma{\bfB^{(1)}}^{\epsilon}\Pi _{j=1}^2\Ind {\{\norma{(\bfA\bfX)^{(j)}}+\norma{\bfB^{(j)}}>\eta _jc^{-1}_{\a}e^{s_j} \}}\right ] \notag \\
		&~+C\sum _{\bfm\in \Z ^2} e^{\langle \bb \xi ,\bfm \rangle }\E \left [e^{-m_2\epsilon}\norma{\bfB^{(2)}}^{\epsilon}\Pi _{j=1}^2\Ind {\{\norma{(\bfA\bfX)^{(j)}}+\norma{\bfB^{(j)}}>\eta _jc^{-1}_{\a}e^{s_j} \}}\right ] \notag \\
		\leq&~ C\E \Big[ \norma{\bfB^{(1)}}^{\epsilon} \sum _{\bf m\leq \bar{\bf m}}e^{(\xi _1-\epsilon)m_1}	e^{\xi _2m_2} \Big] +
		C\E \Big[ \norma{\bfB^{(2)}}^{\epsilon} \sum _{\bf m\leq \bar{\bf m}} e^{\xi _1m_1}	e^{(\xi _2-\epsilon)m_2} \Big], \label{eq:dRiEstimate1}
	\end{align}
	where the vector $\bar{\bf m}$ is given by %$j=2$ if $l=1$ and $j=1$ if $l=2$, and
	\begin{equation}\label{eq:m}
		\bar m_j=\lfloor \log |(\bfA\bfX)^{(j)}|_{\a}+|\bfB^{(j)}|_{\a}-\log \eta _j +\log c_{\a}\rfloor\  \mbox{for}\ j=1,2.
	\end{equation}
	Considering the first sum in \eqref{eq:dRiEstimate1}, we reverse signs to obtain geometric series and subsequently plug in the definition of $\bar{\bf m}$, as follows.
	%\begin{align}
	\begin{align}
		&\sum _{-m_1 \ge (-\bar{m}_1)}  e^{-(\xi _1-\epsilon)(-m_1)}\, \sum _{-m_2 \ge (-\bar{m}_2)}e^{-\xi _2(-m_2)}~\leq~ e^{\bar m_1(\xi _1-\epsilon)}\left (1-e^{-\xi _1+\epsilon  }\right )^{-1}e^{\bar m_2\xi _2}\left (1-e^{-\xi _2 }\right )^{-1} \notag \\
		\leq~ & C(\bb \xi, \epsilon )\, \eta _1^{-\xi _1+\epsilon}\eta _2^{-\xi _2}\big(\norma{(\bfA\bfX)^{(1)}}+\norma{\bfB^{(1)}} \big) ^{\xi _1-\epsilon  }
		\big(\norma{(\bfA\bfX)^{(2)}}+\norma{\bfB^{(2)}}\big) ^{\xi _2}\label{eq:eta}
		%\left (|(\bfA\bfX)^{(j)}|_{\a}+|\bfB^{(j)}|_{\a}\right  ) ^{\xi _j }.
	\end{align}
	%\end{align}
	The second sum in \eqref{eq:dRiEstimate1} is estimated in the same way. Hence, there is $C=C(\bb \xi, \epsilon)$ such that for any $k \in K$,
	\begin{align}
		& \sum _{\bfm\in \Z ^2}\sup _{\mathbf{s}\in F_{\bfm}}|g _{\bb \xi }(\mathbf{s},k)| \notag \\
		\leq~& C \eta _1^{-\xi _1+\epsilon}\eta _2^{-\xi _2}
		\E \Big[ \norma{\bfB^{(1)}}^{\epsilon} \big(\norma{(\bfA\bfX)^{(1)}}+\norma{\bfB^{(1)}}\big) ^{\xi _1-\epsilon  } \big(\norma{(\bfA\bfX)^{(2)}}+\norma{\bfB^{(2)}} \big) ^{\xi _2} \Big] \notag \\
		&+C \eta _1^{-\xi _1}\eta _2^{-\xi _2+\epsilon }\E \Big[ \norma{\bfB^{(2)}}^{\epsilon} \big( \norma{(\bfA\bfX)^{(1)}}+\norma{\bfB^{(1)}} \big) ^{\xi _1  } \big(\norma{(\bfA\bfX)^{(2)}}+\norma{\bfB^{(2)}}  \big) ^{\xi _2-\epsilon } \Big] ~=:~D \label{eq:estimatedRi2}
	\end{align}
	At this point, it is useful to collect some information regarding \eqref{eq:dRiwithM}. Following the same steps as above, we obtain for $f_{M,1}(\bb x)=f(M^{-1\slash \bb \alpha } \bb x^{(1)}, \bb x^{(2)})$ the following upper bound in \eqref{eq:dRiEstimate1}: 
	\begin{equation*}
	\E \left [M^{-\epsilon} \norma{\bfB^{(1)}}^{\epsilon} \sum _{ \bf m\leq \bar{ \bf m}(M)}e^{(\xi _1-\epsilon)m_1}	e^{\xi _2m_2} \right ]+
	\E \left [ \norma{\bfB^{(2)}}^{\epsilon} \sum _{\bf m \leq \bar{ \bf m} (M)}e^{\xi _1m_1}	e^{(\xi _2-\epsilon)m_2} \right ],
	\end{equation*}
	where $\bar m_1(M)\leq -\log M+1+\bar m_1$ and $\bar m_2(M)=\bar m_2$. It then follows along the same lines as before, that for any $k \in K$,
\begin{equation}\label{eq:estimateMxi1}
\sum _{\bfm\in \Z ^2}\sup _{\mathbf{s}\in F_{\bfm}}|g _{\bb \xi , M,1 }(\mathbf{s},k)|\leq DM^{-\xi _1},
\end{equation}	
where $D$ is defined in \eqref{eq:estimatedRi2}. In the same way, we obtain for $f_{1,M}$ the estimate
\begin{equation}\label{eq:estimateMxi2}
		\sum _{\bfm\in \Z ^2}\sup _{\mathbf{s}\in F_{\bfm}}|g _{\bb \xi , 1,M }(\mathbf{s},k)|\leq DM^{-\xi _2}.
\end{equation}	 
	 
\medskip

It remains to prove that $D$ is indeed finite, then both \eqref{eq:driLem} and \eqref{eq:dRiwithM} follow. Recalling  that $\norma{\bfA^{(j)}\bfX^{(j)}}=|A_j|^{\alpha_j}\norma{\bfX^{(j)}}$, we need to show that the expectations below are finite. 
	\begin{align*}
		&\E \Big[ \norma{\bfB^{(1)}}^{\epsilon} \big(|A_1|^{\alpha_1} \norma{\bfX ^{(1)}}+\norma{\bfB^{(1)}}\big) ^{\xi _1-\epsilon  } \big(|A_2|^{\alpha_2} \norma{\bfX^{(2)}} +\norma{\bfB^{(2)}} \big) ^{\xi _2} \Big]\\
		+&\E \Big[ \norma{\bfB^{(2)}}^{\epsilon} \big(|A_1|^{\alpha_1}\norma{\bfX ^{(1)}}+\norma{\bfB^{(1)}} \big) ^{\xi _1  } \big(|A_2|^{\alpha_2} \norma{\bfX^{(2)}} +\norma{\bfB^{(2)}} \big) ^{\xi _2-\epsilon }  \Big] =: E_1+E_2
	\end{align*}
 Expanding the products, using that $\bb X$ is independent of $(\bb A, \bb B)$ and the inequality
 	$$ x^{\epsilon} y^{s-\epsilon} \leq x^s+y^s $$
 valid for all $x,y>0$ and $0< \epsilon<s$, we obtain that $E_1$ is bounded by
	\begin{align*}
		& \E \big[\norma{\bfB^{(1)}}^{\epsilon}|A_1|^{\alpha_1(\xi _1-\epsilon )} |A_2|^{\alpha_2\xi _2} \big]  \E \big[ \norma{\bfX ^{(1)}}^{\xi _1-\epsilon} \norma{\bfX ^{(2)}}^{\xi _2}  \big] \\
	%	&+\E \big[ \norma{\bfB^{(2)}}^{\epsilon} |A_1|^{ \alpha_1 \xi _1} |A_2|^{\alpha_2(\xi _2-\epsilon )} \big] \E \big[ \norma{\bfX ^{(1)}}^{\xi _1} \norma{\bfX ^{(2)}}^{\xi _2-\epsilon } \big]\\
		&+\E \big[ \norma{\bfB^{(1)}}^{\epsilon} |A_1|^{\alpha_1(\xi _1-\epsilon )} \norma{\bfB^{(2)}}^{\xi _2} \big] \E \big[ \norma{\bfX ^{(1)}}^{\xi _1-\epsilon} \big]\\
		%\left (\E |\bfB^{(2)}|_{\a}^{\epsilon}e^{(\xi _2-\epsilon )U_2} |\bfB^{(1)}|_{\a}^{\xi _1} \right ) \left(\E |\bfX ^{(2)}|_{\a}^{\xi _2-\epsilon} \right ) \\
		&+\E \big[\norma{\bfB^{(1)}|}^{\xi _1} |A_2|^{\alpha_2\xi _2} \big] \E \big[ \norma{\bfX ^{(2)}}^{\xi _2} \big]
		%+\left (\E |\bfB^{(2)}|_{\a}^{\xi _2} e^{\xi _1U_1} \right ) \left(\E |\bfX ^{(1)}|_{\a}^{\xi _1} \right )
		+\E \big[ \norma{\bfB^{(1)}}^{\xi _1} \norma{\bfB^{(2)}}^{\xi _2} \big] \\
		& \leq \Big( \E \big[\norma{\bfB^{(1)}}^{\xi_1} |A_2|^{\alpha_2\xi _2} \big] + \E \big[|A_1|^{\alpha_1 \xi _1} |A_2|^{\alpha_2\xi _2} \big] \Big) M(\xi_1-\epsilon, \xi_2) \\
		&+ \Big(\E \big[ \norma{\bfB^{(1)}}^{\xi_1}  \norma{\bfB^{(2)}}^{\xi _2} \big] + \E \big[  |A_1|^{\alpha_1\xi _1} \norma{\bfB^{(2)}}^{\xi _2} \big] \Big) M(\xi_1 - \epsilon,0) \\
			&+\E \big[\norma{\bfB^{(1)}|}^{\xi _1} |A_2|^{\alpha_2\xi _2} \big] M(0,\xi_2)
		%+\left (\E |\bfB^{(2)}|_{\a}^{\xi _2} e^{\xi _1U_1} \right ) \left(\E |\bfX ^{(1)}|_{\a}^{\xi _1} \right )
		+\E \big[ \norma{\bfB^{(1)}}^{\xi _1} \norma{\bfB^{(2)}}^{\xi _2} \big]	\\
		& \leq \psi(\bb \xi) \big( M(\xi_1-\epsilon, \xi_2) + M(\xi_1 - \epsilon,0) + M(0,\xi_2) +1 \big)
		\end{align*}
	where $M(\xi_1, \xi_2)= \E \big[ \norma{\bfX ^{(1)}}^{\xi _1} \norma{\bfX ^{(2)}}^{\xi _2}  \big]$. By our assumptions, $(\xi_1, \xi_2) \in (0,1)^2$ and hence $\psi(\bb \xi)$ is finite by Lemma \ref{lem:properties.Psi.blocks}. The finiteness of $M(\xi_1 - \epsilon,0)$ and $M(0,\xi_2)$ follows from Theorem \ref{thm:moments:blocks} together with Lemma \ref{lem:properties.phi2}. Theorem \ref{thm:moments:blocks} also provides us with $M(\xi_1-\epsilon,\xi_2)<\infty$ if we can show that $\phi(\xi_1-\epsilon, \xi_2)<1$. This can be seen as follows. Since $\nabla \phi(\bb \xi^*)$ is parallel to (1,1), we may expand $\phi$ locally around $\xi^*$ by
		$$ \phi(\bb \xi^*+\bb h) = 1 + \skalar{\bb h, (1,1)} + o(\bb h),$$
		which gives in particular that
		$$ \phi(\xi_1^*-\epsilon, \xi_2^*+\delta) <1, \quad \phi(\xi_1^*+\delta, \xi_2^*-\epsilon)<1$$
		for sufficiently small $0< \delta < \epsilon$.
		Hence, by our choice of $\delta$ and $\epsilon$, $\phi (\xi _1 - \epsilon , \xi _2)<1$ as well as $\phi (\xi _1 , \xi _2-\epsilon )<1$. This concludes the proof that $E_1<\infty$. 	
	A bound for $E_2$ is obtained in the same way.
\end{proof}

\begin{lemma}\label{lem:implicit.renewal}
	Let $f \in \bb H^{\epsilon}(\R^d \setminus [\mathsf{blocks}])$. It holds for any $k \in K$  that 
		\begin{equation}\label{eq:implicit.renewal}
			 e^{\skalar{\bb \xi^*, s}}\E [f(e^{-\mathbf{s}\slash 
			\a}k \bb X) ]= g_{\bb \xi^*} * \wt{\mathds{U}}_{\bb \xi^*}(\bb s, k).
		\end{equation}
		In particular, there is $C$ independent of $f$ such that
		\begin{equation}\label{eq:bound:exiE}
			 \sup_{\bb s, k} \abs{e^{\skalar{\bb \xi^*, s}}\E [f(e^{-\mathbf{s}\slash 
				\a}k \bb X)]} \le C \norm{g_{\bb \xi^*}}_{\dri} 
		\end{equation}
		 and moreover
		 \begin{equation}\label{eq:bound:logt.txi.E}
		 	 \sup_{t >1} \,  (\log t)^{1/2} t^{\xi^*_1+ \xi^*_2} \abs{\E\Big[ f\big(t^{-1/\bb \alpha} \bb X\big)\Big]} ~\le~  C \norm{g_{\bb \xi^*}}_{\dri}. 
		 \end{equation}
\end{lemma}

\begin{proof}
	Let $f$ be in $ \bb H^\epsilon(\R^d \setminus [\mathsf{blocs}])$ and set $\bar f(\mathbf{s},k)= \E [f(e^{-\mathbf{s}\slash 
		\bb \a}k\bb X)]$. Then
	$$\bar f*\wt \mu (\mathbf{s},k)=\sum _{k\in K} \int _{\R ^2}\E [f(e^{-(\mathbf{s} -\mathbf{r})\slash 
		\bb \a}kk'\bb X)]\ d\wt \mu(\mathbf{r},k')=\E[ f(e^{-\mathbf{s}\slash 
		\bb \a}k\bfA \bfX)]$$
	Similarly, for any $n \in \N$,
	$$\bar f*\wt \mu ^n(\mathbf{s},k)= \E [f(e^{-\mathbf{s}\slash 
		\bb \a}k\bfA _1...\bfA _n\bfX)]= \E [f(e^{(-\bb s + \bb S_n)/\bb \alpha}  k\bb L_n \bb X )],$$
		where $\bb L_n=  \prod_{i=1}^n \bb K_i$.
		Using that $\lim_{n \to \infty} S_{n,j}=-\infty$ $\P$-a.s. (as a consequence of \eqref{eq:logAi}) and that the support of $f$ is bounded away from $[\mathsf{blocks}]$ and thus also bounded away from $0$, it follows that
		\begin{equation}
			 \lim_{n \to \infty} \bar f*\wt \mu ^n(\mathbf{s},k) =0 \label{eq:fmun:to.zero}
		\end{equation}
		for all $\bb s \in \R^2$ and $k \in K$ by an appeal to the dominated convergence theorem.
		Observe further that for any bounded function $h: \R^2 \times K\to \R$, upon defining $h_{\bb \xi^*}(\bb s,k):=e^{\skalar{\bb \xi^*, \bb s}}h(\bb s,k)$ it holds that
		\begin{align*}
			e^{\skalar{\bb \xi^*, \bb s}} h * \wt{\mu}(\bb s,k) &= \sum_{k' \in K} \int_{\R^2} e^{\skalar{\bb \xi^*, \bb s}} h(\bb s - \bb r, kk') d\wt{\mu}(\bb r, k') \\
			& = \sum_{k' \in K} \int_{\R^2} e^{\skalar{\bb \xi^*, (\bb s-\bb r)}} h(\bb s - \bb r, kk') e^{\skalar{\bb \xi^*, \bb r}} d\wt{\mu}(\bb r, k') = h_{\bb \xi^*} * \wt{\mu}_{\bb \xi^*}(\bb s,k)
		\end{align*}
		Hence
	\begin{align*}
		g_{\bb \xi^*}(\mathbf{s},k)&=e^{\skalar{\bb \xi^*, \bb s}} \big(\bar f(\mathbf{s},k)-\bar f*\wt \mu (\mathbf{s},k) \big) = \bar{f}_{\bb \xi^*}(\mathbf{s},k)-\bar{f}_{\bb \xi^*}*\wt{\mu}_{\bb \xi^*} (\mathbf{s},k)
	\end{align*}
	and 
	\begin{align*}
		&g _{\bb \xi^*}*\left (\sum _{n=0}^N   \wt{\mu} _{\bb \xi^*}^n\right )(\mathbf{s},k)=	\sum _{n=0}^N   g _{\bb \xi^*}*\wt{\mu} _{\bb \xi^*}^n(\mathbf{s},k)=e^{\langle \bb \xi^* , \mathbf{s}\rangle }
		(\bar f(\mathbf{s},k) - \bar f *\wt \mu ^N(\mathbf{s},k))	
	\end{align*}
	 By \eqref{eq:bound:dRi.any} of Lemma \ref{lem:dRi.proof.of.Renewal.Theorem} we have that $\sup_{\bb s, k} \abs{g_{\bb \xi^*}} * \wt{\mathds{U}}_{\bb \xi^*}(\bb s, k)< \infty$ and hence also $g _{\bb \xi^*}*\left (\sum _{n=0}^N \wt{\mu} _{\bb \xi^*}^n\right )$ is uniformly bounded and converges to  $g_{\bb \xi^*}*\wt{\mathbb{U}}_{\bb \xi^*}$. Using \eqref{eq:fmun:to.zero} we conclude that 
	\begin{equation*}
		g _{\bb \xi^* }*\wt{\mathbb{U}}_{\bb \xi^*}=\bar f _{\bb \xi^*},
	\end{equation*}
	which proves \eqref{eq:implicit.renewal} and (with the help of \eqref{eq:bound:dRi.any}) also \eqref{eq:bound:exiE}. In the same way, \eqref{eq:bound:logt.txi.E} follows from \eqref{eq:bound:dRi.rho}, using that $\bb \rho=\bb m_{\bb \xi^*}$ is parallel to $(1,1)$ by Lemma \ref{lem:properties.phi2}.
\end{proof}

Now we may conclude the assertions of Proposition \ref{prop:gdri}.

\begin{proof}[Proof of Proposition \ref{prop:gdri}]
	Considering \eqref{eq:dri},  the inequality $1+\abs{s_i}\leq \delta^{-1} e^{\d |s_i|}$ (valid for $\delta \in (0,1)$)
	allows us to bound 
	$$(1+\norm{s}) g_{\bb \xi^*}(\bb s,k) \le C\big( g_{\bb \xi* + \delta \bb e_1} (\bb s,k) + g_{\bb \xi* - \delta \bb e_1} (\bb s,k) + g_{\bb \xi* + \delta \bb e_2} (\bb s,k) + g_{\bb \xi* - \delta \bb e_2} (\bb s,k) \big).$$
	Hence \eqref{eq:dri} follows from Lemma \ref{lem:dri}.
	The bound \eqref{eq:boundLimitDRI} is proved in Lemma \ref{lem:implicit.renewal}. The bound \eqref{eq:boundLimitDRIwithM} then follows from \eqref{eq:boundLimitDRI} by an application of Lemma \ref{lem:dri},  \eqref{eq:dRiwithM}.	
\end{proof}

\begin{proof}[Proof of Theorem \ref{thm:main:blocks}]
	Abbreviate $\chi(t):=(\log t)^{1/2} t^{\xi^*_1+\xi^*_2}$.
	In order to prove \eqref{eq:precise}, we have to show that $\lim_{t \to \infty} \chi(t) \E \big[ f(t^{-1/\bb \alpha} \bb X)\big]$ exists for all $f \in \mathcal{C}_c(\R^d \setminus [\mathsf{blocks}])$. This then gives the existence of a Radon measure $\Lambda_2$ on $\R^d \setminus [\mathsf{blocks}]$. 
	
	\medskip
	
	\textbf{Proof of \eqref{eq:precise} for Hölder functions}.
	By Lemma \ref{lem:implicit.renewal} and the two-dimensional renewal theorem with group $K$,  Theorem \ref{thm:renewalextended},  we have for any $f \in \bb H^{\epsilon}(\R^d \setminus[\mathsf{blocks}])$ and for any $k_1 \in K$ that
	\begin{align*}
		\lim _{r\to \8}r^{1\slash 2}\bar f _{\bb \xi^*}(rm_{\bb \xi^*}, k_1 )=c\sum _{k\in K}\int _{\R ^2}g _{\bb \xi^* }(\mathbf{s},k) \ d\mathbf{s}
	\end{align*}
	with $c$ independent of $f$.
	 Using that $\nabla \phi (\bb \xi^* )$ is proportional to $(1,1)$ by Lemma \ref{lem:properties.phi2},
	% it holds that $\nabla \phi  = 2^{-1\slash 2} \| \nabla \phi \|(1,1)$ and substituting% $r2^{-1\slash 2} \| \nabla \phi \|\to r$ 
	we have 
		\begin{align}\label{eq:limrenewal}
			\lim_{t\to \8} r^{1\slash 2}e^{(\xi^* _1+\xi^* _2)r}\bar f (r (1,1), k_1)=\wt c\sum _{k\in K}\int _{\R ^2}g _{\bb \xi^*} (\mathbf{s},k) \ d\mathbf{s}
		\end{align}
		with a different constant. Let $r=\log t$. We have
	\begin{equation*}
		\bar f ((\log t ) (1,1))=\E [f (e^{-\log t \slash \a _1}X_1,...,e^{-\log t \slash \a _d}X_d)]=\E [f(t^{-1\slash \bb \a}\bfX)]
	\end{equation*}
	 and so
	\begin{equation}\label{eq:drilim}
	\lim _{t\to \8}(\log t)t^{\xi^* _1+\xi^* _2}\E [f(t^{-1\slash \bb \a}\bfX)]=
	\wt c\sum _{k\in K}\int _{\R ^2}g _{\bb \xi^*} (\mathbf{s},k) \ d\mathbf{s}.
	\end{equation}
	which proves \eqref{eq:precise} for H\"older functions. 
	
	\medskip
	
	\textbf{Proof of \eqref{eq:precise} for continuous functions with compact support, Existence of $\Lambda_2$}. Consider a compact set $K \subset \R^d \setminus [\mathsf{blocks}]$. Choosing $g \in \bb H^{\epsilon}(\R^d \setminus[\mathsf{blocks}])$ such that $g \ge \Ind{K}$, we obtain 
	\begin{equation}\label{eq:bound Compact set}
		 \limsup_{t \to \infty} \chi(t) \E \Big[  \Ind{K}(t^{-1/\bb \alpha} \bb X)\Big] \le \lim_{t\to \8} \chi(t) \E [ g(t^{-1/\alpha} \bb X)] \le C_K< \infty.
	\end{equation}
	Now for any $f \in \mathcal{C}_c(\R^d \setminus [\mathsf{blocks}])$ and any $\delta>0$ there are  $g, h \in \bb H^{\eps}(\R^d \setminus[\mathsf{blocks}])$ the support of which is contained in a compact set $K \subset \R^d \setminus [\mathsf{blocks}]$ and such that $g \le f \le h$ and $ \norm{g -h}_\infty < \delta / C_K.$
	We have
	\begin{align*} \lim_{t\to \8} \chi(t) \E \big[ g(t^{-1/\bb \alpha} \bb X) \big] \le& \liminf_{t\to \8} \chi(t) \E \big[ f(t^{-1/\bb \alpha} \bb X) \big]\\ \le& \limsup_{t\to \8} \chi(t) \E \big[ f(t^{-1/\bb \alpha} \bb X) \big] \le \lim_{t\to \8} \chi(t) \E \big[ h(t^{-1/\bb \alpha} \bb X) \big]
	\end{align*}
	and by \eqref{eq:bound Compact set} it follows that
	$$ \limsup_{t \to \infty} \abs{ \chi(t) \E \big[ g(t^{-1/\bb \alpha} \bb X)\big] - \chi(t) \E \big[ h(t^{-1/\bb \alpha} \bb X)\big]} \le \limsup_{t \to \infty} \norm{g-h}_\infty \chi(t) \E \Big[  \Ind{K}(t^{-1/\bb \alpha} \bb X)\Big] < \delta. $$
	This proves the existence of $\lim_{t \to \infty} \E \big[ f(t^{-1/\bb \alpha} \bb X) \big]$ for any $f \in \mathcal{C}_c(\R^d \setminus [\mathsf{blocks}])$. It follows  (cf. e.g. \cite[Chapter 7]{Folland1999} that there is a unique Radon (i.e., locally finite) measure $\Lambda_2$ on $\R^d \setminus [\mathsf{blocks}]$ such that
	$$ \lim_{t \to \infty} \chi(t) \E \big[ f(t^{-1/\bb \alpha}\bb X)\big]= \int f(\bb x) \Lambda_2(d\bb x) $$
for any $f \in \mathcal{C}_c(\R^d \setminus [\mathsf{blocks}])$.

\medskip

\textbf{Proof of \eqref{eq:precise} for continuous functions}. 
We now extend the convergence to $\bb C( \R^d \setminus [\mathsf{blocks}])$. By writing $f=f^+ -(-f)^+$ it suffices to prove the result for nonnegative functions. Given nonnegative $f \in \bb C( \R^d \setminus [\mathsf{blocks}])$, there are $\eta_i$ such that $\supp f \subset \{ \bb x \, :\, \norma{\bb x^{(i)}} > \eta_i \, i \in \{1,2\}\}.$
For any $R>0$ we may decompose 
$$ f =f_{1;R} + f_{2;R} + f_{3;R}$$
such that $f_{1;R} \in \mathcal{C}_c(\R^d \setminus [\mathsf{blocks}])$ and $$\supp f_{2;R} \subset \{ \bb x \, :\, \norma{\bb x^{(1)}} > R/\eta_2, \norma{\bb x^{(2)}}> \eta_2\}, \ \supp f_{3;R} \subset \{ \bb x \, :\, \norma{\bb x^{(1)}} > \eta_1, \norma{\bb x^{(1)}}> R/\eta_1\}.$$ 
Consider now $g,h \in \bb H^\epsilon(\R^d \setminus [\mathsf{blocks}])$ such that 
$$ g \ge \norm{f}_\infty \Ind{\{ \bb x \, : \, \norma{\bb x^{((i))}} > \eta_2 \, i \in \{1,2\} \}}, \quad h \ge \norm{f}_\infty \Ind{\{ \bb x \, : \, \norma{\bb x^{(i)}} > \eta_1 \, i \in \{1,2\} \}}.$$
Recalling \eqref{eq:fM}, it holds that $\abs{f_{2;R}} \le g_{R,1}$ and $\abs{f_{3;R}} \le h_{1,R}$. By Lemma \ref{lem:implicit.renewal} and \eqref{eq:dRiwithM}  it follows that 
	$$ \chi(t) \abs{ \E \big[ f_{2;R}(t^{-1/\alpha} \bb X)\big]} \le C_f R^{-\xi^*_1}, \qquad  \chi(t) \abs{ \E \big[ f_{3;R}(t^{-1/\alpha} \bb X)\big]} \le C_f R^{-\xi^*_2}$$
	for all $t>0$ and $R>0$.
	Hence as well  $\abs{\int f_{2;R}(\bb x) \Lambda_2(d \bb x)} \le C_f R^{-\xi^*_1} $, $\abs{\int f_{3;R}(\bb x) \Lambda_2(d \bb x)} \le C_f R^{-\xi^*_2} $. This proves in particular that $\int f(\bb x) \Lambda_2(d \bb x)$ is well defined.
	
	Now for any $\delta>0$, choose $R$ such that $C_fR^{-\xi^*_2}, C_f R^{-\xi^*_3} < \delta/5$ and thereupon $t_0$ such that $\abs{\chi(t)\E \big[f_{1;R}(t^{-1/\bb \alpha} \bb X) \big] - \int f_{1;R}(\bb x) \Lambda_2(d \bb x)} < \delta/5$ for all $t\ge t_0$. Then, for all $t \ge t_0$,
	\begin{align*}
		&~\Big| \chi(t) \E \big[ f(t^{-1/\bb \alpha} \bb X) - \int f(\bb x) \Lambda_2(d \bb x)\big]\Big| \\
			\le &~\Big| \chi(t) \E \big[ f(t^{-1/\bb \alpha} \bb X) - \chi(t) \E \big[ f_{1;R}(t^{-1/\bb \alpha} \bb X)\big]\Big| + \Big| \chi(t) \E \big[ f_{1;R}(t^{-1/\bb \alpha} \bb X) - \int f_{1;R}(\bb x) \Lambda_2(d \bb x)\big]\Big| \\
			&~ +  \Big|\int f_{1;R}(\bb x) \Lambda_2(d \bb x) - \int f(\bb x) \Lambda_2(d \bb x)\big]\Big| \\
			\le &~ \chi(t) \Big| \E \big[ f_{2;R}(t^{-1/\alpha} \bb X)\big]\Big| +  \chi(t) \Big| \E \big[ f_{3;R}(t^{-1/\alpha} \bb X)\big]\Big| + \frac{\delta}{5} + \Big|\int f_{2;R}(\bb x) \Lambda_2(d \bb x)\Big| +\Big|\int f_{3;R}(\bb x) \Lambda_2(d \bb x)\Big| \\
			< &~ \delta.
	\end{align*}
	We have thus proved \eqref{eq:precise} in full generality.

	\medskip
	
	\textbf{Proof of \eqref{eq:measure1}}.	Employing the same approximation as above, it suffices to prove \eqref{eq:measure1}  for $f \in \bb H^\epsilon(\R^d \setminus[\mathsf{blocks}])$. This will be done as the next step. 
	Consider any $f \in \bb H^\epsilon(\R^d \setminus[\mathsf{blocks}])$.
	\begin{align*}
	&\int f(k \bb x) \Lambda_2(d \bb x) =\lim _{t\to \8} \chi(t)\E [f(t^{-1\slash \bb \a}k \bfX)]\\
	 =& \lim _{t\to \8} \chi(t)	\bar f ((\log t ) (1,1), k)=
	 \lim _{t\to \8}\chi(t) \bar f ((\log t ) (1,1))\\
	 =&\lim _{t\to \8}\chi(t) \E [f(t^{-1\slash \bb \a}\bfX)] = \int f(\bb x) \Lambda_2(d \bb x)
	\end{align*}
	Hence the first part of \eqref{eq:measure1} follows.
	Let now $f_r(\bb x)=f ( r^{1\slash \bb \a } \bb x )$.
			On one hand,
		\begin{equation*}
			\lim _{t\to \8}(\log t)^{1\slash 2}t^{\xi^* _1+\xi^* _2}\E [f_r (t^{-1\slash \bb \a}X )]=\int f(r^{1/\bb \alpha} \bb x) \Lambda_2(d \bb x)
		\end{equation*}
		and on the other,
		\begin{align*}
			&~\lim _{t\to \8}(\log t)^{1\slash 2}t^{\xi^* _1+\xi^* _2}\E [f _r (t^{-1\slash \bb \a}X  )]\\ =&~\lim _{t\to \8}\frac{(\log (t))^{1\slash 2}}{(\log (r^{-1}t))^{1\slash 2}}(\log (r^{-1}t))^{1\slash 2}r^{\xi^* _1+\xi^* _2}(r^{-1}t)^{\xi^* _1+\xi^* _2}\E[f ((r^{-1}t)^{-1\slash \bb \a}\bfX )]\\
			=&~r^{\xi^* _1+\xi^* _2} \int f(\bb x) \Lambda(d \bb x) ,
		\end{align*}
		which proves \eqref{eq:measure1}.
		 As a consequence, we obtain
					\begin{equation}\label{eq:measure2}
			\limsup _{t\to \8} (\log t)^{1\slash 2}t^{\xi^* _1+\xi^* _2}\P \left (\|\bfX ^{(i)}\| _{\bb \a }>rt,\ i=1,2\right ) \leq Cr^{-\xi^* _1-\xi^* _2}.
		\end{equation}
		
		Namely, consider  $f\in \bb H^\epsilon(\R^d \setminus [\mathsf{blocks}])$, $\supp f\subset \{ \mathbf{x}: \|\mathbf{x}^{(i)}\| _{\bb \a}\geq 1\slash 2\}$
		and $f= 1$ on the set  $W= \{ \mathbf{x}: \|\mathbf{x}^{(i)}\| _{\bb \a}\geq 1, i=1,2\}$.
		Then $\Ind W\leq f$ and so 
		\begin{equation*}
			\P \left (\|\bfX ^{(i)}\|_{\bb \a }>t,\ i=1,2	\right )= \E [\Ind W (t^{-1\slash \bb \a}\bfX  )]\leq 
			\E [f (t^{-1\slash \bb \a}\bfX )]
		\end{equation*}
		Hence
		\begin{equation*}
			\limsup _{t\to \8} \chi(t)\P \left ( \|\bfX ^{(i)}\|_{\bb \a }>t,\ i=1,2
			\right )<\lim _{t\to \8} \chi(t) \E[f (t^{-1\slash \a}\bfX  )]=: C.
		\end{equation*}
		Using \eqref{eq:measure2}, 
		\begin{equation*}
			\limsup _{t\to \8} \chi(t)\P \left ( \|\bfX ^{(i)}\|_{\bb\a }>rt,\ i=1,2
			\right )\leq \lim _{t\to \8} \chi(t)\E [f\ _{r^{-1}} (t^{-1\slash \bb \a}X  )]= Cr^{-\xi^* _1-\xi^* _2} .
		\end{equation*}
		
		\medskip
		
		\textbf{Proof of \eqref{eq:measure3}}.		
		Now we shall prove \eqref{eq:measure3}. Let $0<\d <1$, 
		$\phi \in \Cf_c(\wt S_1) $ and let $g$ be a continuous function on $(0,\8)$, $\supp g\subset (1-\frac{\d}{2} , \8)$, $g(s)=1$ for $s\geq 1$. We define $f$ writing in polar coordinates (see \eqref{polar}),
		\begin{align*}
			&f(s,\omega)= g(s)\phi (\omega )\\
			&f_{1+\d}(s,\omega)= g((1+\d)s)\phi (\omega )\\
			&f_{1-\d}(s,\omega)= g((1-\d)s)\phi (\omega ).
		\end{align*}   
		Observe that $g((1- \d)s)=0$ for $s\leq 1+ \d \slash 2$ and $g((1+ \d)s)\geq 1$ for $s\geq 1- \d \slash 2$. Hence, %and for sufficiently small $\d$ we have
		%$\min \left (1-(1+\d)^{-1}, (1-\d)^{-1}-1 \right )\geq \frac{\d }{2}$.  Hence
		\begin{equation*}
			f_{1+\d}(s,\omega) - f_{1-\d}(s,\omega)\geq \Ind {[1-\d \slash 2, 1+\d \slash 2 ]} (s)\phi (\omega).
		\end{equation*}
		Given $\eps>0 $ we take $\d $ such that $(1+\d) ^{\xi^* _1+\xi^* _2}-(1-\d) ^{\xi^* _1+\xi^* _2}\leq \eps $.
		We write  $\bfX$ in polar coordinates as
		$$
		\bfX=\Theta (\|\bfX\|_{\a}, \omega _\bfX), \quad  \omega_{\bfX}\in S^{d-1}. $$
				Then
	 	\begin{align*}
		\int \Ind{\wt S_1}(\bb x) \phi(\bb x) \Lambda_2(d \bb x)  &\leq	\limsup _{t\to \8}(\log t)^{1\slash 2}t^{\xi^* _1+\xi^* _2}\E[ \Ind {\left [ 1-\d \slash 2, 1+\d \slash 2 \right ]} (t^{-1\slash \bb \a}\|\bfX\|_{\a} )\phi (\omega _\bfX )]\\
		&\leq \int f_{1+\d}\,d\Lambda _2 - \int f_{1-\d} \,d\Lambda _2 
			= \left ( (1+\d) ^{\xi^* _1+\xi^* _2}-(1-\d) ^{\xi^* _1+\xi^* _2}\right )\int f\, d\Lambda _2 \\
			&\leq \eps \int f\, d\Lambda _2,
		\end{align*}
		and \eqref{eq:measure3} follows.
		
		\medskip
		
		\textbf{Proof of \eqref{eq:measure4}}.	
		To show \eqref{eq:measure4}, we take $\phi \in \Cf_c(\wt S _1)$ and 
			$f(s,\omega )=f(s^{1\slash \bb \a}\omega)=\Ind { [1,\8)}(s)\phi (\omega )$. %B_1(0)^c}(s^{1\slash \a }\omega)
	Then, by \eqref{eq:measure3}, $\Lambda _2 (\{ \mathbf{x}: f \ \mbox{is not continuous at}\ \mathbf{x} \})=0$ and by the Portmanteau Theorem, 
	\begin{equation}\label{eq:nu}
		\lim _{t\to \8}(\log t)^{1\slash 2}t^{\xi^* _1+\xi^* _2}\E [f(t^{-1\slash \bb \a}\bfX)]=\int f\, d\Lambda _2,
	\end{equation}
	which is finite by \eqref{eq:measure2}.
	%Moreover, by \eqref{eq:measure2}, for some $\eta$ depending on the support of $\phi$ 
	%\begin{equation*}
	%	\left | \lim _{t\to \8}(\log t)^{1\slash 2}t^{\xi^* _1+\xi^* _2}\E[ f(t^{-1\slash \bb \a}\bfX %)]\right |
	%	\leq \limsup _{t\to \8}(\log t)^{1\slash 2}t^{\xi^* _1+\xi^* _2}\E [\phi (\o )\Ind {\{\|\bfX %^{(i)}\|_{\bb \a }\geq \eta t, i=1,2\}}]<\8.
		%\Ind { [1,\8)}(\|\bfX ^{(i)}\|_{\a })
%	\end{equation*}
	Hence \eqref{eq:nu} defines a positive functional on $\Cf_c(\wt S_1)$ i.e. there is a Radon measure $\nu $ on $\wt S_1$ such that $\int f\, d\Lambda _2 =\int \phi\,  d\nu$.
	
	Similarly, for
	$f_r(s,\omega)=\Ind { [r,\8)}(s)\phi (\omega )=f_1\circ r^{-1\slash \bb \a}$,
	%\phi (\omega )\Ind { B_r(0)^c}(\delta _s\omega)
	\begin{eqnarray*}
	\lim _{t\to \8}(\log t)^{1\slash 2}t^{\xi^* _1+\xi^* _2}\E [f_r(t^{-1\slash \bb \a}\bfX )]=\int f_r\, d\Lambda _2
	\end{eqnarray*}
	exists and
	$\int f_r\,d\Lambda _2 =r^{-\xi^* _1-\xi^* _2}\int f_r\, d\Lambda _2= r^{-\xi^* _1-\xi^* _2}\int \phi \, d\nu $.
	
Therefore,  
\begin{align*}
&\int f_r\ d\L _2 =r^{-\xi^* _1-\xi^* _2}\int \phi\, d \nu = (\xi^* _1+\xi^* _2)\int _r^{\8}\frac{ds}{s^{\xi^* _1+\xi^* _2+1}} \int \phi\, d\nu\\
& = (\xi^* _1+\xi^* _2)\int _0^{\8}\mathbf{1} _{[r,\8)}(s)\phi (\omega )\frac{ds}{s^{\xi^* _1+\xi^* _2+1}} d\nu (\omega )=(\xi^* _1+\xi^* _2)\int _0^{\8}f_r(s^{1\slash \a }\omega )\frac{ds }{s^{\xi^* _1+\xi^* _2+1}}d \nu (\omega) , 
\end{align*}
which implies \eqref{eq:measure4}.

\medskip

\textbf{Proof of invariance properties of $\nu$}.
The $K$-invariance of $\nu$ is a direct consequence of the $K$-invariance of $\Lambda_2$, provided in the first part of \eqref{eq:measure1}.
Since bounded continuous functions determine the measure uniquely, we also have \eqref{eq:measure4} for bounded measurable functions.

\medskip

\textbf{Proof that $\Lambda_2$ and $\nu$ are nonzero, and that $\nu$ is unbounded}. These final assertions will be proved in Section \ref{sect:positivity} below, where we finish the proof of Theorem \ref{thm:main:blocks}.
\end{proof}

\section{Positivity of the Limit Measure}\label{sect:positivity}

In this section we will prove that a) the measure $\Lambda_2$ is nonzero, and that b) its spectral measure $\nu$ is unbounded with explosion at $[\mathsf{blocks}]$. Both assertions will follow from Theorem \ref{thm:positivity:measure}. The nontriviality of $\Lambda_2$ is a direct consequence of \eqref{eq:positivity}. The same result will also be used in the proof of Proposition \ref{thm:nu:unbounded} below to deduce that $\nu$ is unbounded.

\begin{theorem}\label{thm:positivity:measure}
	Assume  that for every $R>0$,
	\begin{equation}\label{eq:R}
		\P \left (\norma{\bfX ^{(1)}}>R, \norma{\bfX ^{(2)}}>R  \right )=c_R>0.
	\end{equation}
	Then for $\bb \xi^*$ as in Theorem \ref{thm:main:blocks}, there is $c>0$ such that for every $\eps >0$
	\begin{equation}\label{eq:positivity}
		\liminf _{t\to \8}\,	(\log t)^{1\slash 2}t ^{\xi^* _1 +\xi^* _2}\P \left (\norma{\bfX ^{(1)}}>t, \norma{\bfX ^{(2)}}>\eps t \right )
		>c\eps ^{-\xi^* _2}.
	\end{equation}
\end{theorem}

Sufficient conditions for \eqref{eq:R} will be given at the end of the Section.
The proof of Theorem \ref{thm:positivity:measure} will make use of the decomposition
\begin{equation}\label{eq:decomp:X}
	 \bb X = \bb X_{\le n} + \bb A_1 \cdots \bb A_n \bb X_{>n}
\end{equation}
 and rely on the following Proposition. We use here the shorthand notation 
 $$\Pi ^{(1)}_n:=e^{S_{n,1}} = \norma{ (\bb A_1 \cdots \bb A_n)^{(1)}}, \quad \Pi ^{(2)}_n:=e^{S_{n,2}}= \norma{ (\bb A_1 \cdots \bb A_n)^{(2)}}$$

\begin{proposition}\label{thm:keyasym}
	%Suppose Cramer condition \eqref{eq:Cramer}. 
	%that the covariance matrix $K$ in \eqref{eq:covariance} is strictly positive definite. 
	Under the assumptions of Theorem \ref{thm:main:blocks} there is $c>0$ such that for every $\eps >0$
	\begin{equation}\label{eq:2piestim}
		\liminf _{t\to \8}  t^{\xi^*_1+\xi^*_2}(\log t)^{1\slash 2}	\P \left ( \text{There is $n \in \N$ such that } \Pi ^{(1)}_n>t, \Pi ^{(2)}_n> \eps t \right )\geq c \eps ^{-\xi^* _2}.
	\end{equation}
	In particular, there is $t_0$ such that 	
	\begin{equation}
		\P \left (\exists _n \ \Pi _n^{(1)}>t, \Pi _n ^{(2)}>t \right )\geq c  t^{-\xi^*_1-\xi^*_2}(\log t)^{-1\slash 2}\quad \mbox{for}\ t\geq t_0.
	\end{equation}
\end{proposition}

The proof of Proposition \ref{thm:keyasym} will be given in Section \ref{sec:key} below. We now proceed with the proof of Theorem \ref{thm:positivity:measure}. 

\begin{proof}[Proof of Theorem \ref{thm:positivity:measure}]
	We start by relating the event in \eqref{eq:positivity} to the event in \eqref{eq:2piestim}. Introducing  $\wt W_n: = \big \{ \Pi ^{(1)}_n>t,\Pi ^{(2)}_n> \eps t \big \}$, consider the sequence of disjoint sets  $W_n: =\wt W_n\setminus (\wt W_1\cup ...\cup \wt W_{n-1})$ and define for parameters $R>0$, $M>0$ (to be chosen later on) the sequence of disjoint sets
	\begin{equation*}
		U_n:= W_n\cap \big\{ \norma{\bb X_{\leq n}^{(1)}}\leq Mt, \norma{\bb X_{\leq n}^{(2)}}\leq M\eps t\big\} \cap \big\{ \norma{(\bb X_{>n} ^{(1)}}>R, \norma{\bb X_{>n} ^{(2)}}>R \big\}.
	\end{equation*}
	Using the decomposition \eqref{eq:decomp:X} and recalling \eqref{eq:norm:subadditiv}, we have  for $i \in \{1,2\}$,
		\begin{align*}
		\norma{\bfX ^{(i)}} ~\geq &~ c_{\bb \a}^{-1} \norma{(\bfA _1...\bfA _n\bfX_{>n})^{(i)}} - \norma{\bb X_{\le n} ^{(i)}}\\
		=&~ c_{\bb \a}^{-1} \Pi ^{(i)}_n \norma{\bb X_{>n}^{(i)}}-\norma{\bb X_{\le n} ^{(i)}} 
	\end{align*}
	Hence, on the event $U_n$, for every $\eps>0$,
	\begin{equation*}
		\norma{ \bfX ^{(1)}} \geq c_{\a} ^{-1}Rt-Mt>t, \quad \norma{\bfX ^{(2)}} \geq c_{\a} ^{-1}\eps Rt-\eps M t>\eps t,
	\end{equation*}		
	%for $t\geq t_0$, 
	provided $c_{\a}^{-1}R>M+1$. 
	Below we will first choose $M$ and then $R$ accordingly. Both choices will be independent of the value of $\eps$. 	
	Using that $(\Pi_k^{(1)}, \Pi_k^{(2)})_{ k \le n}$ and $\bb X_{\le n}$ are independent of $\bb X_{>n}$, we have
	\begin{align*}
		&~ \P \left (\norma{\bfX ^{(1)}}>t, \norma{\bfX ^{(2)}}>\eps t \right )~\geq~  \sum _n \P ( U_n) \\
		=&~ \sum _n \P \left (   W_n\cap \big\{ \norma{\bb X_{\leq n}^{(1)}}\leq Mt, \norma{\bb X_{\leq n}^{(2)}}\leq M\eps t\big\}\right ) \P \left ( \norma{(\bb X_{>n} ^{(1)}}>R, \norma{\bb X_{>n} ^{(2)}}>R\right )\\
		=&~  \P \bigg ( \bigcup _n \Big(W_n\cap \big\{ \norma{\bb X_{\leq n}^{(1)}}\leq Mt, \norma{\bb X_{\leq n}^{(2)}}\leq M\eps t\big\}\Big) \bigg ) \P \left ( \norma{\bfX  ^{(1)}}>R, \norma{\bfX  ^{(2)}}>R \right )
	\end{align*}
	Hence in order to prove \eqref{eq:positivity}, it suffices by assumption \eqref{eq:R} to show that there is $c>0$ (independent of $\eps$) such that
	\begin{equation}\label{eq:minorization}
		\P \bigg ( \bigcup _n \Big(W_n\cap \big\{ \norma{\bb X_{\leq n}^{(1)}}\leq Mt, \norma{\bb X_{\leq n}^{(2)}}\leq M\eps t\big\}\Big) \bigg )\geq \frac{c}{2}\eps ^{-\xi^* _2}t ^{-\xi^* _1 -\xi^* _2}(\log t)^{-1\slash 2} 
	\end{equation}
	holds for all for $t\geq t_0=t_0(\eps )$.
	
	Abbreviating $Z_n:=\big\{ \norma{\bb X_{\leq n}^{(1)}}\leq Mt, \norma{\bb X_{\leq n}^{(2)}}\leq M\eps t\big\}$, we consider
	\begin{align}
	 \P \big( \bigcup_n (W_n \cap Z_n) \big) &= \sum_{n} \P \big( W_n \cap Z_n \big) = \sum_{n} \big( \P(W_n) - \P(W_n \cap Z_n^c) ) \notag \\
	 &= \P \big( \bigcup_n W_n \big)- \P \big( \bigcup_n (W_n \cap Z_n^c) \big) \label{eq:P:Wn}
	\end{align}
By \eqref{eq:2piestim} of Proposition \ref{thm:keyasym} there is $c>0$ such that for all $\eps>0$, 
		\begin{equation*}
		\liminf _{t\to \8} \ t^{\xi^* _1+\xi^* _2}(\log t)^{1\slash 2}\P \Big ( \bigcup _n  W_n \Big )\geq c\eps ^{-\xi^* _2}.
	\end{equation*}	 
	Hence, in view of \eqref{eq:P:Wn}, we can deduce \eqref{eq:minorization} if we can show that
		\begin{equation}\label{eq:complement}
		\limsup _{t\to \8}\ t^{\xi^* _1+\xi^* _2}(\log t)^{1\slash 2}\P \Big ( \bigcup _n \big ( W_n\cap \{ \norma{\bb X_{\le n}^{(1)}}> Mt\ \mbox{or}\  \norma{\bb X_{\le n}^{(2)}}> M\eps t\}\big )\Big)\leq 
		\frac{c}{4}\eps ^{-\xi^* _2}
	\end{equation}	
	holds for $c$ given by \eqref{eq:2piestim}. In the proof, we have to distinguish two cases.

	\textbf{Proof of \eqref{eq:complement} in the case $\max_{1 \le i \le d} \alpha_i \le 1$}. In this case, $\norma{ \cdot}$ is subadditive. That is, $c_{\bb \alpha}=1$ in \eqref{eq:norm:subadditiv}. Introducing for $j \in \{1,2\}$
	\begin{equation}\label{eq:def:Y}
		Y_j:=\sum _{k=1}^{\8}\Pi ^{(j)}_{k-1}\big (\norma{\bfB _k^{(j)}}+1\big ),
	\end{equation}
	we obviously have $\Pi ^{(j)} _n\leq Y_j$ for all $n \in \N$; and also
	\begin{equation*}
		\norma{\bb X_{\le n}^{(j)}} = \norma{\sum _{k=1}^n \bfA ^{(j)}_1...\bfA ^{(j)}_{k-1}\bfB ^{(j)} _k} \leq \sum _{k=1}^n \norma{ \bfA ^{(j)}_1...\bfA ^{(j)}_{k-1}\bfB ^{(j)} _k} \leq \sum _{k=1}^n \Pi ^{(j)}_{k-1} \norma{\bfB ^{(j)} _k}\leq Y_j. 
	\end{equation*}

	Observe that the sets
	$$ W_n\cap \{ \norma{\bb X_{\le n}^{(1)}}> Mt\ \mbox{or}\  \norma{\bb X_{\le n}^{(2)}}> M\eps t\} ,\ n\in \N $$ are disjoint and  each of them is contained in $  \{Y_{2}>\eps t, Y_1> Mt \} \cup    \{ Y_{1}> t, Y_2> M\eps t \}$. Indeed, 
	\begin{align*}
		&~ \big\{\Pi ^{(1)}_n> t, \Pi ^{(2)}_n>\eps t, \norma{\bb X_{\le n}^{(1)}}> Mt\ \mbox{or}\  \norma{\bb X_{\le n}^{(2)}}> M\eps t \big\} \\
		\subset&~   \big\{\Pi ^{(2)}_n>\eps t, \norma{\bb X_{\le n}^{(1)}}> Mt\big\}  
		\cup  \big\{\Pi ^{(1)}_n> t,  \norma{\bb X_{\le n}^{(2)}}> M\eps t\big\}\\
		\subset&~ \{Y_{2}>\eps t, Y_1> Mt \} \cup    \{ Y_{1}> t, Y_2> M\eps t \}.
	\end{align*}
	Therefore, it is enough to prove that 
	\begin{equation}\label{eq:domination1}
		\limsup _{t\to\8}\ (\log t)^{1\slash 2}t^{\xi^* _1+\xi^* _2}\P \left (  Y_{2}>\eps t, Y_1> Mt \right )\leq CM^{-\xi^* _1}\eps ^{-\xi^* _2}
	\end{equation}
	and
	\begin{equation}\label{eq:domination2}
		\limsup _{t\to\8}\ (\log t)^{1\slash 2}t^{\xi^* _1+\xi^* _2} \P \left (  Y_{1}> t, Y_2> M\eps t \right )\leq CM^{-\xi^* _2}\eps ^{-\xi^* _2},
	\end{equation}
	since we may then choose $M$ sufficiently large such that $C \cdot \max\{M^{-\xi^*_1}, M^{-\xi^*_2} \} \le \frac{c}{4}$.
	
	Let $f\in \mathbf{H}^\eps (\R ^2 \setminus [\mathsf{axes}])$, $\supp f\subset \{ x: |x_j| \geq 1\slash 2\}$, $f(x)=1$  if $|x_1|\geq 1, |x_2|\geq 1$. Consider $f_{M/\eps,1 }(x)=f(\eps M^{-1}x_1, x_2)$ and $f_{1,M\eps }(x)=f(x_1, M^{-1}\eps ^{-1}x_2)$. Then
	\begin{align*}
		\P \left (  Y_{2}>\eps  t, Y_1> Mt \right )&\leq  \E [f(\eps M^{-1}(\eps t)^{-1}Y_1, (\eps t)^{-1}Y_2)]= \E [f_{M/\eps,1 }((\eps t)^{-1} \bb Y)]\\[.2cm]
		\P \left (  Y_{1}>  t, Y_2> \eps Mt \right )&\leq  \E [f(t^{-1}Y_1, t^{-1}M^{-1}\eps ^{-1}Y_2)]= \E [f_{1,M\eps }(t^{-1}\bb Y)].
	\end{align*}
	As a consequence of \eqref{eq:def:Y}, $\bb Y=(Y_1,Y_2)$ satisfies the equations
	\begin{equation*}
		Y_j\eqdist e^{U_j}Y_j+\left (\norma{\bfB ^{(j)}}+1\right ), \quad j=1,2.
	\end{equation*}
	We may thus apply Proposition \ref{prop:gdri} with $\bb Y$ instead of $\bb X$ and obtain from \eqref{eq:boundLimitDRIwithM}
	\begin{align*}
		&~\limsup _{t\to \8}\ (\log t)^{1\slash 2}t^{\xi^* _1+\xi^* _2}\E [f_{M/\eps,1 }((\eps t)^{-1} \bb Y)] \\  =&~ \limsup _{t\to \8}\ \bigg(\frac{\log t}{\log (\eps t)}\bigg)^{1\slash 2} \eps^{-\xi^* _1-\xi^* _2} (\log \eps t)^{1\slash 2}(\eps t)^{\xi^* _1+\xi^* _2}\E [f_{M/\eps,1 }((\eps t)^{-1} \bb Y)] \\
		 \leq&~ \eps^{-\xi^*_1 - \xi^*_2} \cdot C (M/\eps)^{-\xi^*_1} =  C M^{-\xi^* _1}\eps ^{-\xi^* _2}		
		\end{align*}
		as well as
	\begin{align*}
		\lim _{t\to \8}\ (\log t)^{1\slash 2}t^{\xi^* _1+\xi^* _2} \E [f_{1,M\eps }(t^{-1}\bb Y)]&\leq C M^{-\xi^* _2}\eps ^{-\xi^* _2}.
	\end{align*}	
	This shows \eqref{eq:domination1} and  \eqref{eq:domination2} in the case when $\norma{\cdot}$ is subadditive.
	
	\textbf{Proof of \eqref{eq:complement} in the case $\max_{1 \le i \le d} \alpha_i > 1$}. In this case,  we choose $\gamma \in (0,1)$ such that $\a_i \gamma \le 1$ for every $i$ and consider 
	\begin{equation*}
		\norma{\bb x}^{\gamma}= \sup _{1 \le i \le d}|x_i|^{\a_i \gamma}
	\end{equation*}
	which is subadditive by our choice of $\gamma$. It holds for $j\in \{1,2\}$
	\begin{equation*}
		|Y_j|^\gamma ~\le \sum _{k=1}^{\8}\left (\Pi ^{(j)}_{k-1}\right )^{\gamma}\left (\norma{\bfB _k^{(j)}}^{\gamma}+1\right ) := Z_j, 
	\end{equation*}
	and hence $\left (\Pi ^{(j)}_n\right )^{\gamma}, \norma{\bb X_{\le n}^{(j)}}^{\gamma} \leq Z_j$.
%	 Indeed,
%	\begin{equation*}
%		|\sum _{i=1}^n \bfA ^{(j)}_1...\bfA ^{(j)}_{i-1}\bfB ^{(j)} _i| ^{\gamma} _{\a } \leq \sum _{i=1}^n |\bfA ^{(j)}_1...\bfA ^{(j)}_{i-1}\bfB ^{(j)} _i|^{\gamma} _{\a }  \leq \sum _{i=1}^n \left (\Pi ^{(j)}_{i-1}\right )^{\gamma} |\bfB ^{(j)} _i|^{\gamma} _{\a }\leq Y_j. 
%	\end{equation*}
We have
	\begin{align*}
		 &~ \{ \Pi ^{(1)}_n> t, \Pi ^{(2)}_n>\eps t, \norma{X_{\le n}^{(1)}}> Mt\ \mbox{or}\  \norma{\bb X_{\le n}^{(2)}}> M\eps t \} \\
		 = &~  \Big\{(\Pi ^{(1)}_n)^\gamma> t^\gamma, (\Pi ^{(2)}_n)^\gamma>(\eps t)^\gamma, \norma{X_{\le n}^{(1)}}^\gamma> (Mt)^\gamma\ \mbox{or}\  \norma{\bb X_{\le n}^{(2)}}^\gamma> (M \eps t) ^\gamma \Big\} \\
%		&\subset \left ( \Pi ^{2}_n> t, |X_n^{(1)}| _{\a}> Mt \right ) 
%		\cup \left ( \Pi ^{1}_n> t,  |X_n^{(2)}| _{\a}> Mt \right )\\
		\subset &~   \Big\{  Z_{2}> (\eps t)^{\gamma}, Z_1> (Mt)^{\gamma} \Big\} \cup \Big\{  Z_{1}> t^{\gamma}, Z_2> (M\eps t)^{\gamma} \Big\}.
	\end{align*}
Reasoning as before, we have	\begin{align*}
	\P \left (   Z_{2}> (\eps t)^{\gamma}, Z_1> (Mt)^{\gamma} \right )&\leq  \E [f(\eps M^{-\gamma}(\eps t)^{-\gamma}Z_1, (\eps t)^{-\gamma}Z_2)]= \E [f_{(M/\eps)^\gamma,1 }((\eps t)^{-\gamma} \bb Z)]\\[.2cm]
	\P \left (  Z_{1}> t^{\gamma}, Z_2> (M\eps t)^{\gamma}\right )&\leq  \E[ f(t^{-\gamma}Z_1, t^{-\gamma}M^{-\gamma}\eps ^{-\gamma}Z_2)]= \E [f_{1,(M\eps)^\gamma }(t^{-\gamma}\bb Z)].
\end{align*}
Observing that $\bb Z=(Z_1, Z_2)$ satisfies the equations
	\begin{equation*}
		Z_j \eqdist e^{\gamma U_j}Z_j+\left (\norma{\bfB ^{(j)}}^{\gamma}+1\right ), \quad j=1,2,
	\end{equation*}
	we may apply Proposition \ref{prop:gdri} to $\bb Z$ with $\wt {\bb \xi} = \gamma^{-1} \bb \xi^*$ and obtain
	\begin{align*}
		&~\limsup _{t\to \8}\ (\log t)^{1\slash 2}t^{\xi^* _1+\xi^* _2}\E \big[ f_{(M/\eps)^\gamma,1 }((\eps t)^{-\gamma} \bb Z)\big] \\ 
		 =&~ \limsup _{t\to \8}\ \bigg(\frac{\log t}{\gamma \log (\eps t)}\bigg)^{1\slash 2} \eps^{-\xi^* _1-\xi^* _2} (\gamma \log \eps t)^{1\slash 2}\big((\eps t)^\gamma\big)^{\wt \xi _1+\wt \xi _2}\E \Big[ f_{(M/\eps)^\gamma,1 }\Big(\big((\eps t)^\gamma\big)^{-1} \bb Y\Big) \Big] \\
		\leq&~ \eps^{-\xi^*_1 - \xi^*_2} \cdot C \big((M/\eps)^\gamma\big)^{-\wt \xi_1} =  C M^{-\xi^* _1}\eps ^{-\xi^* _2},		
	\end{align*}
	the second estimate being similar. This shows \eqref{eq:complement} in the case when $\norma{\cdot}$ is not subadditive; and concludes the proof.
\end{proof}

%{\blue ??? Argument for \eqref{eq:R}: quote structure of the support from Ewa's book. Negativity of the top lyapunov exponent follows from $\E |A_i|^\epsilon<0$, $i \in \{1,2\}$ by an argument given in Damek, Muneya, Swiatkowsky.
%
%add some sufficient conditions that guarantee that $A_1, A_2$ are larger 1 simultaneously (e.g. positive dependence) and that $\bb X$ is not contained on a line.

\begin{lemma} \label{lem:X:unbounded}
	Suppose there are $(\bb a, \bb b)$  in the semi-group generated by  $\supp (\bfA , \bfB )$ such that both $|a_1|>1$ and $|a_2|>1$, and let $\bb x_0:=(\Id - \bb a)^{-1} \bb b$. Suppose that there is $\bb x_1  \in \supp (\bb X)$ such that $(\bb x_1 - \bb x_0)^{(i)} \neq \bb 0$ for both $i \in \{1,2\}$.
		 Then for every $R>0$,
		$$	\P (|\bfX ^{(1)}|_{\a }>R, |\bfX ^{(2)}| _{\a }>R )>0.$$
\end{lemma}

\begin{proof} Let $(\bb a, \bb b)$ and $\bb x_0$ be as in the statement of the lemma. Note that $\bb x_0$ is the unique fixed point of the mapping $h: \bb x \mapsto \bb a \bb x + \bb b$.
	Observe that for every $\bb x$
	\begin{equation*}
		\bb a \bb x+ \bb b=\bb a( \bb x-\bb x_0)+\bb a\bb x_0 + \bb b = \bb a(\bb x - \bb x_0) + \bb x_0
	\end{equation*}
	and hence, upon iterating,
	\begin{equation*}
		h^n(\bb x)= \bb a^n (\bb x -\bb x_0) + \bb x_0
	\end{equation*}
	Applying this to $\bb x_1$, we have for both $i \in \{1,2\}$
	\begin{equation*}
	\norma{(h^n(\bb x))^{(i)}}=	\norma{\left (\bb a^n(\bb x_1-\bb x_0)\right )^{(i)}}=|A_i|^n \norma{(\bb x_1-\bb x_0)^{(i)}}\to \8,
	\end{equation*}
	as $n \to \infty$.
	Moreover, for every $n$, $h^n(\bb x_1)$ belongs to $\supp(\bb X)$ (see \cite[Proposition 4.3.1]{Buraczewski2016} for a proof) and so the conclusion follows. 
	
\end{proof}

\begin{remark}\label{rem:suff.cond.unbounded} Sufficient conditions for the assertions of Lemma \ref{lem:X:unbounded} are as follows.
	\begin{itemize}
		\item  The existence of $(\bb a, \bb b)$  in the semi-group generated by  $\supp (\bfA , \bfB )$ such that both $|a_1|>1$ and $|a_2|>1$ is guaranteed if we assume that  $\P(\exists_n\, |\Pi_n^{(1)}|>1, |\Pi_n^{(2)}|>1)>0$; in particular, it follows if we have that $\P(|A_1|>1, |A_2|>1)>0$. Since the latter assumption is easy to check on the input random variables, we refrain from studying further sufficient conditions.
		\item Concerning the existence of $\bb x_1$ that differs from $\bb x_0$ in both blocks, it is shown in \cite[Proposition 4.3.2]{Buraczewski2016}, extending \cite{Alsmeyer2009}, that under our assumptions, the law of $\bb X$ does not have atoms and is of pure type, i.e., either singular or absolutely continuous with respect to Lebesgue measure. In the latter case, $\supp(\bb X)$ contains an open set and the existence of $\bb x_1$ follows. As a consequence of \cite[Proposition 4.3.1]{Buraczewski2016}, simple sufficient conditions for the absolute continuity of the law of $\bb X$ are for example that 
		\begin{enumerate}
			\item $\bb B$ is constant and the law of $(A_1, A_2)$ is nonsingular w.r.t. Lebesgue measure on $\R^2$, or
			\item $\bb B$ is independent of $\bb A$ and is nonsingular w.r.t. Lebesgue measure.
		\end{enumerate}
	\end{itemize}	
\end{remark}

\begin{proposition}\label{thm:nu:unbounded}
	The measure $\nu $ is unbounded.
\end{proposition}
\begin{proof}
	For every $\delta>0$ consider a nonnegative function $f_\delta \in \bb H^\epsilon(\R^d \setminus [\mathsf{blocks}])$ such that
$$ \Ind{\{ \norma{\bb x^{(1)}} \geq 1/2, \norma{\bb x^{(2)}}\geq \d/2 \}} \ge f_\delta \ge\Ind{\{   \norma{\bb x^{(1)}} \geq 1, \norma{\bb x^{(2)}}\geq \d \}}.$$
	 Then, by \eqref{eq:positivity}, for all sufficiently large $t$,
	\begin{equation*}
		t^{\xi^* _1 +\xi^* _2}(\log t)^{1\slash 2}\E [f_{\d}(t^{-1\slash \bb \a } \bfX)]\geq t^{\xi^* _1 +\xi^* _2}(\log t)^{1\slash 2} \P \left (  \norma{\bb X^{(1)}}>t, \norma{\bb X^{(2})}>\d t \right ) \ge c \delta^{-\xi^*_2}
	\end{equation*}
	for a constant $c$ that is independent of $\delta$. On the other hand, by \eqref{eq:measure4}
	\begin{align*}
		&~c\d ^{-\xi^* _2}\leq \lim _{t\to \8}t^{\xi^* _1 +\xi^* _2}(\log t)^{1\slash 2}\E [f_{\d}(t^{-1\slash {\bb \a} } \bfX)]\\
		=&~\int _{\wt S_1}\int _{0}^{\8 }f_{\d }(s^{1\slash {\bb \a}}\omega)\frac{1}{s^{\xi^* _1+\xi^* _2+1}}d\nu (\omega)\\
		\leq&~ \int _{\wt S_1} \int _{1\slash 2}^{\8 }\frac{1}{s^{\xi^* _1+\xi^* _2+1}}d\nu (\omega)= (\xi^* _1+\xi^* _2)^{-1}2^{\xi^* _1+\xi^* _2}\nu (\wt S_1)
	\end{align*}
	where we used that $f_\delta \neq 0$ requires % (using $\norma{\omega^{(1)}} \le \norma{\omega}=1$)
	$$s \ge s \norma{\omega^{(1)}} =  \norma{(s^{1\slash \bb \a}\omega )^{(1)}} \ge 1/2.$$ Letting $\eps \to 0$ we see that $\nu $ is unbounded.
\end{proof}

\begin{proof}[Proof of Theorem \ref{thm:main:blocks} - continued]
	We were left with the proof that $\Lambda_2$ and $\nu$ are nonzero, and that $\nu$ is unbounded under assumption \eqref{eq:suppX:unbounded}. The nontriviality of $\Lambda_2$ and hence of $\nu$ follows from \eqref{eq:positivity}, which gives (with $\eps=1$) a lower bound for $\int f(\bb x)\Lambda_2(d \bb x)$ for a continuous function $f$ such that
	$ f \geq \Ind{\{\norma{x^{(1)}}>1, \norma{x^{(2)}}>1\}}.$
	The unboundedness of $\nu$ is the content of Proposition \ref{thm:nu:unbounded}.
\end{proof}

	\subsection{Key Lemma}\label{sec:key}
	In this Section we provide the proof of Proposition \ref{thm:keyasym}. It will be based on the following Lemma where we employ an Edgeworth expansion to obtain for a fixed $n$ in a suitable range the probability that $\bb S_n$ attains values close to $\log t$. 

Let  $\bb \xi^*$ be as in Theorem \ref{thm:main:blocks}. Then $\E_{\bb \xi^*} [\bb U]$ is parallel to (1,1), hence there is $\rho>0$ such that $\E _{\bb \xi^*} [U_i]=\rho $ for $i=1,2$. Denoting $\s ^2_i=\E _{\bb \xi^*} \left ( U_i-\rho \right )^2 $, consider the standardized variables  $V_i=\left ( U_{i}-\rho \right )\slash \s _i$, $i=1,2$ and 
write $\Sigma_{\bb V}$ for the covariance matrix of $\bb  V=(V_1, V_2)$.
By Assumption \eqref{assump:cramer} it holds in particular that $\mu _{\bb \xi^*}$ is nonarithmetic.
% i.e. $|\hat \mu (t)|< 1$ for $t\neq 0$. Indeed, if $|\hat \mu (t)|=1$ then for every integer $n$, $|\hat \mu (nt)|=1$. 
Hence, the covariance matrix $\Sigma_{\bb V}$ is strictly positive definite (otherwise, $(V_1, V_2)$ would be supported on a line).

Let
\begin{equation}\label{eq:covariance}
	h (\bb w)= (2\pi )^{-1}(\det \Sigma_{\bb V})^{-1\slash 2}e^{-\langle \bb w,\Sigma_{\bb V}^{-1} \bb w\rangle }.
\end{equation}
%$\mathbf{Cov}$ being the covariance matrix of $V^1, V^2$.
%We want to prove that 
%\begin{equation*}
%	n\iint _{-kM\rho < \s _i\sqrt{n}u^i < D -kM\rho} q_n(u)\ du\geq c
%\end{equation*}for all $k$. 
Observe that there are $c_1,c_2,c_3$ such that
\begin{equation}\label{eq:bound:psi}
	c_3 e^{-c_1 \| \bb w\| ^2}\leq h (u)\leq  c_3 e^{-c_2 \| \bb w\| ^2}.
\end{equation}
The preparatory lemma will be formulated in terms of the following nontrivial constants.
\begin{align*}
	c(\rho, \s _1, \s _2):=& (\rho +1) ^2(\s _1^{-2}+\s _2^{-2}),\\
	c:=&\frac{1}{2}c_3 \exp({-c_1 c(\rho, \s _1, \s _2)}) (\s _1\s _2)^{-1}e^{-\xi^* _1-\xi^* _2},\\
	c':=&(c_3+c_4)(\s _1\s _2)^{-1}+1.
\end{align*}

\begin{lemma}\label{lem:application.Edgeworth}
	%Suppose Cramer condition \eqref{eq:Cramer}
	%the covariance matrix $K$ in \eqref{eq:covariance} is strictly positive definite.
	Under the assumptions of Theorem \ref{thm:main:blocks}, let $n_0 = \lceil \log t \slash \rho \rceil $. %given {\color{blue} $0<D<\rho M$}, 
	%there are $c, c'>0$ such that 
	For every $\eps $ there is $t_0=t_0(\eps )$ such that for every $\ell \leq \sqrt{n_0}$ and $n=n_0+\ell$ we have
	\begin{equation}\label{eq:below}
		\P \left ( n_0\rho <S _{n,1}< n_0 \rho +1, n_0+\log \eps  <S _{n,2}< n_0 \rho +\log \eps +1\right )\geq c \eps ^{-\xi^* _2} t^{-\xi^*_1-\xi^*_2}n_0^{-1}.
	\end{equation}
	and
	\begin{equation}\label{eq:above}
		\P \left ( n_0\rho <S _{n,1}< n_0 \rho +1, n_0\rho +\log \eps <S _{n,2}< n_0 \rho +\log \eps +1\right )\leq c' \eps ^{-\xi^* _2} t^{-\xi^*_1-\xi^*_2}n_0^{-1}.
	\end{equation}
\end{lemma}
%\begin{rem}
%	At the moment I use theorem 19.2 of the book Bahatt. and density of $W_n$ below, I need an extra assumption that the characteristic function of $U$ is in some $L^p$. This should be removed by using distribution function which I do not understand yet.
%\end{rem}
\begin{proof}
	We abbreviate
		\begin{equation*}
		I_n:=\left \{ n_0\rho <S _{n,1}< n_0 \rho +1,\ n_0\rho +\log \eps <S _{n,2}< n_0 \rho +\log \eps +1\right \}.
	\end{equation*}		
	To estimate the probabilities in \eqref{eq:below} and \eqref{eq:above}, we will use an Edgeworth expansion (see \eqref{eq:Edge}) with respect to the shifted measure $\P_{\bb \xi^*}$, under which $\bb S_n$ has drift $\rho (1,1)$. We introduce for $i=1,2$ the sums
	%We change the measure $\wt \mu =e^{\langle \xi^*,u\rangle}\mu (u)$. Let
	%$U_j=(U^1_j,U^2_j)$. Then $\E _{\wt \mu} U^i=\rho $. Let $\s ^2_i=\E _{\wt \mu} \left ( U^i-\rho \right )^2 $
	%and $V^i_j=\left ( U^i_j-\rho \right )\slash \s _i$.
%	We are going to use \eqref{eq:Edge}. Let ,  
	\begin{equation} \label{eq:Wn}
		W_{n,i}:=\frac{S_{n,i}-n\rho }{\s _i\sqrt{n}}
	\end{equation}
	which are standardized under $\P_{\bb \xi^*}$ (note that $W_1^i=V_i$).
	Using that $n=n_0+\ell$, we can rewrite this definition as
	\begin{equation*}
	\s _i\sqrt{n}W_{n,i}=S_{n,i}-n\rho = S_{n,i}-n_0\rho -\ell \rho 
\end{equation*}
to  obtain
\begin{equation*}
	I_n= \{ -\ell \rho < \s _1\sqrt{n}W_{n,1} < 1 -\ell \rho ,\ -\ell \rho +\log \eps < \s _2\sqrt{n}W_{n,2} < 1+\log \eps -\ell \rho \}.
\end{equation*}
	On $I_n$ we  have 
	$  0<\ell\rho + \s _1\sqrt{n}W_{n,1} < 1 $ and 
	$ \log \eps <\ell\rho + \s _2\sqrt{n}W_{n,2} <\log \eps +1$. Hence on $I_n$,
	\begin{equation*}
		\langle \bb \xi^* , \mathbf{S_n} \rangle =	 \langle \bb \xi^*,n_0(\rho ,\rho )\rangle +
		\langle \bb \xi^*, \ell (\rho , \rho )+ \sqrt{n}( \s _1W_{n,1}, \s _2W_{n,2}) \rangle \leq 
		(\xi^* _1+\xi^* _2)n_0\rho + \xi^* _1+\xi^* _2+ \xi^* _2\log \eps
	\end{equation*}
	and so
	\begin{align*}
		%&\P \left ( n_0\rho <S _{n,1}< n_0 \rho +D, n_0\rho <S _{n,2}< n_0 \rho +D\right )\\
		\P (I_n)&=\E _{\bb \xi^*}\big[e^{-\langle \bb \xi^*,\mathbf{S_n} \rangle}\Ind {I_n}\big] 
		\geq  t^{-\xi^*_1-\xi^*_2}e^{-(\xi^*_1+\xi^*_2)}\eps ^{-\xi^* _2} 
		\P _{\bb \xi^*}(I_n)%\iint _{-kM\rho < \s _i\sqrt{n}u^i < D -kM\rho} q_n(u)\ du,
	\end{align*}
	In order to prove \eqref{eq:below}, recall that $n=n_0 +\ell$ with $\ell \leq \sqrt{n_0}$. It thus suffices to prove that for sufficiently large $n$ (or equivalenty, sufficiently large $t$), there is $c>0$ such that
	\begin{equation}\label{eq:belowI}
		n\P _{\bb \xi^*}(I_n)\geq c.
	\end{equation}
	
	Let
	\begin{equation*}
		C_n= \{ (w_1, w_2): -\ell \rho < \s _1\sqrt{n}w_1 < 1 -\ell\rho , -\ell\rho +\log \eps  < \s _1\sqrt{n}w_2 < 1 -\ell\rho +\log \eps \}.
	\end{equation*}
	%where $q_n $ is the density of the law of $W_n$.
	Since $\bb \xi^* \in \interior{I}$, it holds that $\E _{\bb \xi^*}[\| \mathbf{U}\| ^4]<\8$ and by the Edgeworth expansion \eqref{eq:Edge}, we have
	\begin{equation*}
		%q_n (u)= \psi (u)+o(1), %n^{-1\slash 2}W_1(u)\psi (u)+n^{-1}W_2(u)\psi (u)+ o (n^{-1}).
		n \P _{\bb \xi^*}(I_n) = n \P _{\bb \xi^*}(\bb W_n \in C_n) = n \iint _{C_n} h (\bb w)\  d \bb w + nH_2 (C_n)+ o(1),
	\end{equation*}
	uniformly for all $n$ and $C_n$ where  $H_2 $  is the measure with density 
	$n^{-1\slash 2}P_1 \psi +n^{-1}P_2 \psi   $ for certain polynomials $P_1, P_2$  as in \eqref{eq:Edge}.  
	But on $C_n$, for sufficiently large $n$ (and hence $t$), we have 
	%the set ${-kM\rho < \s _i\sqrt{n}u^i < D -kM\rho}$, for $n$ such that $D\leq \rho M$,  we 
	\begin{align*}
		%c(\rho, \s _1, \s _2)\leq 
		\| \bb w\| ^2~=&~ w_1^2 + w_2^2 \le \frac{(1+\ell\rho)^2}{\sigma_1^2 n} + \frac{(1+\ell \rho+ \log \eps)^2}{\sigma_2^2 n}\\
		\leq&~ (\rho +1) ^2(\s _1^{-2}+\s _2^{-2})= c(\rho, \s _1, \s _2)
	\end{align*}
	because $0\leq \frac{\ell}{\sqrt{n}}\leq 1$ and $\frac{\abs{\log \eps}}{\sqrt{n}}\leq 1$ if $n$ (or equivalently, $t$) large enough. 
	%In fact $c(\rho, \s _1, \s _2)=\rho ^2(\s _1^{-2}+\s _2^{-2})$.
	%c(\rho, \s _1, \s _2, D)=\frac{1}{2}c(\rho, \s _1, \s _2)$ provided $\frac{D}{\sqrt{n}}\leq \frac{\rho}{2}$.
	Hence, using \eqref{eq:bound:psi} and that the Lebesgue measure of the set $C_n$
	%$\frac{-kM\rho}{\s _i\sqrt{n}} < u^i <\frac{D}{\s _i\sqrt{n}} -\frac{kM\rho}{\s _i\sqrt{n}}$ 
	is $\frac{1}{\s _1 \s _2n}$, we have
	\begin{equation*}
		%n\iint _{-kM\rho < \s _i\sqrt{n}u^i < D -kM\rho} 
		n\iint _ {C_n}h (\bb w)\ d\bb w\geq n c_3 e^{-c_1 c(\rho, \s _1, \s _2)}  \frac{1}{n} (\s _1\s _2)^{-1}=c_3 e^{-c_1 c(\rho, \s _1, \s _2)}  (\s _1\s _2)^{-1}.
	\end{equation*}
	Similarly,
	\begin{align*}
		nH_2(C_n)&\leq n\cdot n^{-1\slash 2}\iint _{C_n} (P_1(\bb w)+P_2(\bb w))\psi (u)\ d\bb w\\
		&\leq n\cdot n^{-1\slash 2}\iint _{C_n} c_4 \  d \bb w \\
		%n\iint _{-kM\rho < \s _i\sqrt{n}u^i < D -kM\rho} %W_i(u) \psi (u)C(W_i) c_3 e^{-c_2 c(\rho, \s _1, \s _2, D)}\ du 
		&\leq c_4 n^{1\slash 2}  \frac{1}{n}(\s _1\s _2)^{-1}\leq c_4 n^{-1\slash 2} (\s _1\s _2)^{-1} \to 0
	\end{align*}
	Hence \eqref{eq:belowI} follows.
Similarly, we obtain the upper estimate 
\begin{align*}
	%&\P \left ( n_0\rho <S _{n,1}< n_0 \rho +D, n_0\rho <S _{n,2}< n_0 \rho +D\right )\\
	\P (I_n)\leq \eps ^{-\xi^* _2} e^{-\langle \bb \xi^*,n_0(\rho ,\rho )\rangle} \P _{\bb \xi^*}(I_n)
	% \E _{\bb \xi^*} \Ind {-kM\rho < \s _i\sqrt{n}W_n^i < D -kM\rho ,\ i=1,2}\\
	\leq \eps ^{-\xi^* _2}t^{-\xi^*_1-\xi^*_2} n^{-1} c'
	%\left (c_3 D^2 (\s _1\s _2)^{-1}+\wt C D^2n^{-1\slash 2}+o(1)\right )
\end{align*}
from which \eqref{eq:above} follows.
\end{proof}
\begin{proof}[Proof of Proposition \ref{thm:keyasym}]
	Observe that the estimate \eqref{eq:below} considers a single $n=n(\ell)=n_0+\ell$ and is of order $n_0^{-1}$. The desired estimate in \eqref{eq:2piestim} is of order $n_0^{-1/2}$ and considers the existence of some $n$. The obvious idea now it to take the union of the probabilites estimated in \eqref{eq:below} over the range $\ell \in \{1, \dots \sqrt{n_0}\}$, using the inclusion-exclusion principle. Then the leading term (sum of all probabilities) has the right order. To control the second term (sum of probabilites of intersections) we need however to introduce a {\em grid size} $M$, that is, we will only sum over $\ell \in \{kM \, : \, kM \leq \sqrt{n_0}\}$. This will give us enough control to bound the probability of the intersections of events.

%Suppose first that $K$ is positive definite.
%Choose $\zeta $ such that $\beta=\phi (\zeta )<1$ and $M$ such that
%\begin{equation*}
%\frac{\eta ^{M}}{1-\eta ^M}c'\leq \frac{1}{2}c.
%\end{equation*}
For the subsequent calculations, it will be convenient to index the sets $I_{n(k)}$ with $k$, using the relation $n=n_0+Mk$. Therefore, denote
$$I_k:= I_{n(k)} = \left \{  n_0\rho <S _{n,1}< n_0 \rho +1, n_0\rho +\log \eps <S _{n,2}< n_0 \rho +\log \eps +1\right \}.$$ 
%Since $n\slash n_0\to 1$, we have
By \eqref{eq:below} we have
\begin{equation*}
	\sum _{k=1}^{\sqrt{n_0}/M} \P (I_k)\geq \frac{\sqrt{n_0}}{M}c\eps ^{-\xi^* _2}t^{-\xi^* _1-\xi^* _2}n_0^{-1}
\end{equation*}
for $t\geq t(\eps )$. Choose $\bb \zeta \in [0,1]^2$ such that $\beta=\phi (\bb \zeta )<1$ and $M$ such that
\begin{equation*}
	\frac{\beta ^{M}}{1-\beta ^M}c'e^{\zeta _1+\zeta _2}\leq \frac{1}{4}c.
\end{equation*} 
We shall prove that  
\begin{equation}\label{eq:sum:kl}
	\sum _{\begin{subarray}{c}
			k,l=1\\ k \neq l
		\end{subarray}}^{\sqrt{n_0}/M} \P (I_k\cap I_l)\leq 2 \frac{\beta ^{M}}{1-\beta ^M}\frac{\sqrt{n_0}}{M}
	c'\eps ^{-\xi^* _2}t^{-\xi^*_1-\xi^*_2}n_0^{-1}e^{\zeta _1+\zeta _2} \quad \bigg( \le \frac12 c \eps^{-\xi^*_2} t^{-\xi^*_1-\xi^*_2} M^{-1} n_0^{-1/2} \bigg).
\end{equation}
Then for $t\geq t (\eps )$ %taking $M$ large enough (provided $n_0$ is large enough), we have
\begin{equation*}
	\P \left (\bigcup _k I_k\right )\geq \sum _k \P (I_k)-\sum _{k\neq l} \P (I_k\cap I_l)\geq \frac{c}{2}\eps ^{-\xi^* _2}M^{-1}t^{-\xi^*_1-\xi^*_2}n_0^{-1\slash 2}
\end{equation*}
and \eqref{eq:2piestim} will follow since
$$ t^{\xi^*_1+\xi^*_2}(\log t)^{1\slash 2}	\P \left ( \exists\, n \, : \,  \Pi ^{(1)}_n>t, \Pi ^{(2)}_n> \eps t \right ) \geq t^{\xi^*_1+\xi^*_2}(\log t)^{1\slash 2}\P \left (\bigcup _k I_k\right ) \geq \frac{c}{2}\eps ^{-\xi^* _2}M^{-1} $$

\textbf{Proof of \eqref{eq:sum:kl}}.
For $k<l$ and sufficiently large $t$, we use that on $I_k \cap I_l$, we have control over the increments $\bb S_{n_0+l M} - \bb S_{n_0 + kM}$, as follows. The crucial point is that by enlarging the grid size $M$, the probability of relatively large increments decreases exponentially fast. 
\begin{align}
	&~\P (I_k\cap I_l) \notag\\
	\leq&~ \P \left (t<\Pi_{n_0+kM}^{(1)}<te,\ \eps t<\Pi_{n_0+kM}^{(2)}<\eps te,\  t<\Pi_{n_0+lM}^{(1)},\ \eps t<\Pi_{n_0+lM}^{(2)}  \right ) \notag \\
	\leq&~ \P \left (t<\Pi_{n_0+kM}^{(1)}<te,\ \eps t<\Pi_{n_0+kM}^{(2)}<\eps te,\  \exp{(S_{n_0+lM,i}-S_{n_0+kM,i})}>e^{-1},\ i=1,2 \right )	\notag\\
	=&~ \P \left (t<\Pi_{n_0+kM}^{(1)}<te,\ \eps t<\Pi_{n_0+kM}^{(2)}<\eps te \right) \P \left(\Pi_{(l-k)M}^{(i)}>e^{-1},\ i=1,2 \right) \label{eq:PIkIl}	
\end{align}
%\begin{align}
%	&~\P (I_k\cap I_l) \notag\\
%	\leq&~ \P \left (t<\exp({S_{n_0+kM}^{(1)}})<te,\ \eps t<\exp({S_{n_0+kM}^{(2)}})<\eps te,\  t<\exp({S_{n_0+lM}^{(1)}}),\ \eps t<\exp({S_{n_0+lM}^{(2)}})  \right ) \notag \\
%	\leq&~ \P \left (t<\exp({S_{n_0+kM}^{(1)}})<te,\ \eps t<\exp({S_{n_0+kM}^{(2)}})<\eps te,\  \exp{(S_{n_0+lM}^{(i)}-S_{n_0+kM}^{(i)})}>e^{-1},\ i=1,2 \right )	\notag\\
%	=&~ \P \left (t<\exp({S_{n_0+kM}^{(1)}})<te,\ \eps t<\exp({S_{n_0+kM}^{(2)}})<\eps te \right) \P \left(\exp{(S_{(l-k)M}^{(i)})}>e^{-1},\ i=1,2 \right) \label{eq:PIkIl}	
%\end{align}
The first probability in \eqref{eq:PIkIl} is bounded by \eqref{eq:above}, for the second probability we use the Markov inequality as follows. Let $\bb \zeta$ be as above.
\begin{align*}
	&~\P \left(\exp{(S_{(l-k)M,i}}>e^{-1},\ i=1,2 \right) = \P \left(\exp{(\zeta_i S_{(l-k)M,i})}>e^{-\zeta_i},\ i=1,2 \right) \\
	\leq &~ \P \left(\exp{(\zeta_1 S_{(l-k)M,1}+ \zeta_2 S_{(l-k)M,2})}>e^{-\zeta_1-\zeta_2},\ i=1,2 \right) \\
	\leq&~  \phi (\bb \zeta )^{(l-k)M}e^{\zeta _1+\zeta _2}.
\end{align*}
We have thus obtained (abbreviating $\beta=\phi(\bb \zeta)<1$). 
$$ \P (I_k\cap I_l) \leq \frac{c'}{n_0}\eps ^{-\xi^* _2}t^{-\xi^*_1-\xi^*_2} \beta^{(l-k)M}e^{\zeta _1+\zeta _2}$$
Considering \eqref{eq:sum:kl}, we first bound for any fixed $k$ the sum over $l>k$ using a geometric series.
\begin{align*}
	\sum _{l>k}\P (I_k\cap I_l)&\leq \frac{c'}{n_0}\eps ^{-\xi^* _2}t^{-\xi^*_1-\xi^*_2} e^{(\zeta _1+\zeta _2)} \sum _{l>k}\beta ^{(l-k)M}
	\leq \frac{c'}{n_0}\eps ^{-\xi^* _2}t^{-\xi^*_1-\xi^*_2}\beta ^{M}\left (1-\beta ^M\right )^{-1}e^{(\zeta _1+\zeta _2)}
\end{align*}
and then estimate the sum over $k$ by just multiplying with the number of summands.
\begin{equation*}
\sum _{\begin{subarray}{c}
		k,l=1\\ k \neq l
\end{subarray}}^{\sqrt{n_0}/M} \P (I_k\cap I_l) = 2	\sum _{\begin{subarray}{c}
k,l=1\\ k < l
\end{subarray}}^{\sqrt{n_0}/M} \P (I_k\cap I_l)\leq 2 \frac{c'}{n_0}\frac{\sqrt{n_0}}{M}\eps ^{-\xi^* _2}t^{-\xi^*_1-\xi^*_2}\beta ^{M}\left (1-\beta ^M\right )^{-1}e^{(\zeta _1+\zeta _2)}
\end{equation*}
This proves \eqref{eq:sum:kl} and finishes the proof of the Theorem.
\end{proof}

\section{The Renewal Theorem on $\R^2 \times K$}\label{sect:2dRenewal.with.K}

In this section, we prove a renewal theorem for random walks on $\R^2 \times K$, where $K$ is a finite Abelian group with unit $e$. In our setting, $K$ is a subgroup of $\Z_2^d$.

Let  $\wt Q $ be a probability measure on the group $G=\R ^2\times K=\R ^2 \times \Z _2^d$.
We define $Q $ on $\R ^2$ by $Q (I)=\wt Q (I\times K)$. Let $\bb \rho  =(\rho _1, \rho _2)$ be the mean of $Q $.
%	As in Section \ref{sec:key}, $m_{\theta}=(\rho , \rho)$. 
Let $\mathbb{U}=\sum _{n=0}^{\8}Q ^n$, $ \mathbb{\wt U}=\sum _{n=0}^{\8}\wt Q ^n$. 

For $\bfm =(\bfm_1,\bfm_2)\in \Z ^2$, let $F_{\bfm}=\{ \mathbf{s}\in \R^2: \bfm _j\leq s_j\leq \bfm_{j}+1, j=1,2\}$.
For a continuous function $\psi $ on $G$ let 
\begin{equation*}
	\| \psi \| _{\dri} =\sum _{k\in K}\sum _{\bfm}\sup _{\mathbf{s}\in F_{\bfm}}\left ( (1+\|\mathbf{s}\|)|\psi (\mathbf{s},k)|\right )
\end{equation*}

\begin{theorem}\label{thm:renewalextended}
	Suppose that the closed subgroup generated by $\supp(\wt{Q})$ is equal to $G=\R^2 \times K$ and that $\bb \rho = \int \bb x \mu(d\bb x) \neq \bb 0$. Assume further that $\int \norm{ \bb x}^2 \, Q(d \bb x) < \infty$.
%	Assume further that either
%	\begin{enumerate}
%		\item $\int \norm{ \bb x}_2^2 \, Q(d \bb x) < \infty$ {\red (and nonlattice, but this holds already since $\supp(Q)$ generates $\R^2$.)}, or
%		\item $\int \norm{ \bb x}_2^5 \, Q(d \bb x) < \infty$ and 
%		$$ \limsup_{\norm{\bb t} \to \infty} \abs{\hat{Q}(\bb t)} <1. $$
%	\end{enumerate}
%	{\red 1. is with Carlsson; 2. with Edgeworth - own proof}
	%for any bounded set $I$
	%\begin{equation*}	\lim _{r\to \8}r^{1\slash 2}\wt U(rm+I)=c\lambda ^2(I).
	%\end{equation*}
	Then there is $c>0$ such that for every continuous function $f $ with  $\| f \| _{\dri}<\8 $ and every $k_1 \in K$ we have
	\begin{equation}\label{eq:simple}
		\lim _{r\to \8}r^{1\slash 2}\, f *\mathbb{\wt U}(r\bb \rho ,k_1)=c\sum _{k\in K}\int _{\R ^2}f (\mathbf{s},k) \ d\mathbf{s}.
	\end{equation}
	
\end{theorem}

\begin{remark}
	We will apply Theorem \ref{thm:renewalextended} for $\wt{Q}=\wt{\mu}_{\bb \xi^*}$. Then assumption \eqref{assump:cramer} implies that $\supp(\mu_{\bb \xi^*})$ generates $\R^2$ as an additive group, and $K$ is chosen as the subgroup of $\Z_2^d$ that is generated by $\supp(\bb K)$. Hence it holds that  the closed subgroup generated by $\supp(\wt{\mu}_{\bb \xi^*})$ is equal to $G=\R^2 \times K$ and the theorem is applicable; the finiteness of moments being guaranteed since $\bb \xi^* \in \interior{I}$.
\end{remark}

\begin{proof}
	Let $W$ be a bounded measurable subset of $G$. Then there is a bounded set $I\subset \R ^2$ such that $W\subset I\times K$ and $\partial I$ has Lebesgue measure $0$. Then
	\begin{equation*}
		r^{1\slash 2}\mathbb{\wt U} ((r\bb \rho,k)W)\leq r^{1\slash 2}\mathbb{\wt U} ((r\bb \rho,k)(I\times K))\leq r^{1\slash 2}\mathbb{\wt U} ((r\bb \rho+I)\times K)= r^{1\slash 2} \mathbb{U} (r\bb \rho+I)
	\end{equation*}
	is bounded, by the (classical) renewal theorem on $\R^2$ (see Theorem \ref{thm:Stam1}), as a function of $r$. Observe that
	\begin{equation*}
		\mathbb{\wt U} ((r\bb \rho,k)W)=\delta _{(-r\bb \rho,k)}*\mathbb{\wt U}(W).
	\end{equation*}
	Hence the set of measures $\delta _{(-r\rho,k)}*\mathbb{\wt U}$ is weakly relatively compact which means that from any sequence $r_m\to \8 $ we can extract a subsequence $r_n$ such that for every $f\in \mathcal{C}_c(G)$
	\begin{equation*}
		\lim _{r_n}	r_n^{1\slash 2}\,f *\delta _{(-r_n\rho,k)}*\mathbb{\wt U}
	\end{equation*}
	exists and defines a Radon measure $\nu $. Since $\mathbb{\wt U}*\wt Q = \mathbb{\wt U}  -\d _{(0,e)}$ and $\lim_{r_n} r_n^{1/2} f * \delta(-r_n \rho,k)= \lim_{r_n} r_n^{1/2} f(-r_n \rho,k)=0$ we have  
	\begin{equation*}
		f * \nu = \lim _{r_n}	r_n^{1\slash 2}f *\delta _{(-r_n\rho ,k)}*(\mathbb{\wt U} - \delta_{(0,e)}) = \lim _{r_n}	r_n^{1\slash 2}f *\delta _{(-r_n\rho ,k)}*\mathbb{\wt U} *\wt Q = f *\nu *\wt Q
	\end{equation*}
	and so $\nu = \nu *\wt Q.$
	Therefore, by the Choquet-Deny Theorem for Abelian groups \cite{Choquet-Deny1960}, $\nu $ is a Haar measure on $G$, that is $\nu = cd\lambda ^2\times dk$ for some constant $c$, with $\lambda^2$ denoting the Lebesgue measure on $\R^2$ and $dk$ the counting measure on $K$. Evaluating
	$$ \lim _{r_n}	r_n^{1\slash 2}\,f *\delta _{(-r_n\rho,k)}*\mathbb{\wt U} ~=~ c \sum_{k \in K} \int f(\bb s,k) d\bb s$$
	for a function $f(\bb s,k)=f(\bb s)$ that does not depend on $k$, we may use the two-dimensional limit theorem (see Theorem \ref{thm:Stam1}), to obtain the value $c$ which hence does not depend on the sequence. Hence all subsequential limits coincide and \eqref{eq:simple} follows for $f \in \Cf_c(G)$.
	
	%\eqref{eq:simple} holds for a simple function with compact support instead of %$\psi $. 
	Consider now a continuous function $f$ such that $\norm{f}_{\dri}<\infty$. This implies that for all $\eps>0$ there is a finite set $M_\eps \subset \Z^2$ such that
	$$\sum _{k\in K}\sum _{\bfm \in M_\eps^c}\sup _{\mathbf{s}\in F_{\bfm}}\left ( (1+\|\mathbf{s}\|)|f (\mathbf{s},k)|\right )<\frac{\eps}{3 C}$$
	where $C$ is the constant from Lemma \ref{lem:dRi.proof.of.Renewal.Theorem}. 
	Hence we can decompose $f$ in such a way that
	\begin{equation*}
		f =f _1 +f _2, \quad \mbox{where}\ f_1\in \Cf_c(G),\ \| f_2 \| _{\dri}\leq \eps/3 .  
	\end{equation*}
	By Lemma \ref{lem:dRi.proof.of.Renewal.Theorem} we obtain
	\begin{align*}
		&~ \abs{r^{1\slash 2}\,f *\delta _{(-r\rho,k)}*\mathbb{\wt U} - c \sum_{k \in K} \int f(\bb s,k) d\bb s} \\
		\le &~ \abs{r^{1\slash 2}\,(f-f_1) *\delta _{(-r\rho,k)}*\mathbb{\wt U}} + \abs{r^{1\slash 2}\,f_1 *\delta _{(-r\rho,k)}*\mathbb{\wt U} - c \sum_{k \in K} \int f_1(\bb s,k) d\bb s} \\
		&~ + \abs{c \sum_{k \in K} \int (f-f_1)(\bb s,k) d\bb s} \\
		\le&~ C \norm{f_2}_{\dri} + \abs{r^{1\slash 2}\,f_1 *\delta _{(-r\rho,k)}*\mathbb{\wt U} - c \sum_{k \in K} \int f_1(\bb s,k) d\bb s} + C \norm{f_2}_{\dri} \le \frac{\eps}{3} + \frac{\eps}{3} + \frac{\eps}{3}
	\end{align*}
	for all sufficiently large $r$. This proves \eqref{eq:simple} for any $f \in \mathcal{C}(G)$ with $\norm{f}_{\dri}<\8$. 
\end{proof}

\begin{lemma}\label{lem:dRi.proof.of.Renewal.Theorem} There is a constant $C$ such that for every $f \in \mathcal{C}(G) $	\begin{equation}\label{eq:bound:dRi.rho}
		\sup_{r>0,k \in K} r^{1\slash 2}\, \abs{f}  *\mathbb{\wt U}(r\bb \rho,k)\leq C \| f \| _{\dri}.
	\end{equation}
	as well as
	\begin{equation}\label{eq:bound:dRi.any}
		\sup_{\bb s \in \R^2,k \in K} \abs{f}  *\mathbb{\wt U}(\bb s,k)\leq C \| f \| _{\dri}.
	\end{equation}
\end{lemma}

\begin{proof} We concentrate first on proving \eqref{eq:bound:dRi.rho}.
	Let $c_m= \inf _{\mathbf{s}\in F_{\bfm}}(1+\| \mathbf{s}\|)$. We have
	\begin{equation*}
		|f (\mathbf{s},k_1)|\Ind {F_{\bfm}\times K}\leq c_{\bfm}^{-1}\left (\sum _{k\in K}\sup _{\mathbf{s}\in F_{\bfm}}\left ( (1+\|\mathbf{s}\|)| f (\mathbf{s},k)|\right )\right )\Ind {F_{\bfm}\times K}.
	\end{equation*}
	Since $\Ind {F_{\bfm}\times K}*\mathbb{\wt U}= \Ind {F_{\bfm}}*\mathbb{U}$, we have 
	\begin{align}
		r^{1\slash 2}\, \abs{f} *\mathbb{\wt U}(r\bb \rho,k_1)&\leq \sum _{\bfm} r^{1\slash 2}c_{\bfm}^{-1}\left (\sum _{k\in K}\sup _{\mathbf{s}\in F_{\bfm}}\left ( (1+\| \mathbf{s}\|)
		\abs{f (\mathbf{s},k)}\right )\right )\int _{\R ^2} \Ind {F_{\bfm}}(r\bb \rho -u)d\mathbb{U}(u) \notag \\
		&=\sum _{\bfm} \sum _{k\in K}\sup _{\mathbf{s}\in F_{\bfm}}\left ( (1+\| \mathbf{s}\|)\abs{f (\mathbf{s},k)}\right )r^{1\slash 2}c_{\bfm}^{-1}\mathbb{U} (r\bb \rho-F_{\bfm}). \label{eq:bound.by.dri}
	\end{align}
	It is enough to prove that
	\begin{equation*}
		\sup _{r>0,\bfm}r^{1\slash 2}c_{\bfm}^{-1}\mathbb{U}(r\rho-F_{\bfm})<\8 .
	\end{equation*}
	Consider fixed $D>0$, to be chosen below. For $\bfm $ such that $r^{1\slash 2}c_{\bfm}^{-1}\leq D$ we have 
	\begin{equation*}
		r^{1\slash 2}c_{\bfm}^{-1}\mathbb{U}(r\rho-F_{\bfm})\leq D\sup _{\bfm}\mathbb{U}(r\rho-F_{\bfm})\le D\sup _{\bb u \in \R^2}\mathbb{U}( \bb u + F_{\bb 0})< \8.
	\end{equation*}
	Indeed, given $\bb u\in \R^2$ let $N_{\bb u}=\inf \{ n\in \N : \bb S_n\in u + F_{\bb 0}\}$ where $\bb S_n$ is a random walk with increment law $\P( (\bb S_1-\bb S_0) \in \cdot)=Q$. Then by the Markov property
	\begin{align}
		\mathbb{U}(\bb u+F_{\bb 0})=&\E \Big[ \sum _{n=0}^{\8} \Ind {\bb u+F_{\bb 0}}(\bb S_n) \Big]=\E \Big[ \sum _{n=N_{\bb u}}^{\8} \Ind {\bb u+F_{\bb 0}}((\bb S_n-\bb S_{N_{\bb u}})+\bb S_{N_{\bb u}}) \Big] \notag\\
		\leq &\E \Big[ \sum_{k=0}^\infty \Ind {\{N_u=k\}} \sum _{n=k}^{\8} \Ind {\{\| S_n-S_{k}\| \leq \sqrt{2}\}} \Big]\leq \P(N_u<\8) \cdot \mathbb{U} \big(2 F_{\bb 0}\big). \label{eq:bound:U}
	\end{align}
	Note that $\mathbb{U}(2 F_{\bb 0})$ is finite since the random walk is transient.
	Suppose now that  $r^{1\slash 2}c_{\bfm}^{-1}\geq D$ i.e. $c_{\bfm} \leq r^{1\slash 2}\slash D$ and $r \ge D^2$ (since $c_{\bb m}\ge 1$). Using $F _{\bfm}=\bfm + F_{\bb 0}$ we have $\| \bfm \| \leq c_{\bb m} \leq \frac{1}{D}r^{1\slash 2}$. Now choose $D$ such that $\| \bfm\| / \norm{\bb \rho} \le \frac12 r^{1/2}$ for all $\bb m$ satisfying $r^{1/2} c_{\bb m} \ge D$. By Theorem \ref{thm:Carlsson}, there is a universal constant $C$ such that
	\begin{equation*}
		r^{1\slash 2}\mathbb{U}(r\bb \rho +\bb x +F_{\bb 0}) \le C 
	\end{equation*}
	for all $r \in \R$ and $\bb x$ orthogonal to $\bb \rho$. We may now decompose $\bb m=s{\bb \rho} + \bb m_{\perp}$ with $\bb m_{\perp} \perp \bb \rho$ and use that $s \le \norm{\bb m}/\norm{\rho} \le \frac12 r^{1/2}$.
	\begin{align*} r^{1/2} \mathbb{U} (r\rho - F_{\bb m}) &= r^{1/2} \mathbb{U} \big((r-s)\rho - \bb m_{\perp} - F_{\bb 0} \big)\\
		 &= \bigg(\frac{r}{r-s}\bigg)^{1/2} (r-s)^{1/2}\, \mathbb{U} \big((r-s)\rho - \bb m_{\perp} - F_{\bb 0} \big) \le C'
	\end{align*} 
	where we have used that $r/(r-s)$ is uniformly bounded since $s \le \frac12 r^{1/2}$ and $r \ge D^2$.
	
	\medskip
	
	The bound \eqref{eq:bound:dRi.any} follows in a similar, but simpler way: In \eqref{eq:bound.by.dri} we omit $r^{1/2} c_{\bb m}^{-1}$ and replace $r \rho$ by arbitrary $\bb s$. Then the assertion follows by using the uniform estimate for $\mathbb{U}(\bb u + F_{\bb 0})$, given in \eqref{eq:bound:U}.
\end{proof}
 
\appendix
\section{Limit Theorems for Random Walks in $\R^2$}
For the reader's convenience we quote here the two-dimensional Edgeworth expansion and the renewal theorems used in our proofs.

\subsection{Edgeworth Expansions}

\begin{theorem}[Corollary 20.5 in \cite{Bhattacharya2010}]
	Let $\bb Y_1,...,\bb Y_n$ be i.i.d random vectors with values in $\R ^2$ and law $Q$. Suppose that $\E [\bb Y_1]=0$, $\E [\|\bb Y_1\| ^4]<\8 $ and that the characteristic function $\hat{Q}$ of $Q$ satisfies  
	\begin{equation*}
		\limsup_{\|t \| \to \8}|\hat Q (t)|<1 .
	\end{equation*}	
	Let $\Sigma$ be the covariance matrix of $\bb Y_1$ and 
	\begin{equation}
		h(\bb x)= (2\pi )^{-1}(\det \Sigma)^{-1\slash 2}e^{-\langle \bb x,\Sigma^{-1} \bb x\rangle }.
	\end{equation}
	Let $Q_n$ be the law of $n ^{-1\slash 2}(\bb Y_1+...+\bb Y_n)$.
	Then
	\begin{equation}\label{eq:Edge}
		\sup _{C\in \mathcal{C}}\left | Q_n (C)-H (C)- H_2 (C)\right |=o(n^{-1}), 
	\end{equation}
	where $\mathcal{C}$
	%= \mathcal{C}(const)$ 
	is the family of convex sets $C \subset \R^2$;
	$H$ is the measure with density $h $, $H_2$ is the measure with density 
	$n^{-1\slash 2}P_1h +n^{-1}P_2h  $ and $P_1, P_2$ are polynomials depending only on $Q$.  	%such that for fixed $const$ we have
	%$|\{ x: d(x,\partial C)\leq \eps \}|\leq const \ \eps $
\end{theorem}

\subsection{Two-Dimensional Renewal Theory}

Renewal theory for multi-dimensional random walks dates back to works of \cite{Doney1966,Stam1969}, see also \cite{Carlsson1982,Hoeglund1988,Nagaev1979} and references therein. 
We use the following two-dimensional renewal theorem. Let $\bb U$ be a random vector in $\R^2$ with law $Q$ and corresponding renewal measure $\mathbb{U} = \sum_{n=0}^\infty Q^n$.
Write $\hat{Q}$ for the characteristic function of $Q$. 

%{\red In the theorem below, nonarithmetic means that $\{ \bb t \, : \, \hat{Q}(\bb t)=1\}=\{\bb 0\}$. (see \cite[Definition 1.1]{Stam1969})}

\begin{theorem}[{\cite[Theorem 5.3]{Stam1969}}]\label{thm:Stam1}
	Assume that $Q$ is nonarithmetic and $\E [\norm{\bb U}^2] < \infty$. Denote $\bb m = \E[ \bb U]$. Then, for $c=\frac{\bb m}{\norm{\bb m}}$ and any bounded set $A$ such that $\lambda^2(\partial A)=0$, it holds
	$$ \lim_{t \to \infty} t^{1/2}\cdot\mathbb{U}(t \bb c+A)= (2 \pi)^{-1/2} (\det B)^{-1/2} \abs{m}^{-1/2} \lambda^2(A)$$
	for a strictly positive definite matrix $B$, determined by the law $Q$.
\end{theorem}

%If we consider directions $c$ that are not parallel to $m$, the renewal measure vanishes with a rate:
%
%\begin{theorem}[{\cite[Theorem 5.2]{Stam1969}}]
%	For any bounded set $A$, it holds uniformly in $c$ from a closed subset of $S^1$ not containing $\frac{m}{|m|}$, that
%	$$ \lim_{t \to \infty} t\cdot \mathds{U}(tc +A)=0.$$
%\end{theorem}
%
%
%
%We obtain the following result for the renewal measure $\mathds{U}_\theta$ associated with $U$ under $P_\theta$.
%
%\begin{theorem}\label{thm:renewal2D+measurechange}
%	Suppose that the assumptions of Theorem \ref{thm:Stam1} are satisfied under $\P_\theta$. Then
%	$$ \lim_{t \to \infty} t^{1/2}\mathds{U}_\theta(tc+A)
%	%	 e^{t\skalar{\theta,c}} t^{1/2} \sum_{n=0}^\infty \E \big[ e^{\skalar{\theta, S_n - tc}} \mathds{1}_{\{S_n - tc \in A\}} \big]
%	= \begin{cases}
%		C \lambda^2(A) & c = \frac{m_\theta}{\abs{m_\theta}} \\ 0&  \neg (c \parallel m_\theta)
%	\end{cases}$$ 
%\end{theorem}

We will also require a uniform bound for arbitrarily shifted sets (not necessarily in the direction of the drift). This is given by the following result of \cite{Carlsson1982}.
%
%{\red Carlsson writes nonlattice, but this is the same as nonarithmetic in Stam...}

\begin{theorem}[{\cite[Theorem 1]{Carlsson1982}}]\label{thm:Carlsson}
	Under the assumptions of Theorem \ref{thm:Stam1} it holds that there is a constant $C$ such that for all $t>0$,
	$$  t^{1/2}\mathds{U}(t\bb c+\bb x+A)\leq C$$
	uniformly in $\bb x$ orthogonal to $\bb c$.
\end{theorem}

Note that \cite{Carlsson1982} uses the term nonlattice instead of nonarithmetic as in \cite{Stam1969}, but their definitions coincide.

%\bibliographystyle{plain}
%%\bibliographystyle{apalike}
%\bibliography{sorvE.bib}

\end{document}